\documentclass[review]{elsarticle}

\usepackage{lineno,hyperref}
\modulolinenumbers[5]
\usepackage{bm}
\usepackage{amssymb}

\usepackage{amsmath}
\usepackage{cases}
\usepackage{relsize}
\usepackage{comment}

\journal{}

\usepackage[margin=0.75in,top=0.75in ]{geometry}

\newtheorem{theorem}{Theorem}[section]

\newtheorem{definition}{Definition}[section]

\newtheorem{lemma}{Lemma}[section]

\newtheorem{proposition}{Proposition}[section]
\newtheorem{remark}{Remark}[section]

\newenvironment{proof}[1][Proof]{\textbf{#1.} }{\ \rule{0.5em}{0.5em}}

\newcommand{\R}{\mathbb{R}}

\makeatletter
\def\widebreve#1{\mathop{\vbox{\m@th\ialign{##\crcr\noalign{\kern\p@}%
  \brevefill\crcr\noalign{\kern0.1\p@\nointerlineskip}%
  $\hfil\displaystyle{#1}\hfil$\crcr}}}\limits}

\makeatletter
\newcommand*\bigcdot{\mathpalette\bigcdot@{.6}}
\newcommand*\bigcdot@[2]{\mathbin{\vcenter{\hbox{\scalebox{#2}{$\m@th#1\bullet$}}}}}
\makeatother

\def\brevefill{$\m@th \setbox\z@\hbox{}%
  \scalebox{0.7}{\rotatebox[origin=c]{90}{(}} \kern4pt $}
\makeatletter

\newcommand{\OF}{\Omega^f_0}

\newcommand{\OS}{\Omega^s_0}

\newcommand{\ciajb}{c_{i \alpha j \beta}}

\newcommand{\diajb}{d_{i \alpha j \beta}}

\newcommand{\biajb}{b_{i \alpha j \beta}}

\newcommand{\al}{\alpha}

\newcommand{\be}{\beta}

\newcommand{\xit}{\tilde{\xi}}
\newcommand{\xxi}{\bm{\tilde \xi}}

\newcommand{\vv}{\bm{\tilde v}}

\newcommand{\nv}{\bm \nabla \bm \vv}

\newcommand{\tv}{\bm \nabla \bm {\breve v}}

\newcommand{\ma}{\bm {\mathcal{A}}}

\newcommand{\tc}{\bm{\breve{\ma}}}

\newcommand{\nc}{\bm \nabla \tc}

\newcommand{\cch}{\textup{cof}(\bm \nabla \tc)}

\newcommand{\tciajb}{\breve{c}_{i \alpha j \beta}}

\newcommand{\tdiajb}{\breve{d}_{i \alpha j \beta}}

\newcommand{\tbiajb}{\breve{b}_{i \alpha j \beta}}

\newcommand{\qciajb}{{c}^{q}_{i \alpha j \beta}}

\newcommand{\lciajb}{{c}^{l}_{i \alpha j \beta}}

\newcommand{\Id}{\textbf{\textup{Id}}}

\newcommand{\norm}{\lvert \lvert}

\newcommand{\inc}{(\bm \nabla \tc)^{-1}}

\newcommand{\itnc}{(\bm \nabla \tc)^{-t}}

\newcommand{\cof}{\textup{cof}}

\newcommand{\nga}{\bm \nabla \bm {\tilde \gamma}}

\newcommand{\nze}{\bm \nabla \bm {\tilde \zeta}}

\newcommand{\tdet}{\textup{det}}

\newcommand{\tca}{\bm{\breve{\mathcal{A}}}}

\newcommand{\tp}{\bm {\breve \varphi}}

\newcommand{\txi}{\bm {\breve \xi}}

\begin{document}

\begin{frontmatter}

\title{ANALYSIS OF THE INTERACTION BETWEEN AN
INCOMPRESSIBLE FLUID AND A QUASI-INCOMPRESSIBLE
NON-LINEAR ELASTIC STRUCTURE
}

\author{Fatima ABBAS
}
\address{Laboratory of Mathematics (EDST)  - Lebanese University - Lebanon \\
Laboratoire de Math\'ematiques Appliqu\'ees du Havre (LMAH) - Le Havre University - France}
\ead{fatima-abs@hotmail.com}

\author{Ayman MOURAD}
\address{Department of Mathematics - Faculty of Science (I) - Lebanese University - Lebanon}
\ead{ayman\_imag@yahoo.fr}

%

\begin{abstract}
We are interested in studying an unsteady fluid-structure interaction problem in a three-dimensional space. We consider a homogeneous Newtonian fluid which is modeled by the Navier-Stokes equations. Whereas the motion of the structure is described by the quasi-incompressible non-linear Saint Venant-Kirchhoff model. We establish the local in time existence and uniqueness of solution for this model. For this sake, first we rewrite the non-linearity of the elastodynamic equation in an explicit way. Then, a linearized problem is introduced in the Lagrangian reference configuration and we prove that it admits a unique solution. Based on the \textit{a priori} estimates on the solution of this problem together with the fixed point theorem we prove that the non-linear problem admits a unique local in time solution. At last, by the inf-sup condition we reach to the existence of the fluid pressure.
\end{abstract}


\end{frontmatter}

\linenumbers

\section*{Introduction}

Fluid-structure interaction (FSI)\index{FSI} problem is a wide spread subject which has gain a lot of concern and interest among mathematicians.~This is due to the fact that many real-world problems consider the analysis of FSI problems as an essential tool to avoid failure. For example, they are considered in the design of many engineering systems such as aircrafts, engines and bridges, where the FSI oscillations are studied. Also, in biological field, fluid-structure interaction problems play an important role in the analysis of aneurysms and blood flow in stenosed arteries.~Various kinds of fluid-structure interaction problems have been studied by modeling the fluid by either Stokes or Navier-Stokes equations coupled with an equation modeling the structure. Some deal with incompressible fluids \cite{Boulakia06,CS04,CG2000}, others with compressible fluids \cite{Boulakia2,Boulakia3}. Structures modeled with plate equations or shell equations were treated in \cite{Flori2}. The Stokes equations coupled with beam equation were analyzed in \cite{CG1998}. The case of a free boundary FSI with the flow being incompressible and coupled with a linear Kirchhoff elastic material has been treated in \cite{CS04}, where the existence and uniqueness locally in time of such motion has been proved. In \cite{CG2000} the existence locally in time of a weak solution for an incompressible fluid with a rigid structure has been proved. Similar model has been studied in \cite{Desj} considering a variable density where the global existence of the solution has been proved, that is, the existence of the solution until collisions occur between either the structure and boundaries or between two structures.
For the coupling of an incompressible fluid with elastic structure, the existence of global weak solutions has been proved in \cite{Boulakia06} when adding a regularizing term to the structure motion. In 3D, the work in \cite{CG2002} has proved the existence of steady solutions of the incompressible Navier-Stokes equations when coupled with the non-linear Saint Venant-Kirchhoff model. Whereas, the existence and uniqueness of a regular solution has been proved in the case of compressible Navier-Stokes equations coupled with the non-linear Saint Venant-Kirchhoff model in \cite{Boulakia2}, and with linear elastic model in \cite{Boulakia3}. 
\\ \\
\indent
In our work we consider the interaction between an incompressible homogeneous Newtonian fluid modeled by the Navier-Stokes equations surrounded by a hyperelastic quasi-incompressible structure modeled by the non-linear Saint Venant-Kirchhoff model. 
We couple them in a one domain, by considering a common boundary and imposing some conditions on it. 
First, we introduce the coupled system at time $t$, which consists of the incompressible homogeneous Navier-Stokes equations with the elastodynamic equations modeled by the non-linear Saint Venant-Kirchhoff model. 
From mathematical point of view, Navier-Stokes equations are studied in the Eulerian (spatial) framework, whereas elastic structures are studied in the Lagrangian (material) framework. In order to be able to study the coupled system we use the deformation mappings of both the fluid and the structure domains to rewrite the coupled system in the Lagrangian framework, in particular, in the reference configuration corresponding to the time $t=0$. Indeed, since we are working with a problem involving a free moving boundary, the Lagrangian frame allows us to consider working on a fixed domain.
As for the Saint Venant-Kirchhoff model we rewrite it in an explicit form which enables us to easily deal with it when applying the fixed point theorem as well as to find some bounds on it. 
In a second step we partially linearize our system by considering the deformations to be given for given fluid velocity $\bm {\breve v}$ and structure's displacement $\txi$, such that the couple $(\bm {\breve v},\txi)$ is in some fixed point space. The third step consists of formulating an auxiliary problem, which comes from the classical system by changing slightly the coupling conditions coming from the elastodynamic equations associated to the structure. The weak formulation is derived by considering a transformation of a divergence-free setting, so that the fluid pressure term will disappear. Using Faedo-Galerkin approach we define Galerkin approximations of the solutions and derive a priori estimates for the Galerkin sequence. By passing to the limit, and using compactness results with Aubin-Lions-Simon Theorem we prove the existence and uniqueness of a solution for the auxiliary problem. Based on the results concerning the auxiliary problem, and using the fixed point theorem we prove the existence and uniqueness of the solution of the partially linearized problem. Coming back to the non-linear problem, we use the fixed point theorem approach to prove the existence of a solution for the non-linear fluid-structure interaction problem. Finally, we establish the existence of an $L^2$ fluid pressure by verifying the inf-sup condition.

%
%
%
%
%
%
%
%
%
%
%
%
\section {Fluid-Structure Interaction Problem}\label{FSI}
The fluid is governed by the homogeneous incompressible Navier-Stokes equations.
Let $T>0$ be given. At time $t$, let $\Omega_f(t) \subset \R^3$ denotes a regular (enough) bounded connected domain representing the lumen of the artery. Denote by $\partial \Omega_f(t)= \Gamma_\textup{in}(t) \cup \Gamma_\textup{out}(t) \cup \Gamma_f(t)$ its smooth boundary.
The incompressible Navier-Stokes equations formulated in the Eulerain coordinates are
\begin{subnumcases}{}
    \rho_f \Big( \partial_t \bm v +(\bm v \cdot \nabla) \bm v \Big) -\bm \nabla \cdot \bm \sigma_f(\bm v,p_f) = 0 & \textrm{in $\Omega_f(t) \times (0,T) $}, \label{4-eq1} \\
    \nabla \cdot \bm v = 0 & \textrm{in $ \Omega_f(t) \times (0,T) $}, \label{4-eq2} \\
   \bm  \sigma_f(\bm v,p_f) \ \bm n_f = \bm g_f & \textrm{on $\Gamma_f(t) \times (0,T)$, } \label{4-eq3} \\
   \bm v=\bm v_\textup{in} & \textrm{on $\Gamma_{\textrm{in}}(t) \times (0,T)$, }\label{4-eq4} \\
    \bm \sigma_f(\bm v,p_f) \ \bm n_f = 0 & \textrm{on $\Gamma_{\textrm{out}}(t) \times (0,T)$, }\label{4-eq5} \\
    \bm v=\bm v_0    & \textrm{in $\Omega_f(t)$ at $t=0$, }\label{4-eq6}
\end{subnumcases}
where $\bm v=(v_1,v_2,v_3)^t$ is the fluid velocity, $p_f$ is its pressure and $\rho_f >0 $ is its density.
We denote by $\bm g_f$ an external load on $\Gamma_f(t)$. 
The term $\bm \sigma_f(\bm v,p_f)$ is the Cauchy stress tensor of the fluid whose expression is
\[
\bm \sigma_f(\bm v,p_f) = 2\mu \bm D(\bm v)-p_f \ \Id,
\]
with $\mu$ is its dynamic viscosity and $\bm D(\bm v)=\dfrac{\bm{\nabla v}+ \bm{(\nabla v)}^t}{2}$ is the symmetric gradient. 
On the other hand, the structure is considered to be a quasi-incompressible homogeneous hyperelastic material modeled by the non-linear
Saint Venant-Kirchhoff model \cite{Ciarlet}. We denote by $\Omega_s(t) \subset \R^3$ a regular enough domain that represents the structure at any time $t>0$ and by $\partial \Omega_s(t) $ its smooth boundary such that $\partial \Omega_s(t) = \Gamma_1(t) \cup \Gamma_2(t)$.
The structure displacement $\bm \xi_s$ satisfies the following equations
\begin{equation}\label{EqS1}
\begin{cases}
\rho_s \partial^2_t \bm \xi_s - \bm \nabla \cdot \bm \sigma^s_\textup{Qinc}(\bm \xi_s) = 0 & {\rm{in}} \; \Omega_s(t) \times (0,T), \\
\bm \sigma^s_\textup{Qinc}(\bm \xi_s) \ \bm {\tilde n}_s = \bm g_s & \textup{on} \; \Gamma_1(t) \times (0,T), \\
\bm \xi_s = 0  & \textup{on} \; \Gamma_2(t) \times (0,T), 
\end{cases}
\end{equation}
where $\bm \sigma^s_\textup{Qinc}$ is the Cauchy stress tensor characterizing the quasi-incompressible property of the structure and $\bm g_s$ is a surface external force applied on $\Gamma_1(t)$.
To set up the FSI system, the domains $\Omega_f(t)$ and $\Omega_s(t)$ are coupled by considering $\Gamma_1(t) \equiv \Gamma_f(t)$. Here and after the common boundary will be denoted by $\Gamma_c(t)$. To ensure the global energy balance of the system some coupling conditions representing the continuity of the velocities and stresses must be imposed on the boundary $\Gamma_c(t)$.
These coupling conditions are given as 
\begin{equation}\label{Coup-FS-Cdtn}
\begin{cases}
\bm v = \partial_t \bm \xi_s, & \textup{on} \quad \Gamma_c(t) \times (0,T),\\
\bm \sigma_f(\bm v,p_f) \ \bm n  = \bm \sigma^s_\textup{Qinc}(\bm \xi_s) \bm n  & \textup{on} \quad \Gamma_c(t) \times (0,T),
\end{cases}
\end{equation}
where $\bm n$ is the outward normal from $\Omega_f(t)$ to $\Gamma_c(t)$.\\
Finally, we introduce the initial conditions
\begin{itemize}
\item $\bm v(.,0)=\bm v_0$ \ \ in \ $\Omega_f(0)$,
\item $\bm \xi_s(.,0)=\bm \xi_0$  \ \ in \ $\Omega_s(0)$,
\item $\partial_t \bm \xi_s(.,0)=\bm \xi_1$  \ \ in \ $\Omega_s(0)$,
\item $p_f(.,0)=p_{f_0}$  \ \ in \ $\Omega_f(0)$,
\end{itemize}
which satisfy
\begin{equation}\label{intialcdtns}
\bm v_0 \in H^6(\Omega_f(0)),
\ \bm \xi_0 \in H^4(\Omega_s(0)),
\ \bm \xi_1 \in H^3(\Omega_s(0))
\ \textup{and} \
\ p_{f_0} \in H^3(\Omega_f(0)).
\end{equation}
Let $\Omega(t)= \left[ \overline{\Omega_f(t)} \cup \overline{\Omega_s(t)} \right]^{\circ}$
and $\partial \Omega(t) = \left[ \partial \Omega_f(t) \cup \partial \Omega_s(t) \right] \setminus \left[ \partial \Omega_f(t) \cap \partial \Omega_s(t) \right]$.\\
At time $t>0$, the coupled system is given by
\begin{subnumcases}{}
    \rho_f \Big( \partial_t \bm v +(\bm v \cdot \nabla) \bm v \Big) -\bm \nabla \cdot \bm \sigma_f(\bm v,p_f) = 0 & \textrm{in $\Omega_f(t) \times (0,T)$,} \label{eqcoup1} \\
    \nabla \cdot \bm v = 0 & \textrm{in $ \Omega_f(t) \times (0,T), $} \label{eqcoup2} \\
    \bm v=\bm v_\textup{in} & \textrm{on $\Gamma_{\textrm{in}}(t) \times (0,T)$,}\label{eqcoup4}\\
    \bm \sigma_f(\bm v,p_f) \ \bm n = 0 & \textrm{on $\Gamma_{\textrm{out}}(t) \times (0,T)$, }		\label{eqcoup5} \\
   	\rho_s \partial^2_t \bm \xi_s - \bm \nabla \cdot \bm \sigma^s_\textup{Qinc}(\bm \xi_s)=0 & \textrm{in $ \Omega_s(t)            	\times (0,T)$},\label{eqcoup6} \\
	\bm \xi_s = 0  & \textrm{on $ \Gamma_2(t) \times (0,T) $},\label{eqcoup8}\\
	\bm v =  \partial_t \bm \xi_s & \textrm{on $ \Gamma_c(t) \times (0,T) $},				\label{eqcoup9} \\
	\bm \sigma_f(\bm v ,p_f) \ \bm n =\bm \sigma^s_\textup{Qinc}(\bm \xi_s) \ \bm n  & \textrm{on $ \Gamma_c(t) \times (0,T)$},\label{eqcoup100}\\
	\bm v(.,0)=\bm v_0 \ \textrm{and} \ p_f(.,0)=p_{f_0} & \textrm{in $\Omega_f(0)$, }\label{eqcoup101}\\
   	\bm \xi_s(.,0)=\bm \xi_0 \ \textrm{and} \ \partial_t \bm \xi_s(.,0)=\bm \xi_1 & \textrm{in $\Omega_s(0)$. }\label{eqcoup102}
\end{subnumcases}
The Navier-Stokes equations are defined on the domain $\Omega_f(t)$ which evolves over time from the initial configuration $\Omega_f(0)$ according to a position function
\begin{align*}
%
\bm{\mathcal{A}}(.,t):\Omega_f(0) &\longrightarrow \Omega_f(t) \nonumber \\
\bm {\tilde x} & \longrightarrow  \bm{\mathcal{A}}(\bm {\tilde x},t)=\bm x
\end{align*}
that associates to the Lagrangian coordinate of a fluid particle its Eulerian coordinate.
For all $\bm {\tilde x} \in \Omega_f(0)$ the function $\bm {\mathcal{A}}(\bm {\tilde x},.)$ satisfies
\begin{equation*}
\begin{cases}
\partial_t \bm{\mathcal{A}}(\bm {\tilde x},t)=\bm v(\bm{\mathcal{A}}(\bm {\tilde x},t),t)  \qquad \text{for} \ t \in(0,T),
 \\
\bm{\mathcal{A}}(\bm {\tilde x},0)=\bm {\tilde x}.
\end{cases}
\end{equation*}
The function $\bm \ma$ is called the Arbitrary Lagrangian-Eulerian (ALE) map.

Similarly, the elastodynamic equations in the displacement $\bm \xi_s$ are defined on the domain $\Omega_s(t)$ which evolves over time from the initial configuration $\Omega_s(0)$ according to a position function
\begin{align*}
%
\bm \varphi_s(.,t):\Omega_s(0) &\longrightarrow \Omega_s(t) \nonumber \\
\bm {\tilde y} &\longrightarrow  \bm \varphi_s(\bm {\tilde y},t)=\bm y
\end{align*}
and we have
\begin{align}
\label{Defo-structure}
 \bm \varphi_s(\bm {\tilde y},t)=\bm {\tilde y}+\bm \xi_s(\bm \varphi_s(\bm {\tilde y},t),t).
\end{align}
Notice that, using \eqref{Defo-structure} we have
\begin{align*}
 \bm \varphi_s(\bm {\tilde y},0)=\bm {\tilde y}+\bm \xi_s(\bm {\tilde y},0), \ \textrm{that is} \ \bm {\tilde y}=\bm {\tilde y} + \bm \xi_0 
\end{align*}
which yields $\bm \xi_0=\bm 0$.
\\
In the  sequel, we omit the subscript $s$ of the structure displacement and deformation, that is, we write $\bm \xi_s \equiv \bm \xi$ and 
$\bm \varphi_s \equiv \bm \varphi$. Further, we refer to the space elements in $\OF$ and $\OS$ by $\bm {\tilde x}$.
%
%
%
%

The definition of these two mappings enables us to write System \eqref{eqcoup1}-\eqref{eqcoup102} on the domain $\Omega(0)$.
To do so, we consider the following change of variables in terms of the deformation mappings $\bm {\mathcal{A}}$ and $\bm \varphi$.
For all $\bm {\tilde x}$ in $\Omega_f(0)$ and $\Omega_s(0)$ and $t$ in $(0,T)$ set
\begin{equation}\label{change-ofv}
\bm \vv(\bm {\tilde x},t)=\bm v(\bm{\mathcal{A}}(\bm {\tilde x},t),t),
\ \bm \xxi(\bm {\tilde x},t)=\bm \xi(\bm \varphi(\bm {\tilde x},t),t) \
 \textrm{and} 
\ \tilde{p}_f(\bm {\tilde x},t)=p_f(\bm{\mathcal{A}}(\bm {\tilde x},t),t). \
\end{equation}
On the reference domain $\Omega_f(0)$, the fluid stress tensor is given by \cite[Section 2.1.7]{Richter} as
\begin{equation}\label{fluidstress-intial0}
\begin{split}
\bm {\tilde \sigma}^0_f(\bm \vv,\tilde{p}_f)
&=
%
%
\Bigl(
\mu \bigl(\bm \nv (\bm{\nabla \mathcal{A}})^{-1}
+
(\bm{\nabla \mathcal{A}})^{-t} (\bm \nv)^t \bigr) -
\tilde{p}_f \Id \Bigr) \textrm{cof}(\bm{\nabla \mathcal{A}}) 
 \\
& =
\bm {\tilde \sigma}^0_f(\bm \vv)-\tilde{p}_f 
\textrm{cof}(\bm{\nabla \mathcal{A}}).
\end{split}
\end{equation}
As for the quasi-incompressible structure, the Cauchy stress tensor is given in terms of the first Piola-Kirchhoff stress tensor $\bm{P}$ as \cite[Lemma 2.12]{Richter}
\begin{equation}\label{Inc-First-pio}
\begin{split}
\bm{P}_\textup{Qinc}
&=
\textrm{det}(\bm{\nabla \varphi}) \big( \bm \sigma^s_\textup{Qinc}(\bm \xi) \circ \bm \varphi \big) (\bm{\nabla \varphi})^{-t}
\\
&=
\bm{P}
+
\mathsf{C}(\tdet(\bm{\nabla \varphi})-1)\cof(\bm{\nabla \varphi})
\end{split}
\end{equation}
with $\mathsf{C}>0$ a sufficiently large constant and
\begin{equation}\label{First Piola}
\bm{P}= \bm{\nabla \varphi} \bm{S}(\bm{\nabla \varphi})
\end{equation}
where
\begin{equation*}
\bm{S}(\bm{\nabla \varphi})=2 \mu_s \bm{E}(\bm{\nabla \varphi})+\lambda_s \textrm{tr}(\textbf{E}(\bm{\nabla \varphi}))\Id
\end{equation*}
is the second Piola-Kirchhoff stress tensor and
\begin{align*}
\bm{E}(\bm{\nabla \varphi}) = \dfrac{1}{2}(({\bm{\nabla \varphi}})^t \bm{\bm{\nabla \varphi}} -\Id)
\end{align*}
is the Green-Lagrange strain tensor with ($\mu_s,\lambda_s$) $\in \R^*_+ \times \R_+$ are the Lam\'e coefficients.\\
In particular, when considering the Saint Venant-Kirchhoff stress tensor, Expression \eqref{First Piola} can be rewritten in terms of the displacement $\bm \xxi$ as 
\begin{equation}
\bm{P}= (\Id+ \bm{\bm{\bm{\nabla \xxi}}}) 
\bigg(\mu_s 
\Big(
\bm{\nabla \xxi}+(\bm \nabla \bm \xxi)^t 
+(\bm \nabla \bm \xxi)^t \bm \nabla \bm \xxi
\Big)
+
\dfrac{\lambda_s}{2} \Big (2 \nabla \cdot \bm \xxi +| \bm \nabla \bm \xxi|^2 \Big )\Id 
\bigg).
\end{equation}
Using relations \eqref{change-ofv}-\eqref{Inc-First-pio} we reformulate the Navier-Stokes equations and the elastodynamic equations in the Lagrangian coordinates. Hence, we can rewrite the coupled System \eqref{eqcoup1}-\eqref{eqcoup102} on $\Omega_f(0)$ and $\Omega_s(0)$ as
\begin{equation}\label{system in ref conf}
\begin{cases}

    \rho_f \tdet(\bm{\nabla \mathcal{A}})\partial_t \bm \vv -\bm \nabla \cdot \bm {\tilde \sigma}^0_f(\bm \vv,\tilde{p}_f) = 0 & \textrm{in $\Omega_f(0) \times (0,T) $}, 						

    \\

    \nabla \cdot \big(\tdet(\bm{\nabla \ma})
    (\bm{\nabla \mathcal{A}})^{-1} \bm \vv \big) = 0 & \textrm{in $ \Omega_f(0) \times (0,T), $} 						
    \\

    \bm \vv=\bm v_\textup{in} \circ \bm \ma & \textrm{on $\Gamma_{\textrm{in}}(0) \times (0,T)$, }					
\\    
    \bm {\tilde \sigma}^0_f(\bm \vv,\tilde{p}_f) \ \bm {\tilde n} = 0 & \textrm{on $\Gamma_{\textrm{out}}(0) \times (0,T)$, }		
\\
   	\rho_s \tdet(\bm{\nabla \varphi}) \partial^2_t \bm \xxi - \bm \nabla \cdot \bm{P}
   	-
\bm    \nabla \cdot 
   \big[
   \mathsf{C}
   (\tdet(\bm{\nabla \varphi})-1) \cof(\bm{\nabla \varphi})
   \big]
   	=0 & \textrm{in $\Omega_s(0)\times (0,T)$},
\\
	\bm \xxi = 0  & \textrm{on $ \Gamma_2(0) \times (0,T),$}
\\
	\bm \vv = \partial_t \bm \xxi & \textrm{on $ \Gamma_c(0) \times (0,T) $},
\\	
	\bm {\tilde \sigma}^0_f(\bm \vv,\tilde{p}_f) \bm {\tilde n} 
	= 
	\left[ \bm{P}+ 
	\mathsf{C}(\tdet(\bm{\nabla \varphi})-1) \cof(\bm{\nabla \varphi}) \right] \bm  {\tilde n} & \textrm{on $ \Gamma_c(0) \times 			 (0,T)$},
\\	
	\bm \vv(.,0)=\bm v_0 
	\quad 	\textrm{and} \quad 
	\tilde{p}_f(.,0)=p_{f_0} & 
	\textrm{in $\Omega_f(0)$},
\\	
   	\bm \xxi(.,0)=\bm \xi_0=0 
   	\quad \textrm{and} \quad 
   	\partial_t \bm \xxi(.,0)=\bm \xi_1 & \textrm{in $\Omega_s(0)$},
\end{cases}
\end{equation}
where $\bm{\nabla \varphi}= \Id + \bm{\bm{\bm{\nabla \xxi}}}$ is the gradient of the deformation and $\bm {\tilde n}$ is the outward normal of $\Omega_f(0)$ on $\Gamma_c(0)$.
\\
\indent In order to deal with the structure model, we write the elasticity model in the spirit of \cite{Gawi}, that is, we define
\begin{align}\label{ciajb}
\ciajb=
\dfrac{\partial \bm{P}_{i \al}}
{\partial(\partial_\be \bm \xit_j)}.
\end{align}
Let us set
\begin{equation}\label{SVKM}
\ciajb(\bm{\bm{\bm{\nabla \xxi}}}) = \mu_s( \delta_{\beta i} \delta_{\alpha j}
				   +\delta_{\al \be} \delta_{ij} )
				   + \lambda_s(\delta_{i \al} \delta_{j \be})
				   + \ciajb^l (\bm{\bm{\bm{\nabla \xxi}}})  + \ciajb^q (\bm{\bm{\bm{\nabla \xxi}}}),
\end{equation}
where $ \ciajb^l (\bm{\bm{\bm{\nabla \xxi}}})$ is the linear part given by
\begin{align}\label{linearpart}
\ciajb^l (\bm{\bm{\nabla \xxi}}) &=
\mu_s \bigl(\delta_{ij} \partial_\be \xit_\al +
\delta_{\al j} \partial_\be \xit_i +
\delta_{ij} \partial_\al \xit_\be +
\delta_{\al \be} \partial_j \xit_i +
\delta_{i \be} \partial_\al \xit_j +
\delta_{\al \be} \partial_i \xit_j
\bigr) 
\nonumber\\
&
+ \lambda_s \bigl(\delta_{i \al} \partial_\be \xit_j +
\delta_{\al \be} \delta_{ij}(\nabla \cdot \bm \xxi) +
\delta_{j \be} \partial_\al \xit_i \bigr)
\end{align}
and $\ciajb^q$ is the quadratic part written as
\begin{equation}\label{quadraticpart}
\begin{split}
\ciajb^q(\bm{\bm{\bm{\nabla \xxi}}})=&
\mu_s \biggl(
\delta_{ij}(\partial_\be \xi \cdot \partial_\al \xit) +
\partial_\be \xi_i \partial_\al \xit_j +
\delta_{\al \be}(\bm{\bm{\nabla \xxi}}_j \cdot \bm{\bm{\nabla \xxi}}_i)
\biggr)
+
\lambda_s \biggl( \frac{1}{2} \delta_{ij} \delta_{\al \be} \lvert \bm{\bm{\bm{\nabla \xxi}}} \lvert^2
+ \partial_\al \xit_i \partial_\be \xit_j
\biggr).
\end{split}
\end{equation}
Hence, $\ciajb$ can be rewritten as
\begin{align}
\ciajb (\bm{\bm{\bm{\nabla \xxi}}}) = Cst +L(\bm{\bm{\bm{\nabla \xxi}}})+ Q(\bm{\nabla \xxi}),
\end{align}
where $Cst$ is a constant, $L$ is a linear function in $\bm{\bm{\bm{\nabla \xxi}}}$ and $Q$ is a quadratic function in $\bm{\bm{\bm{\nabla \xxi}}}$.\\
Remark that the coefficients $\ciajb$ are symmetric, that is,
\begin{equation}\label{symcoeff}
\ciajb=c_{j \be i \al} \qquad \forall \ i,\al,j,\be \in \{1,2,3\}.
\end{equation}
%
%
%
%
%
%
%
%
%
\begin{lemma}\label{Piola-P}
For $k =i, \al, j, \be \in \{1,2,3\}$, we denote by $\partial_k$ the partial derivative in space and by $\partial_t$ and $\partial_s$ the partial derivatives with respect to time.
Some consequences of the relation \eqref{symcoeff} are the following
\begin{itemize}
\item[1-] For $i,\al \in \{ 1,2,3 \}$, the partial derivatives of $\bm{P}$ with respect to time and space are respectively
\begin{equation*}
\partial_s \bm{P}_{i \al} = \sum_{j,\be =1}^{3} \ciajb(\bm{\bm{\nabla \xxi}}) \partial_{s \be}^2 \xit_j
\quad
\textrm{and}
\quad
\partial_k \bm{P}_{i \al} = \sum_{j,\be =1}^{3} \ciajb(\bm{\bm{\nabla \xxi}}) \partial_{k \be}^2  \xit_j
.
\end{equation*}
\item[2-] The $i-th$ component of the divergence of $\bm{P}$ is given by
\begin{equation}\label{Div-coeff}
(\bm \nabla \cdot \bm{P})_i=
\sum_{\al,j,\be=1}^{3} \ciajb(\bm{\bm{\nabla \xxi}}) \partial_{\al \be}^2 \xit_j
\quad \forall \ i=1,2,3.
\end{equation}
\item[3-] Assuming that $\bm P(\bm \xxi(.,0))=0$ on $\Gamma_1(0)$, the normal component of the stress tensor $\bm{P}$ on the boundary $\Gamma_1(0)$ is 
\begin{equation}\label{Normal-coeff}
\sum_{\al=1}^{3}\bm{P}_{i \al} \tilde{n}_{\al}=
\sum_{\al,j,\be=1}^{3}
\bigg( \int_0^t \ciajb(\bm{\bm{\nabla \xxi}}) \partial_{s \be}^2 \xit_j \ ds \biggr)\tilde{n}_{\al}
\quad \forall \ i=1,2,3.
\end{equation}

\item[4-] The $i\al$-th component of $\bm{P}$ is given by
\begin{equation*}
\bm{P}_{i\al} = \sum_{\al,j,\be=1}^3 \bigg(\int_0^t \ciajb(\bm{\bm{\nabla \xxi}}) \partial_{s \be}^2 \xit_j \ ds\biggr) \quad \forall \ i=1,2,3.
\end{equation*}
\end{itemize}
\end{lemma}
%
%
%
%
%
\begin{proof}
\begin{itemize}
\item[1-] Let $r$ be the index that represents either the time derivative or the space derivative. For the $i\al$-th component of $\bm P(\bm \xxi)$ we have
\begin{align*}
\partial_r(\bm{P}(\bm \xxi))_{i \al}
=
\sum_{j,\be=1}^3 
\dfrac{\partial(\bm{P}(\bm \xxi))_{i\al}}{\partial(\partial_\be \xit_j)}
\dfrac{\partial(\partial_\be \xit_j)}{\partial_r}
=
\sum_{j,\be=1}^3 
\ciajb(\bm{\bm{\nabla \xxi}}) \partial^2_{r \be} \xit_j.
\end{align*}
\item[2-] Considering  $r=\al$ in the first part yields
\begin{align*}
\partial_\al(\bm{P}(\bm \xxi))_{i \al}
=
\sum_{j,\be=1}^3 
\ciajb(\bm{\bm{\nabla \xxi}}) \partial^2_{\al \be} \xit_j.
\end{align*}
But for $i=1,2,3$ we have
\begin{align*}
(\bm \nabla \cdot \bm{P}(\bm \xxi))_i=
\sum_{\al=1}^3 
\partial_\al(\bm{P}(\bm \xxi))_{i \al}
=
\sum_{\al,j,\be=1}^{3} \ciajb(\bm{\bm{\nabla \xxi}}) \partial_{\al \be}^2 \xit_j.
\end{align*}
\item[3-]
For any $\bm \xxi$ in $\Omega_s(0)$ we have
\begin{align*}
\bm{P}_{i \al}(\bm \xxi(.,t))-\bm{P}_{i \al}(\bm \xxi(.,0)) 
= 
\int_0^t  \partial_s \bm{P}_{i \al}(\bm \xxi(.,s)) \ ds \quad \forall \ i,\al=1,2,3.
\end{align*}
Substituting $\partial_s \bm{P}(\bm \xxi(.,s))$ by its expression from the first part gives
\begin{align*}
\bm{P}_{i \al}(\bm \xxi(.,t))-\bm{P}_{i \al}(\bm \xxi(.,0)) 
= 
\sum_{j,\be=1}^3
\int_0^t  \ciajb(\bm \nabla \xxi) \partial^2_{s \be} \xit_j \ ds \quad \forall \ i,\al=1,2,3.
\end{align*}
In particular, on $\Gamma_1(0)$ we have $\bm{P}(\bm \xxi(.,0))=0$. Consequently, taking  the summation over $\al$
yields
\end{itemize}
\[
\sum_{\al =1}^3
\bm{P}_{i \al}(\bm \xxi(.,t))\tilde{n}_{\al} 
= 
\sum_{\al,j,\be=1}^3
\int_0^t  
\big(
\ciajb (\bm \nabla \xxi) \partial^2_{s \be} \xit_j \ ds
\big)
\tilde{n}_{\al}
\quad \forall \ 
i=1,2,3.
\]
\hspace{11.8cm}
\end{proof}

To deal with the quasi-incompressibility condition, we express it in a way similar to that of the first Piola-Kirchhoff stress tensor \eqref{SVKM}. To do so, we use the notation introduced in \cite[p.~5]{Ciarlet} by defining the third-order orientation tensor $\bm{(\varepsilon_{ijk})}$ whose components are the Levi-Civita symbol $\{\varepsilon_{ijk}\}_{ijk}$.
Using the Einstein summation convention on the indices, we define the $ij$-th element of the matrix $\cof(\bm{\bm{\nabla \varphi}})$ by
\begin{align*}
(\cof(\bm{\bm{\nabla \varphi}}))_{ij}
=
\dfrac{1}{2} \varepsilon_{mni} \varepsilon_{pqj}
\partial_p \varphi_m \partial_q \varphi_n.
\end{align*}
Further, the determinant of the 3-by3-matrix $\bm{\bm{\nabla \varphi}}$ is
\begin{align}\label{det}
\tdet(\bm{\bm{\nabla \varphi}})
=
\dfrac{1}{6}
\varepsilon_{ijk}
\varepsilon_{pqr}
\partial_p \varphi_i
\partial_q \varphi_j
\partial_r \varphi_k.
\end{align} 
We define
\begin{align}\label{diajb}
d_{i \al j \be}(\bm{\bm{\nabla \xxi}}) 
=\dfrac{\partial}
{\partial(\partial_\be \xit_j)} 
\Big[
\big(
\tdet(\bm{\bm{\nabla \varphi}})-1
\big)
 \ \cof(\bm{\bm{\nabla \varphi}})
 \Big]_{i \al}.
\end{align}
Clearly, $d_{i \al j \be}(\bm{\bm{\nabla \xxi}}) $ is a polynomial in $\bm{\bm{\nabla \xxi}}$ of degree at most 4. Moreover, for $i=\al$ and $j=\be$ we get the constant terms of this polynomial.
Then we can write
\begin{align}\label{diajb-def}
d_{i \al j \be}(\bm{\bm{\nabla \xxi}})
=
Cst+d^L_{i \al j \be}(\bm{\bm{\nabla \xxi}})
+
d^Q_{i \al j \be}(\bm{\bm{\nabla \xxi}})
+
d^T_{i \al j \be}(\bm{\bm{\nabla \xxi}})
+
d^F_{i \al j \be}(\bm{\bm{\nabla \xxi}})
\end{align}
where $d^L_{i \al j \be}, d^Q_{i \al j \be}, d^T_{i \al j \be}$ and $d^F_{i \al j \be}$ stand for polynomials in $\bm {\nabla \xxi}$ with respective degree $1, 2, 3$ and 4.
This writing enables us to give the $i-th$ component of $\bm \nabla \cdot [\mathsf{C}(\tdet(\bm{\bm{\nabla \varphi}})-1) \cof(\bm{\bm{\nabla \varphi}})]$. 
In fact,
\begin{align}\label{Div-diajb}
\bigg[
\bm \nabla \cdot 
\Big(
\mathsf{C}(\tdet(\bm{\nabla \varphi})-1) \cof(\bm{\nabla \varphi})
\Big)
\bigg]_i=
\mathsf{C}
\sum_{\al,j,\be=1}^3 d_{i \al j \be}(\bm{\bm{\nabla \xxi}}) \partial^2_{\al \be} \xit_j  \quad \textrm{for}  \quad i=1,2,3.
\end{align}

In a way similar to \eqref{Normal-coeff}, for $i=1,2,3$ the normal component of the quasi-incompressible condition on the boundary $\Gamma_1(0)$ is
\begin{equation}\label{Normal-diajb}
\sum_{\al=1}^{3}
\big[
(\tdet(\bm{\bm{\nabla \varphi}})-1) \cof(\bm{\bm{\nabla \varphi}})
\big]_{i \al} \tilde{n}_{\al}=
\sum_{\al,j,\be=1}^{3}
\bigg( \int_0^t \diajb(\bm{\bm{\nabla \xxi}}) \partial_{s \be}^2 \xit_j ds \biggr)\tilde{n}_{\al}
,
\end{equation}
provided that $\big(
\tdet(\bm{\bm{\nabla \varphi}})-1) \cof(\bm{\bm{\nabla \varphi}})
\big)(.,0)=0$ on $\Gamma_1(0)$.
\\
In what follows, for simplicity we set
\begin{align}\label{biajb}
\biajb= \ciajb+ \mathsf{C} \diajb.
\end{align}
%
%
%
%
Using Relations \eqref{Div-coeff}, \eqref{Normal-coeff}, \eqref{Div-diajb} and \eqref{Normal-diajb}, System \eqref{system in ref conf} can be rewritten as
%
%
%
%
%
%
\begin{equation}\label{System at t=0}
\begin{cases}
    \rho_f \tdet(\bm{\nabla \mathcal{A}})\partial_t \bm \vv - \bm\nabla \cdot \bm {\tilde \sigma}^0_f(\bm \vv,\tilde{p}_f) = 0 & \textrm{in $\Omega_f(0) \times (0,T) $},
\\    
    \nabla \cdot (\tdet(\bm{\nabla \ma})(\bm{\nabla \ma})^{-1} \bm \vv) = 0 & \textrm{in $ \Omega_f(0) \times (0,T) $}, 
\\
    \bm \vv=\bm v_{\rm in} \circ \bm {\mathcal{A}} & \textrm{on $\Gamma_{\textrm{in}}(0) \times (0,T)$, }
\\    
    \bm {\tilde \sigma}^0_f(\bm \vv,\tilde{p}_f) \ \bm {\tilde n} = 0 & \textrm{on $\Gamma_{\textrm{out}}(0) \times (0,T)$, }		
\\    
   	\rho_s \tdet(\bm{\bm{\nabla \xxi}}+ \Id)\partial^2_t \xit_i 
   	-     
   	\displaystyle \sum_{\al,j,\be=1}^{3}
   	\biajb(\bm{\bm{\nabla \xxi}})
   	\partial_{\al \be}^2 \xit_j=0, & 							\textrm{in $\Omega_s(0) \times (0,T)$},
\\	
	\bm \xxi = \bm 0  & \textrm{on $ \Gamma_2(0) \times (0,T) $},
\\	
	\bm \vv = \partial_t \bm \xxi & \textrm{on $ \Gamma_c(0) \times (0,T) $},				
\\	
\left[
\bm {\tilde \sigma}^0_f(\bm \vv,\tilde{p}_f) \bm {\tilde n}
\right]_i 
= 
\displaystyle \sum_{\al,j,\be=1}^3 						
\Bigg(\mathlarger \int_0^t \biajb(\bm{\bm{\nabla \xxi}}) 
\partial^2_{s \be} \xit_j ds \Bigg)\tilde{n}_{\al},
 & \textrm{on $ 					\Gamma_c(0) \times (0,T)$},
\\	
	\bm \vv(.,0)=\bm v_0 \quad \textrm{and} \quad \tilde{p}_f(.,0)=p_{f_0} & \textrm{in $\Omega_f(0)$, }
\\	
   	\bm \xxi(.,0)=\bm 0 \quad \textrm{and} \quad \partial_t \bm \xxi(.,0)=\bm \xi_1  & \textrm{in $\Omega_s(0)$.}
\end{cases}
\end{equation}
for $i=1,2,3$.\\
Notice that, unlike System \eqref{system in ref conf}, in this system the boundary condition related to the elastodynamic equation is incompatible with it. Indeed, for Equations \eqref{System at t=0}$_5$ and \eqref{System at t=0}$_8$ to combine we must have
\begin{equation}\label{Elast-natural-bdrycdtn}
[
\bm {\tilde \sigma}^0_f(\bm v,\tilde{p}_f) \bm{\tilde{n}}
]_i =
\sum_{\al,j,\be=1}^3 
\Big( 
\biajb(\bm{\bm{\nabla \xi}}) \partial_{\be} \xit_j
\Big)
\tilde{n}_{\al}
,
 \quad i=1,2,3, \ \textrm{on} \quad \Gamma_c(0) \times (0,T).
\end{equation}
This rewriting \eqref{System at t=0}$_5$ of the elasticity equation is efficient when performing the fixed point theorem on the system.
In fact, it helps to get over the difficulties emerging from the non-linearity of the Saint Venant-Kirchhoff model and the hyperbolic type of the equation.
Due to this disagreement issue between Equations \eqref{System at t=0}$_5$ and \eqref{System at t=0}$_8$, the first step of the work is to consider an \textit{auxiliary problem} including the natural boundary condition \eqref{Elast-natural-bdrycdtn}.\\
\\
By considering the boundary and initial conditions we assume that the following compatibility conditions hold on the initial values
\begin{equation}\label{comp-cdtns}
\begin{cases}{}
\bm v_0 = \bm \xi_1 & \textrm{on \quad $ \Gamma_c(0)$},\\
\bm \sigma_f(\bm v_0,p_{f_0})\bm n = \bm 0 & \textrm{on \quad $ \Gamma_c(0)$},\\
p_{f_0} = 2 \mu \bm D( \bm v_0) & \textrm{in \quad $ \Omega_f(0)$},\\
\nabla p_{f_0}
=
\mu \Delta \bm v_0 &
\textrm{in \quad $ \Omega_f(0)$},\\
\bm \nabla \cdot \bm \sigma_f(\bm v_0,p_{f_0})
=
0
 & \textrm{on \quad $ \Gamma_c(0)$}, \\
\partial_t p_f|_{t=0} \bm n
=
S_1 \bm n + E_1 \bm n
& \textrm{on \quad $ \Gamma_c(0)$},\\
\bm \nabla \cdot
\rho_s
\Big[
S_1
+
\partial_t p_f|_{t=0} \Id
\Big]
=
\rho_f 
\bm \nabla \cdot E_1
& \textrm{on \quad $ \Gamma_c(0)$},\\
\rho_f
\Big(
2 (\nabla \cdot \bm v_0) \bm \nabla \cdot E_1
+
\bm \nabla \cdot E_2
\Big)
=
\bm \nabla \cdot
S_2 & \textrm{on \quad $ \Gamma_c(0)$},\\
\Big(E_2- 2 
\big(
(\nabla \cdot \bm v_0) S_3 + S_4 \big) \Big) \bm n
=
\rho_s S_2 \bm n
& \textrm{on \quad $ \Gamma_c(0)$}.
\end{cases}
\end{equation}
where
\begin{itemize}
\item $S_1= 
-\mu \Big( (\bm D(\bm v_0))^2 - 2 (\bm \nabla \bm v_0)^t \bm \nabla \bm v_0 \Big)
+
\bm \sigma_f(\bm v_0,p_{f_0}) 
S_3$.
\item $E_1= 2 \mu_s \bm \epsilon(\bm \xi_1)
+
\lambda_s(\nabla \cdot \bm \xi_1) \Id
+
\nabla \cdot \bm v_0 \Id$.
\item $
S_2
=
\partial_t^2 p_f|_{t=0} \Id
+
2 \partial_t p_f|_{t=0} 
S_3
+
p_{f_0}
S_4
+
2 \mu \bm \epsilon(\bm \nabla \cdot E_1)
-
2
\Big( (\bm D(\bm v_0))^2 - 2 (\bm \nabla \bm v_0)^t \bm \nabla \bm v_0 \Big)
S_3
+
2
\bm D(\bm v_0) S_4.
$
\item $E_2
=
2 \bm \nabla \bm \xi_1 \
E_1
+
2 \mu_s 
\big((\bm \nabla \bm \xi_1)^t
\bm \nabla \bm \xi_1
+
\lambda_s \bm \nabla \bm \xi_1
+
2 
\big(
(\nabla \cdot \bm v_0) S_3 + S_4 \big)$.
\item $S_3=(\nabla \cdot \bm v_0) \Id- (\bm \nabla \bm v_0)^t $.
\item 
$S_4
=
\dfrac{1}{\rho_f} \bm \nabla \cdot 
\big( \bm \nabla \cdot \bm \sigma_f(\bm v_0,p_{f_0}) \big) \Id
-
\dfrac{1}{\rho_f}
\bm \nabla 
\big( \bm \nabla \cdot \bm \sigma_f(\bm v_0,p_{f_0}) \big)
+ 2 \cof(\bm \nabla \bm v_0).$
\end{itemize}
These conditions are obtained from \eqref{system in ref conf} by considering $t=0$, differentiating in time once and twice \eqref{system in ref conf}$_1$, \eqref{system in ref conf}$_5$, \eqref{system in ref conf}$_7$ and \eqref{system in ref conf}$_8$ then considering $t=0$ and taking into consideration the following identities
\begin{align*}
\partial_t
\big(
(\bm \nabla \bm \ma)^{-1}\big) (.,0) 
= - \bm \nabla \bm v_0 \quad \textup{and} \quad
\partial_t 
\big(
\tdet(\bm \nabla \bm \ma)\big)(.,0)
=
\nabla \cdot \bm v_0
\qquad \textup{in} \quad \Omega_f(0).
\end{align*}
\begin{definition}
Let us define the following spaces


%
%

\[
S^T_m=L^\infty\big(0,T;H^m(\Omega_s(0))\big) \cap W^{m,\infty}\big(0,T;L^2(\Omega_s(0))\big) \qquad 0 \leq m \leq 4,
\]
\[
F^T_1=L^{\infty}\big(0,T;L^2(\Omega_f(0))\big) \cap L^2 \big (0,T;H^1(\Omega_f(0))\big),
\]
\[
F^T_2=L^{\infty}\big(0,T;H^2(\Omega_f(0))\big)
\cap 
H^1 \big(0,T;H^1(\Omega_f(0))\big)
\cap 
W^{1,\infty}\big(0,T;L^2(\Omega_f(0))\big),
\]
\begin{align*}
F^T_4
=
L^{\infty} \big(0,T;H^4(\Omega_f(0)) \big)
&\cap 
H^3   \big(0,T;H^1(\Omega_f(0)) \big)
\cap 
W^{2,\infty} \big(0,T;H^2(\Omega_f(0)) \big)
\cap 
W^{3,\infty} \big(0,T;L^2(\Omega_f(0)) \big),
\end{align*}
%
\begin{align*}
\mathcal{P}_3^T
=
L^{\infty} \big(0,T;H^3(\Omega_f(0)) \big)
&\cap 
H^3   \big(0,T;L^2(\Omega_f(0)) \big)
\cap 
W^{1,\infty} \big(0,T;H^2(\Omega_f(0)) \big)
\cap 
W^{2,\infty} \big(0,T;H^1(\Omega_f(0)) \big),
\end{align*}
\[
H_l^1\big(0,T;L^2(\Gamma_c(0))\big):= \{ \psi \in H^1 \big(0,T;(\Gamma_c(0))\big); \psi(0)=0 \}.
\]
Then, for $M>1$ and $T>0$ we define the following fixed point space
\begin{align*}
A^T_M
=\Big\{
(\bm {\breve v},\txi) \in F^T_4 \times S^T_4,\
\bm \txi(.,0)= \bm 0,\ \partial_t \bm \txi(.,0)=\bm \xi_1 \ \textup{in} \ \Omega_s(0) \
\textup{and} 
\ \norm \bm {\breve v}  \norm_{F^T_4} \leq M,
\ \norm \bm \txi \norm_{S^T_4} \leq M
\Big\}
:= A^T_{M_1} \times A^T_{M_2}
\end{align*}%
\end{definition}
After introducing the spaces needed, we are ready to state the main result of the work.
\begin{theorem}[Main Theorem]
Let \textup{(}$\bm v_0, \bm \xi_1, p_{f_0}$\textup{)} satisfy \eqref{intialcdtns} and \eqref{comp-cdtns}. 
Then, there exists $\overline{T}>0$ such that System \eqref{System at t=0} admits a unique solution defined on $(0,\overline{T})$ satisfying
\begin{align}\label{thm:2}
(\bm v \circ \bm{\mathcal{A}}, \bm \xi \circ \bm \varphi, p_f \circ \bm{\mathcal{A}})
\in
F^{\overline{T}}_4 
\times
S^{\overline{T}}_4
\times 
\mathcal{P}^{\overline{T}}_3
\end{align}
\begin{align}
\bm{\mathcal{A}} \in W^{1,\infty} \big(0,\overline{T};H^2(\Omega_f(0)) \big)\times
W^{2,\infty}\big(0,\overline{T};L^2(\Omega_f(0))\big)
\end{align}
and
\begin{align}
\bm \varphi \in S^{\overline{T}}_2.
\end{align}
\end{theorem}
%
%
%
%
%
%
For simplicity, for all $m,r \geqslant 0$ \ and $ p,q \in [1,+\infty]$, we denote the spaces
$W^{m,p} \big(0,T;W^{r,q}(\Omega_f(0))\big)$
and \\
$W^{m,p} \big(0,T;W^{r,q}(\Omega_s(0))\big)$
by
$W^{m,p}\big(W^{r,q}(\Omega_f(0))\big)$
and
$W^{m,p}\big(W^{r,q}(\Omega_s(0))\big)$, respectively.\\
Also the domain's notation is simplified by writing $\Omega_f(0)=\OF$, $\Omega_s(0)=\OS$ and $\Omega(0)=\Omega_0$. Further, For all $t>0$, define
\[\Sigma_t = \Gamma_c(0) \times (0,t).
\]
%


\section{A Partially Linear System}
\label{Partially linear system}

Let ($\bm v_0,\bm \xi_1,p_{f_0}$) satisfy  \eqref{intialcdtns} and \eqref{comp-cdtns}.
Let $0<T<1$ and consider ($\bm {\breve v},\bm {\breve \xi}$) $\in A^T_M$ to be given. For these given functions we define the associated fluid flow $\tc$ and structure deformation $\bm{\breve{\varphi}}$ by
\begin{equation}\label{GivenFlow}
\tca
(\bm {\tilde x},t)=\bm {\tilde x}+\int^{t}_{0} \bm {\breve v}(\bm {\tilde x},s) \ ds  \qquad \forall \ \bm {\tilde x}  \in \OF,
\end{equation}
and
\begin{equation}\label{GivenDeformation}
\bm{\breve{\varphi}}(\bm {\tilde x},t)=\bm {\tilde x}+ \bm {\breve \xi}(\bm {\tilde x},t)  \qquad \forall \ \bm {\tilde x} \in \OS.
\end{equation}
We use the given $(\bm {\breve v},\bm {\breve \xi})$ to partially linearize the non-linear system. Indeed, we consider the non-linear terms to be given in terms of $(\bm {\breve v},\bm {\breve \xi})$.
%
%
Let $T \leq 1/M^4$ and $M>1$. We shall repeatedly use the following two lemmas which provide
bounds on various norms of the deformation maps 
$\tca$ and 
$\tp$. We omit the proof, for, the bounds are obtained by simple calculations using the generalized Poincar\'e inequality \cite[Proposition III.2.38]{Boyer}, Gr\"onwal inequality and the embedding theorems \cite[Corollary 9.13]{Brezis}.
\begin{lemma}\label{Estimates on Flow}
For the fluid flow $\tca$ given by \eqref{GivenFlow} for a given $\bm {\breve v} \in A^T_{M_1}$, there exists a constant $C=C(\OF)>0$ and a constant $\kappa >0 $ such that
\begin{itemize}
\item[1-] $\norm \tca \norm_{W^{1,\infty}(H^4) \cap W^{3,\infty}(H^2) \cap W^{4,\infty}(L^2) \cap H^4(H^1)}
\leq C(1+M).$
\item[2-] $\norm \nc-\Id \norm_{W^{1,\infty}(H^3) \cap W^{3,\infty}(H^1) \cap H^4(L^2)  }
\leq CM.$
\item[3-] 
$\norm \nc \norm_{L^\infty(H^3)} \leq C$. 
\item[4-]
\label{4-}$\norm \textup{cof}(\nc) \norm_{L^\infty(H^3)}
\leq C$.
\item[5-]
$ \norm \partial_t \textup{cof}(\nc)(t) \norm_{L^2(H^3)}
\leq CT^{1/2}M$.
\item[6-] $\norm (\nc)^{-1}(t) \norm_{L^\infty} \leq C\norm \nc(t) \norm_{L^\infty}^2 \qquad \textrm{for} \ t\in[0,T]$.
\item[7-] 
$\norm \textup{cof}(\nc)-\Id \norm_\mathit{L^\infty(H^3)}
+
\norm (\nc)^{-1}-\Id \norm_\mathit{L^\infty(H^3)}
\leq  \mathit{CT^{\kappa}M}.
$
\item[8-] $\norm \partial_t (\nc)^{-1}(t) \norm_{L^r}
\leq
C \norm \bm \tv(t) \norm_{L^r}, \ \textrm{for} \ r \in [1,+\infty] \ \textrm{and} \ t \in [0,T].$
\item[9-]
$\norm \textup{det}(\nc) \norm_{L^\infty(H^3)}
\leq
CM$
\quad \textup{and} \quad
$\norm \partial_t \textup{det}(\nc) \norm_{L^\infty(H^2)}
\leq 
C M.
$
\item[10-]
$\norm \tdet(\nc)-1 \norm_{L^\infty(H^3)} \leq CT^\kappa M.$
\end{itemize}
\end{lemma}
\begin{remark}
\label{Lower bound on det}
The last part of Lemma \ref{Estimates on Flow} gives
\begin{align*}
\norm \tdet(\nc)-1 \norm_{L^\infty(L^\infty)} \leq CT^\kappa M.
\end{align*}
That is, for all $t$ in $(0,T)$ and $\bm {\tilde x}$ in $\OF$ we have
\[
-CT^\kappa M
\leq
\tdet(\bm \nabla \tca)(\bm {\tilde x},t)- 1 \leq CT^\kappa M.
\]
This gives 
\begin{align*}
\tdet(\bm \nabla \tca)(\bm {\tilde x},t) \geq 1-CT^\kappa M \qquad \forall \ (\bm {\tilde x},t) \in \OF \times (0,T).
\end{align*}
\end{remark}
\begin{remark}
For $0 < n \leq 4$, the quantity $CTM^n$ can be approximated by $CT^\kappa M$, with $\kappa >0$. Indeed, as $TM^4<1$, then we can find $\kappa>0$ such that $TM^4 \leq T^\kappa$.
\end{remark}
%
%


%
%
%
\begin{lemma}\label{Estimate-structure-defo}
Let $M>1$, $T>0$ and $\bm {\breve \xi} \in A^T_{M_2}$ be given. There exists $C>0$ such that for all $i,\al,j,\be \in \{1,2,3\}$, we have:
\begin{itemize}
\item[1-]
\begin{equation}\label{lq-boundedness}
\Big \lvert \Big\lvert \lciajb(\bm \nabla \bm {\breve \xi})+\qciajb(\bm \nabla \bm {\breve \xi}) \Big \lvert \Big\lvert_{S^T_3} \leq C(M+M^2)
\end{equation}
where $\lciajb \ \textrm{and} \ \qciajb$ are defined by the expressions \eqref{linearpart} and \eqref{quadraticpart} respectively.
\item[2-]
For any matrix $\bm A \in \mathcal{M}_3(\R)$, we have
\begin{equation}\label{coercivity}
\sum_{i,\al,j,\be=1}^3 \ciajb(\bm \nabla \bm {\breve \xi})A_{j\be}A_{i \al}
\geq
\frac{\mu_s}{2}    \lvert A+A^t \lvert^2
+\lambda_s            \lvert \textup{tr}(A) \lvert^2
-CT(M+M^2)           \lvert A \lvert^2.
\end{equation}


\item[3-]

\begin{equation}\label{lq1-boundedness}
\Big \lvert \Big\lvert
\diajb^l(\bm \nabla \bm {\breve \xi})
+
\diajb^{Q}(\bm \nabla \bm {\breve \xi})
+
\diajb^{T}(\bm \nabla \bm {\breve \xi})
+
\diajb^{F}(\bm \nabla \bm {\breve \xi})
\Big \lvert \Big\lvert_{S^T_3} \leq C(M+M^2+M^3+M^4).
\end{equation}

\item[4-]
For any matrix $\bm A \in \mathcal{M}_3 (\R)$ we have

\begin{equation}\label{coercivity1}
\sum_{i,\al,j,\be=1}^3 \diajb(\bm \nabla \bm {\breve \xi})A_{j\be}A_{i \al}
\geq
\mathsf{C}           \lvert \textup{tr}(A) \lvert^2
-CT(M+M^2+M^3+M^4)           \lvert A \lvert^2.
\end{equation}


\item[5-]
\begin{align}
\norm \bm \nabla \tp \norm_{L^\infty(H^2(\OS))}
\leq 
C.
\end{align}


\item[6-]
\begin{align}\label{bound of cofactor}
\norm \textup{cof}(\bm \nabla \tp) \norm_{L^\infty(H^2(\OS))}
\leq C
\quad
\text{and}
\quad
\norm \textup{cof}(\bm \nabla \bm {\breve \varphi}) \norm_{L^2(H^2(\OS))}
\leq
CT^{1/2}.
\end{align}


\item[7-]
We have 
\begin{align}\label{det-struct}
\norm \tdet(\bm \nabla \tp) \norm_{L^\infty(H^2)} 
\leq
C 
\quad \text{and} \quad 
\norm \partial_t \tdet(\bm \nabla \tp) \norm_{L^\infty(H^2)} 
\leq
CM.
\end{align}

\item[8-]
\begin{align}
\norm (\bm \nabla \bm {\breve \varphi})^{-1} \norm_{L^\infty(H^2(\OS))}
\leq C.
\end{align}

\item[9-]
For $\bm {\breve \xi} \in A^T_{M_2}$ we have   
\begin{align}\label{bound of det}
\norm \textup{det}(\bm \nabla \tp) - 1 
\norm_{L^\infty(H^2(\OS))}
\leq CTM.
\end{align}


\end{itemize}

\end{lemma}

The main step to establish the local in time existence and uniqueness of solution of the coupled problem is to partially linearize it. This is achieved by considering the non-linear terms to be given, thus the flow map and deformation are given by \eqref{GivenFlow} and \eqref{GivenDeformation} respectively.
For the given $\bm{\breve{\mathcal{A}}}, \bm{\breve \varphi} \ \textrm{and} \ (\bm {\breve v},\bm {\breve \xi}) \in A^T_M$ we denote $\biajb(\bm \nabla \bm {\breve \xi})$ by $ \tbiajb $ and the fluid shear stress is denoted by $\bm {\breve \sigma}^0_f(\bm \vv,\tilde{p}_f)$ when considering $\bm{\breve \mathcal{A}}$ in the expression $\eqref{fluidstress-intial0}$.
Now we write the system \eqref{System at t=0} in the reference configuration at time $t=0$.
Equation \eqref{System at t=0}$_1$ is replaced by
\[
\rho_f \tdet(\bm \nabla \tca) \partial_t \bm \vv -\bm \nabla \cdot \bm {\breve \sigma}^0_f(\bm \vv,\tilde{p}_f) = 0 
\quad \textrm{in}\ \OF\times (0,T)
\]
and Equation \eqref{System at t=0}$_5$ is replaced by
$$
\rho_s \tdet(\bm{\nabla \breve{\varphi}}) \partial^2_t \xit_i
-
\sum_{\al,j,\be=1}^{3}
\tbiajb 
\partial_{\al \be}^2 \xit_j 
=0,
\quad i=1,2,3  \quad \textrm{in}\ \OS\times (0,T).
$$
The coupling conditions on $\Sigma_T$ are given by
\begin{equation}
\begin{cases}
\bm \vv= \partial_t \bm \xxi,\\
\big[\bm{\breve \sigma}^0_f(\bm \vv,\tilde{p}_f) \ \bm {\tilde n}\big]_i =
\mathlarger{\sum}\limits_{\al,j,\be =1}^3 \biggl(\displaystyle \int_0^t \tbiajb \partial^2_{s \be} \xit_j ds \biggr)\tilde{n}_{\al}
\quad \textrm{for} \quad i=1,2,3.
\end{cases}
\end{equation}
For $(\bm {\breve v},\bm {\breve \xi})$ being given in $A^T_M$ , we introduce the following mapping
\[
\Psi: (\bm {\breve v},\bm {\breve \xi}) \longrightarrow (\bm \vv,\bm \xxi)
\]
where $(\bm \vv,\bm \xxi)$ together with $\tilde{p}_f$ form the solution of the partially linearized system.
\\
\par
First, we start by defining an auxiliary problem that considers the boundary condition \eqref{Elast-natural-bdrycdtn}. Choosing a suitable functional space we write the variational formulation where the pressure term disappears. Uniqueness and existence of solution of the auxiliary problem are established in the next section.

\section{An Auxiliary Problem}\label{An Auxiliary Problem}
As we mentioned before, there is a disagreement between the elasticity equation and the stress coupling condition on $\Sigma_T$ attributed to it. Thus, we set up an auxiliary problem in which the natural boundary condition \eqref{Elast-natural-bdrycdtn} is used. This problem constitutes the first tool in establishing the existence and uniqueness of the strong solution of the FSI problem.
We start by introducing the auxiliary problem.
Let  $\bm g =[g_1,g_2,g_3]^t$ be a function in 
$H^1_l \big(\left[0,T\right];L^2(\Gamma_c(0))\big)$,
and consider the following system:
\begin{equation}\label{Auxiliary-pro}
\begin{cases}
 \rho_f \tdet(\bm \nabla \tc)\partial_t \bm \vv -\bm \nabla \cdot \bm{\breve \sigma}^0_f(\bm \vv,	\tilde{p}_f) = 0 & \textrm{in $\OF \times (0,T) $}, 
  		 \\
    \nabla \cdot \big(\tdet(\bm {\nabla \tc})(\bm{\nabla \tc})^{-1} \bm \vv \big) = 0 & \textrm{in $ \OF \times (0,T), $} 						
    \\
    \bm \vv=\bm v_\textup{in} \circ \bm \tc & \textrm{on $\Gamma_{\textrm{in}}(0) \times (0,T)$, }
    \\
    \bm {\breve \sigma}^0_f(\bm \vv,\tilde{p}_f) \ \bm {\tilde n} = 0 & \textrm{on $\Gamma_{\textrm{out}}(0) \times (0,T)$, }		
    \\
   	\rho_s \tdet(\bm {\nabla \tp}) \partial^2_t \tilde{\xi}_i
   	- \displaystyle \sum_{\al,j,\be=1}^{3}
   	 \tbiajb \partial_{\al \be}^2 \xit_j
   	=0 
   	\quad i=1,2,3,
   	& 			 \textrm{in $\OS\times (0,T)$},
   	 \\
	\bm \xxi = \bm 0  & \textrm{on $ \Gamma_2(0) \times (0,T) $},
	\\
	\bm \vv = \partial_t \bm \xxi & \textrm{on $ \Gamma_c(0) \times (0,T) $},					
	\\
	\left[\bm{\breve \sigma}^0_f(\bm \vv,\tilde{p}_f) \bm {\tilde n} \right]_i
	=
	\displaystyle \sum_{\al,j,\be=1}^3
	\Big(
	 \tbiajb \partial_{\be} \xit_j
	\Big) \tilde{n}_{\al}
	+g_i \quad i=1,2,3,
	 & \textrm{on $\Gamma_c(0) \times (0,T)$},
	 \\
	\bm \vv(.,0)=\bm v_0 ,\quad \textrm{and} \quad \tilde{p}_f(.,0)=p_{f_0} & \textrm{in $\OF$,}			
	\\
   	\bm \xxi(.,0)=0 \quad \textrm{and} \quad  \partial_t \bm \xxi(.,0)=\bm \xi_1  & \textrm{in $\OS$}.
   
\end{cases}
\end{equation}
The following lemma states the existence and uniqueness of solution for the auxiliary problem.
\begin{lemma}\label{Ex-Uni-AuxPro}
Let $(\bm {\breve v},\bm {\breve \xi}) \in A^T_M$, $\bm v_0 \in L^2(\OF), \ \bm \xi_1 \in L^2(\OS)\ \textrm{and} \ p_{f_0} \in L^2(\OF) $. For $T$ small with respect to $M$ and the initial conditions, there exists a unique weak solution $(\bm \vv,\bm \xxi) \in F^T_1 \times S^T_1$ of \eqref{Auxiliary-pro}. In addition, this solution satisfies the following a priori estimate
\begin{equation}\label{a-priori-es1}
\norm \bm \vv \norm^2_{F^T_1} + \norm \bm \xxi \norm^2_{S^T_1}
\leq
C
\Big[
\dfrac{\rho_f}{2} \norm \bm v_0 \norm^2_{L^2(\OF)}
+
\dfrac{\rho_s}{2} \norm \bm \xi_1 \norm^2_{L^2(\OS)}
+
 \norm \bm g \norm^2_{H^1(L^2(\Gamma_c(0)))}
 \Big].
\end{equation}
\end{lemma}
\begin{remark}
Taking $T$ small with respect to $M$ and the initial conditions, means that there exists $n_0>0$ and $\varepsilon$ positive such that
\begin{align*}
T \leq \Bigg\{\dfrac{\varepsilon}{M^{n_0}}, \dfrac{\varepsilon}
{h(\norm \bm v_0 \norm_{H^6(\OF)}, \norm \bm \xi_1 \norm_{H^3(\OS)}, \norm p_{f_0} \norm_{L^2(\OF)})} \Bigg\}.
\end{align*}
\end{remark}

\indent From here on, we simplify the notation
for all the norms by omitting the indication for the domain as it is always clear from
the context. For instance, we write $\norm \bm \vv \norm_{L^2} = \norm \bm \vv \norm_{L^2(\OF)}$ and $\norm \bm \xxi \norm_{L^2} = \norm \bm \xxi \norm_{L^2(\OS)}$.
%


%
\par
In order to prove Lemma \ref{Ex-Uni-AuxPro} we proceed as follows. First, we write the variational formulation corresponding to the coupled system using a divergence-free functional space. Then, we use a Faedo-Galerkin approach to find an approximation of the solution, which enables us to find some a priori estimates on the Galerkin sequences. Using the estimates and compactness results we prove the existence and uniqueness of the solution.

\subsection{Variational Formulation}
Consider the following divergence-free functional space
\begin{equation*}
\widetilde{\mathcal{W}}=
\Big\{
\bm {\tilde \eta} \in H^1(\Omega_0)|\  \nabla \cdot (\tdet(\bm {\nabla \tc})(\bm{\nabla \tc})^{-1} \bm {\tilde \eta})=0 \ \textrm{on} \ \OF
\ \textrm{and} \ \bm {\tilde \eta}=0 \ \textrm{on} 
\ \Omega_0 \setminus \tilde{\Gamma}_\textup{out}(0)
\Big\}.
\end{equation*}

Let $[[.,.]]$ denote the weighted $L^2$ inner product defined by
%
%
\begin{align*}
[[\bm {\tilde \gamma},\bm {\tilde \eta}]]= \int_{\OF} 
\rho_f \bm {\tilde \gamma} \cdot \bm {\tilde \eta} \ d\bm {\tilde x} 
+ 
\int_{\OS} 
\rho_s \bm {\tilde \gamma} \cdot \bm {\tilde \eta} \ d\bm {\tilde x} \qquad \forall \ \bm {\tilde \gamma}, \bm {\tilde \eta} \in \widetilde{\mathcal{W}}.
\end{align*} 
This norm is equivalent to the norm
$\norm \ . \ \norm_{L^2(\Omega_0)}$.

In order to derive the variational formulation of \eqref{Auxiliary-pro}, we multiply Equations \eqref{Auxiliary-pro}$_1$ and \eqref{Auxiliary-pro}$_5$ by a test function $\bm {\tilde \eta} \in \widetilde{\mathcal{W}}$, integrate by parts and take into consideration the boundary and the coupling conditions to get


\begin{equation}\label{final-eqn-weak}
\begin{cases}

\rho_f \displaystyle \int_{\OF} 
\tdet(\bm \nabla \tca) \partial_t \bm \vv \cdot \bm {\tilde \eta} \ d\bm {\tilde x}

+

\int_{\OF} \bm {\breve \sigma}^0_f (\bm {\tilde v}) :\bm \nabla \bm {\tilde \eta} \ d\bm {\tilde x}

+\rho_s  \int_{\OS}\tdet(\bm \nabla \tp) \partial^2_t \bm \xxi \cdot \bm {\tilde \eta} \ d\bm {\tilde x}

\vspace{1mm} \\

+
\displaystyle \sum_{i,\al,j,\be=1}^3 \int_{\OS} \tbiajb \partial_\be \xit_j \ \partial_\al \tilde {\eta}_i \ d\bm {\tilde x}

\displaystyle

+ \sum_{i,\al,j,\be=1}^3 \displaystyle \int_{\OS} \partial_\al \tbiajb \partial_\be \xit_j \ \tilde {\eta}_i \ d\bm {\tilde x}



=
\displaystyle
 \int_{\Gamma_c(0)} \bm g \cdot \bm {\tilde \eta} \ d\tilde{\Gamma}

\qquad \forall \ \bm {\tilde \eta} \in \widetilde{\mathcal{W}}.

\end{cases}
\end{equation}
Note that, the space $\widetilde{\mathcal{W}}$ is the transformation of the space
\begin{equation*}
\mathcal{W}=\Big\{ \bm {\eta} \in H^1(\Omega(t)) \mid
\nabla \cdot \bm {\eta}=0 \quad \textrm{on} \quad \Omega_f(t)
\quad \textrm{and} \quad \bm {\eta}=0 \quad \textrm{on} \quad \partial \Omega(t)\setminus \Gamma_\textup{out}(t)\Big\}.
\end{equation*}
This explains the disappearance of the pressure term $\tilde{p}_f$ from the weak formulation.
\begin{remark}
":" corresponds to the Hadamard product of matrices defined by
\begin{align*}
\bm A: \bm B = \sum_{i,j=1}^n A_{i,j} B_{i,j} ,\ \textrm{for} \ \bm A, \bm B \in \mathbb{M}_n(\R).
\end{align*}
\end{remark}

\par
In order to derive the weak formulation we consider a global test function $\bm {\tilde \eta}$ in $\widetilde{\mathcal{W}}$. This will simplify the work. In fact, rather than looking for two solutions using two independent test functions on each sub-domain, we search for one solution $\bm {\tilde \gamma}$ over the domain $\Omega_0$. By considering a global test function we are able to embed the stress condition into the formulation in such a way that it would cancel out on the entire domain.
Further, we will guarantee the existence of a weak solution $\bm {\tilde \gamma}$ in $\widetilde{\mathcal{W}}$. Consequently, $\bm \vv$ and $\bm \xxi$ are considered to be the restriction of $\bm {\tilde \gamma}$ on the sub-domains $\OF$ and $\OS$, respectively. Note that, if we consider the restriction of $\bm {\tilde \eta}$ on the two sub-domains $\OF$ and $\OS$, we cannot guarantee the existence of the weak solutions in the restriction of $\widetilde{\mathcal{W}}$ on each sub-domain. Thus, we introduce the auxiliary function $\bm {\tilde \gamma}$ defined by
\begin{equation}\label{auxiliary-fn}
\bm {\tilde \gamma}=
\begin{cases}
\bm \vv & \textrm{in} \ \OF,\\
\partial_t \bm \xxi & \textrm{in} \ \OS,
\end{cases}
\qquad \textrm{and} \qquad
\bm {\tilde \gamma}_0=
\begin{cases}
\bm v_0 & \textrm{in} \ \OF,\\
\bm \xi_1 & \textrm{in} \ \OS,
\end{cases}
\end{equation}
which is a continuous function on $\Omega_0$, due to the continuity of velocities across the interface $\Gamma_c(0)$ which is given by the condition \eqref{Auxiliary-pro}$_7$.
%
%
%
By this definition, we can write $\bm \vv(t)=\bm {\tilde \gamma}(t)$ on $\OF$, and $\bm \xxi(t) = \int_0^t \bm {\tilde \gamma}(s) ds $ on $\OS$, based on the fact that $\bm \xxi(0)=\bm \xi_0=0$. Then, for all test functions $\bm {\tilde \eta}$ in $\widetilde{\mathcal{W}}$ the weak formulation \eqref{final-eqn-weak} is equivalent to 
%
%
\begin{equation}\label{final-eqn-weak-1}
\begin{cases}
\rho_f \displaystyle \int_{\OF} 
\tdet(\bm \nabla \tca)\partial_t \bm {\tilde \gamma} \cdot \bm {\tilde \eta} \ d\bm {\tilde x}

+\rho_s \int_{\OS} \tdet(\bm \nabla \tp) \partial_t \bm {\tilde \gamma} \cdot \bm {\tilde \eta} \ d\bm {\tilde x}

\vspace{1mm}\\

\displaystyle
+\int_{\OF} \bm {\breve \sigma}^0_f(\bm {\tilde \gamma}) : \bm \nabla \bm {\tilde \eta} \ d\bm {\tilde x}

+\displaystyle \sum_{i,\al,j,\be=1}^3 \int_{\OS} \tbiajb \partial_\be ( \mathsmaller \int_0^t  \tilde{\gamma}(s) ds)_j \ \partial_\al \tilde {\eta}_i \ d\bm {\tilde x}

\vspace{1mm}\\

+ \displaystyle \sum_{i,\al,j,\be=1}^3 \int_{\OS} \partial_\al \tbiajb \partial_\be (\mathsmaller \int_0^t  \tilde{\gamma}(s) ds)_j \ \tilde {\eta}_i \ d\bm {\tilde x}

=
\displaystyle \int_{\Gamma_c(0)} \bm g \cdot \bm {\tilde \eta} \ d\tilde{\Gamma},


\vspace{1mm} \\

\bm \tilde{\gamma}(0) = \bm \tilde{\gamma}_0,

\vspace{1mm} \\

\displaystyle \int_0^t \Big(\bm {\tilde \gamma}(s)|_{\OF} \Big) \bigg|_{\Gamma_c(0)} \ ds
=
 \int_0^t \Big(\bm {\tilde \gamma}(s)|_{\OS} \Big) \bigg|_{\Gamma_c(0)} \ ds, \qquad \forall \ t \in [0,T].

\end{cases}
\end{equation}


\subsection{Galerkin Approximation}
In order to show that the system admits a unique solution we will use a Faedo-Galerkin approach.
Let $ \{\bm \psi_l\}_{l=1}^n$ be a basis of $\widetilde{\mathcal{W}}$ in $L^2(\Omega_0)$ which is orthogonal for the $H^1$-Norm and orthonormal for the $L^2$-Norm.
%
\\
Take
$\widetilde{\mathcal{W}}_n = \textrm{span} \{\bm \psi_1, \ldots ,\bm  \psi_n \}.$
We seek to find a Galerkin approximation
$\{\bm {\tilde \gamma}_n\}_n \in \mathcal{C}^1(0,T;\widetilde{\mathcal{W}}_n) $ of the form
\addtolength{\parindent}{-5mm}
\begin{align}\label{galerkin-form}
\bm {\tilde \gamma}_n = \sum_{l=1}^n f^n_l(t) \bm \psi_l(\bm {\tilde x})
\end{align}
satisfying
\begin{equation}\label{Galerkin-coupling-formulation}
\begin{cases}
\rho_f \displaystyle \int_{\OF} 
\tdet(\bm \nabla \tca)\partial_t \bm {\tilde \gamma}_n \cdot \bm {\tilde \eta}_n \ d\bm {\tilde x}

+\rho_s \displaystyle \int_{\OS} \tdet(\bm \nabla \tp) \partial_t \bm {\tilde \gamma}_n \cdot \bm {\tilde \eta}_n \ d\bm {\tilde x}

\vspace{1mm}\\
+
\displaystyle \int_{\OF} \bm {\breve \sigma}^0_f(\bm {\tilde \gamma}_n) :\bm \nabla \bm {\tilde \eta}_n \ d\bm {\tilde x}

+\displaystyle\sum\limits_{i,\al,j,\be=1}^3 \displaystyle \int_{\OS} 
\tbiajb \partial_\be (\mathsmaller \int_0^t  \tilde{\gamma}_n(s) ds)_j \ \partial_\al \tilde{\eta}_{n,i} \ d\bm {\tilde x} 

\vspace{1mm}\\
+ \displaystyle \sum\limits_{i,\al,j,\be=1}^3 \displaystyle \int_{\OS} \partial_\al \tbiajb \partial_\be (\mathsmaller \int_0^t  \tilde{\gamma}_n(s) ds)_j \ \tilde{\eta}_{n,i} \ d\bm {\tilde x}

=
\displaystyle \int_{\Gamma_c(0)} \bm  g \cdot \bm {\tilde \eta}_n \ d\tilde{\Gamma},

\quad \forall \ \bm {\tilde \eta}_n \in \widetilde{\mathcal{W}}_n,

\end{cases}
\end{equation}
and
\begin{align}\label{galerkin-initial-cdtn}
[[\bm {\tilde \gamma}_n(0), \bm {\tilde \eta}_n]]= [[\bm {\tilde \gamma}_0,\bm {\tilde \eta}_n]] , \qquad \forall \ \bm {\tilde \eta}_n \in \widetilde{\mathcal{W}}_n.
\end{align}
Notice that, trivially  $\bm {\tilde \gamma}_n$ defined in \eqref{galerkin-form} satisfies
%
%
%
\begin{align}\label{galerkin-boundary}
\displaystyle
\int_0^t \Big(\bm {\tilde \gamma}_n(s)|_{\OF} \Big) \bigg|_{\Gamma_c(0)} \ ds
=
\int_0^t \Big(\bm {\tilde \gamma}_n(s)|_{\OS} \Big) \bigg|_{\Gamma_c(0)} \ ds, \qquad \forall \ t \in [0,T].
\end{align}
%
%
We can write \eqref{Galerkin-coupling-formulation}-\eqref{galerkin-initial-cdtn} as an equivalent system of first-order, linear ordinary differential equation (ODE) for $\{f_l^n\}_{l=1}^n$.\\
Set $h^n_l(t)=\int_0^t f^n_l(s) ds$ for $l=1,\cdots,n$. For $1 \leq k \leq n$, the problem \eqref{Galerkin-coupling-formulation}-\eqref{galerkin-initial-cdtn} is equivalent to the following ODE initial value problem
\begin{equation}\label{ODE}
\begin{cases}
\displaystyle
\sum_{l=1}^n 
\dfrac{d}{dt} f_l^n(t) 
\Bigg[

\rho_f
\int_{\OF} \tdet(\bm \nabla \tca) \bm  \psi_l \cdot \bm \psi_k \ d\bm {\tilde x} 

+
\rho_s
\int_{\OS} \tdet(\bm \nabla \tp)\bm  \psi_l \cdot \bm  \psi_k \ d\bm {\tilde x}

\Bigg]

\vspace{1mm} \\
+
\displaystyle
\sum_{l=1}^n

f_l^n(t)\displaystyle \int_{\OF} \bm {\breve \sigma}^0_f(\bm \psi_l) :\bm \nabla \bm \psi_k \ d\bm {\tilde x}

\vspace{1mm} \\

+
\displaystyle
\sum_{l=1}^n

h^n_l(t)
 \underbrace{\Bigg(
\displaystyle\sum \limits_{i,\al,j,\be=1}^3 \displaystyle \int_{\OS} 
\bigg[
\tbiajb \partial_\be \psi_{l,j} \ \partial_\al \psi_{k,i}
+ \partial_\al \tbiajb \partial_\be \psi_{l,j} \ \psi_{k,i} \ d\bm {\tilde x}
\bigg]
\Bigg)}_\mathlarger{[D_{i,\al,j,\be}]_{k}}

\vspace{1mm} \\

=
\displaystyle \int_{\Gamma_c(0)} \bm  g \cdot \bm \psi_k \ d\tilde{\Gamma},

\vspace{1mm} \\


\Large{\dfrac{d}{dt}h^n_l(t)=f^n_l(t)}  \qquad \forall \ 1 \leq l \leq n,
\vspace{1mm} \\

\displaystyle

\sum_{l=1}^n

[[\bm \psi_l,\bm \psi_k]] f^n_l(0)
=
[[ \bm {\tilde \gamma}_0, \bm \psi_k ]], 
\vspace{1mm} \\

\Large h^n_l(0)=0    \qquad \forall \ 1 \leq l \leq n.\\
\end{cases}
\end{equation}
System \eqref{ODE} can be rewritten in the following matrix form 
\small
{
\begin{equation}\label{Matrix-ODE}
\mathop{
\left[
\begin{array}{c | c}
\\
\left[[\bm \psi_l,\bm \psi_k]\right]_{l,k=1}^n 
& \displaystyle{\bm 0}_n
\\ \\
\hline
\\
\displaystyle{\bm 0}_n
& \displaystyle{\bm I}_n
 \\ \\
\end{array}
\right]
}
_{\textstyle \bm A}
\dfrac{d}{dt}
\mathop
{
\left[
\begin{array}{cc}
f_1^n(t) \\
\vdots \\
f^n_n(t)\\
\hline\\
h_1^n(t) \\
\vdots \\
h^n_n(t)\\
\end{array}
\right]^{\mathbf{}}
}_{\textstyle {\dfrac{d}{dt} \bm F}}
=
\mathop{
\left[
\begin{array}{c | c}
\\
\bm S_n
&[\bm D]_n
\\ \\
\hline
\\
\displaystyle{\bm I}_n 
& \displaystyle{\bm 0}_n 
\\ \\
\end{array}
\right]}
_{\textstyle \bm B}
\mathop
{
\left[
\begin{array}{cc}
f_1^n(t) \\
\vdots \\
f^n_n(t)\\
\hline \\
h_1^n(t)\\
\vdots \\
h^n_n(t)\\
\end{array}
\right]}
_{\textstyle \bm F}
+
\mathop
{
\left[
\begin{array}{cc}
\\
\left[
\displaystyle \int_{\Gamma_c(0)} \bm g \cdot \bm \psi_k \ d\tilde{\Gamma} \right]_{k=1}^n
\\ \\
\hline\\
\displaystyle{\bm 0}_{n \times 1}\\ \\
\end{array}
\right]}
_{\textstyle \bm C}
\end{equation}
}
with
\[
[[\bm \psi_l,\bm \psi_k]]_{l,k=1}^n
=
\displaystyle
\left[\rho_f \int_{\OF} \tdet(\bm \nabla \tca) \bm \psi_l \cdot \bm \psi_k \ d\bm {\tilde x}
+
\rho_s  \int_{\OS} \tdet(\bm \nabla \tp) \bm \psi_l \cdot \bm \psi_k \ 
d\bm {\tilde x}
\right]_{l,k=1}^n,
\]
\[
\displaystyle
\bm S_n=
\left[\int_{\OF} \bm {\breve \sigma}^0_f(\bm \psi_l) :\bm \nabla \bm \psi_k \ d\bm {\tilde x} \right]_{l,k=1}^n
\quad \textup{and} \quad
%
\displaystyle
\left[\bm D \right]_n=
[D_{i,\al,j,\be}]_{l,k=1}^n.
\]
The matrix $\bm A$ is a positive definite matrix as the function set $ \left\lbrace \bm  \psi_i \right\rbrace_{i=1}^n $ is linearly independent. Moreover, $\bm A$ is bounded on $\left(0,T \right)$. Further, matrices $\bm B$ and $\bm C$ are bounded on $(0,T)$.
%
%
%
Hence by theory for systems of linear first order ODEs, we get that system \eqref{ODE} admits a unique $\mathcal{C}^1$-solution $\left\{f_1^n ,\ldots,f_n^n,h_1^n,\ldots,h_n^n \right\} $ which yields the existence of a unique Galerkin approximation $\{\bm {\tilde \gamma}_n \}_n$ of \eqref{Galerkin-coupling-formulation}-\eqref{galerkin-initial-cdtn} such that $
\bm {\tilde \gamma}_n \in W^{1,\infty}(0,T;H^1(\Omega_0) ).
$
\\
Now we proceed to derive \textit{a priori} estimates on $\bm {\tilde \gamma}_n$.
%
\subsection{A Priori Estimates}
\label{A Priori Estimates}
%
%
%
\subsubsection*{Step 1: Estimates on $\bm {\tilde \gamma}_n$}
We aim to find some estimates on $\bm {\tilde \gamma}_n$. 
For this sake we set $\bm {\tilde \eta}_n =\bm {\tilde \gamma}_n$ in \eqref{Galerkin-coupling-formulation} to get
%
%
%
\begin{equation}
\begin{cases}

\rho_f \displaystyle \int_{\OF} 
\tdet(\bm \nabla \tca)\partial_t \bm {\tilde \gamma}_n \cdot \bm {\tilde \gamma}_n \ d\bm {\tilde x} +
\displaystyle \int_{\OF} \bm {\breve \sigma}^0_f(\bm {\tilde \gamma}_n) :\bm \nabla \bm {\tilde \gamma}_n \ d\bm {\tilde x}

+\rho_s \displaystyle \int_{\OS} \tdet(\bm \nabla \tp) \partial_t \bm {\tilde \gamma}_n \cdot \bm {\tilde \gamma}_n \ d\bm {\tilde x}

\\

+\displaystyle\sum\limits_{i,\al,j,\be=1}^3 \displaystyle \int_{\OS} \tbiajb\partial_\be (\mathsmaller \int_0^t  \tilde{\gamma}_n(s) ds)_j 
\ 
\partial^2_{\al t}
(\mathsmaller \int_0^t  \tilde{\gamma}_n(s) ds)_i \ d\bm {\tilde x} 

\\

+ \displaystyle \sum\limits_{i,\al,j,\be=1}^3 \displaystyle \int_{\OS} \partial_\al \tbiajb \partial_\be (\mathsmaller \int_0^t  \tilde{\gamma}_n(s) ds)_j \ \tilde{\gamma}_{n,i} \ d\bm {\tilde x}

=
\displaystyle \int_{\Gamma_c(0)} \bm g \cdot \bm {\tilde \gamma}_n \ d\tilde{\Gamma}. 
\end{cases}
\end{equation}
%
%
%
Then, integrating over $(0,t)$ and applying integration by parts yield
%
%
%
%
\begin{equation}\label{Galerkin-2}
\begin{cases}

\dfrac{\rho_f}{2} \displaystyle \int_{\OF}
\tdet(\bm \nabla \tca)(t)\lvert \bm {\tilde \gamma}_n(t) \lvert^2 \ d\bm {\tilde x} +


\dfrac{\rho_s}{2} \displaystyle \int_{\OS}
\tdet(\bm \nabla \tp)(t)
\lvert \bm {\tilde \gamma}_n(t) \lvert^2 \ d\bm {\tilde x} 


\vspace{1mm} \\

+\displaystyle \int_0^t
\displaystyle \int_{\OF}
\dfrac{\mu}{2} \tdet(\nc) \lvert \bm \nabla \bm  {\tilde \gamma}_n (\nc)^{-1}+ (\nc)^{-t} (\bm \nabla \bm {\tilde \gamma}_n)^t \lvert^2 \ d\bm {\tilde x} \ ds 


\vspace{1mm} \\

+\dfrac{1}{2}\displaystyle\sum\limits_{i,\al,j,\be=1}^3 \displaystyle \int_{\OS} \tbiajb(t) \ \partial_\be (\mathsmaller \int_0^t  \tilde{\gamma}_n(s) ds)_j \ \partial_\al (\mathsmaller  \int_0^t  \tilde{\gamma}_n(s) ds)_i \ d\bm {\tilde x}

\vspace{1mm} \\

-

\dfrac{\rho_s}{2}
\mathlarger \int_0^t
\displaystyle \int_{\OS}
\partial_s \tdet(\bm \nabla \tp)
\lvert \bm {\tilde \gamma}_n \lvert^2 \ d\bm {\tilde x} \ ds 


\vspace{1mm} \\

-\dfrac{1}{2} \displaystyle \sum\limits_{i,\al,j,\be=1}^3 \displaystyle \int_0^t
\displaystyle \int_{\OS}
\partial_s \tbiajb \ 
\partial_\be (\mathsmaller \int_0^s \tilde{\gamma}_n(\tau) d\tau)_j \ \partial_\al (\mathsmaller \int_0^s \tilde{\gamma}_n(\tau) d\tau)_i \ d\bm {\tilde x} \ ds

\vspace{1mm} \\

-\dfrac{\rho_f}{2}
\mathlarger \int_0^t
\mathlarger \int_{\OF}
\partial_s \tdet(\bm \nabla \tca) \lvert \bm {\tilde \gamma}_n \lvert^2 \ d\bm {\tilde x} \ ds

\vspace{1mm} \\


+ \displaystyle \sum\limits_{i,\al,j,\be=1}^3 \displaystyle \int_0^t
\displaystyle \int_{\OS}
\partial_\al \tbiajb \partial_\be (\mathsmaller \int_0^s \tilde{\gamma}_n(\tau) d\tau)_j  \partial_s (\mathsmaller \int_0^s \tilde{\gamma}_n(\tau) d\tau)_i \ d\bm {\tilde x} \ ds


\vspace{1mm} \\
=
\displaystyle \int_0^t
\displaystyle \int_{\Gamma_c(0)}
\bm g \cdot  \bm {\tilde \gamma}_n \ d\tilde{\Gamma} \ ds



+\dfrac{\rho_f}{2} \displaystyle \int_{\OF}
\lvert \bm {\tilde \gamma}_n(0) \lvert^2 \ d\bm {\tilde x} +


\dfrac{\rho_s}{2} \displaystyle \int_{\OS}
\lvert \bm {\tilde \gamma}_n(0) \lvert^2 \ d\bm {\tilde x}.

\end{cases}
\end{equation}
We start by deriving estimates on the terms of
\eqref{Galerkin-2}.
\\
First of all, as $ \tdet(\bm \nabla \tca)-1 \geqslant -CT^\kappa M$, then we have
\begin{equation}\label{fluid-est1}
\begin{split}
&\dfrac{\rho_f}{2}
\mathlarger \int_{\OF}
\tdet(\bm \nabla \tca) \lvert \bm {\tilde \gamma}_n(t) \lvert^2 \ d\bm {\tilde x}
-
\dfrac{\rho_f}{2}
\mathlarger \int_0^t 
\mathlarger \int_{\OF} 
\partial_s \tdet(\bm \nabla \tca) \lvert \bm {\tilde \gamma}_n \lvert^2 \ d\bm {\tilde x} \ ds
\\
&
\geq
\dfrac{\rho_f}{2}
(1-CT^\kappa M) \norm \bm {\tilde \gamma}_n \norm^2_{L^\infty(L^2)}
-
\dfrac{\rho_f}{2}
\norm \partial_t \tdet(\bm \nabla \tca) \norm_{L^2(H^2)} \norm \bm {\tilde \gamma}_n \norm^2_{L^\infty(L^2)}
\\
& 
\geq
\dfrac{\rho_f}{2}
(1-CT^\kappa M-CT^\kappa M)
\norm \bm {\tilde \gamma}_n \norm^2_{L^\infty(L^2(\OF))}.
\end{split}
\end{equation}
The fluid stress term is decomposed as follows%
\begin{align*}
\displaystyle \int_0^t
\int_{\OF}
\dfrac{\mu}{2} &\lvert \nga_n (\nc)^{-1}+ (\nc)^{-t} (\nga_n)^t \lvert^2 \ d\bm {\tilde x} \ ds =N_1+N_2.
\end{align*}
For $N_1$ we use Korn's inequality and Lemma \ref{Estimates on Flow}, then there exits $C_k >0$ such that
\begin{align*}
N_1
=
\displaystyle \int_0^t
\int_{\OF}
\dfrac{\mu}{2}
\Big\lvert \nga_n (\nc)^{-1}+ (\nc)^{-t} (\nga_n)^t \Big\lvert^2 \ d\bm {\tilde x} \ ds
%
& 
\geq
\displaystyle \int_0^t
\int_{\OF}
\mu \Big\lvert \bm \epsilon(\bm {\tilde \gamma}_n) \Big\lvert^2
-
\mu \Big \lvert \nga_n \Big((\nc)^{-1}-\Id \Big) \Big\lvert^2 \ d\bm {\tilde x} \ ds.
\\
& 
\geq
\mu (C_k-CT^{\kappa}M) \lvert \lvert \bm {\tilde \gamma}_n \lvert \lvert^2_{L^2(H^1(\OF))}.
\end{align*}
Similarly for $S_2$ we use Lemma \ref{Estimates on Flow} which yields
\begin{align*}
N_2 =
\displaystyle \int_0^t \int_{\OF}
\dfrac{\mu}{2}
(\tdet(\bm \nabla \tca)-1)
\Big\lvert \nga_n (\nc)^{-1}+ (\nc)^{-t} (\nga_n)^t \Big\lvert^2
\ d\bm {\tilde x} \ ds
%
%
%
&
\geq
-\norm \tdet(\nc)-1 \norm_{L^\infty(H^3)}
\norm \bm {\tilde \gamma}_n \norm^2_{L^2(H^1)}
\norm (\nc)^{-1} \norm^2_{L^\infty(H^3)}
\nonumber \\
& 
\geq 
\mu CT^\kappa M \norm \bm {\tilde \gamma}_n \norm^2_{L^2(H^1(\OF))}.
\end{align*}
Therefore,
\begin{align}\label{stress-fluid-estimate}
N_1 + N_2 \geq
\mu (C_k-CT^{\kappa}M) \lvert \lvert \bm {\tilde \gamma}_n \lvert \lvert^2_{L^2(H^1(\OF))}.
\end{align}
As for the integrals on the domain $\OS$, first of all we have
\begin{align}
\mathlarger \int_{\OS}
\tdet(\bm \nabla \tp) \lvert \bm  \gamma_n(t) \lvert^2 \ d\bm {\tilde x}
\geq
(1-CT^\kappa M) \norm \bm {\tilde \gamma}_n \norm^2_{L^\infty(L^2(\OS))}.
\end{align}
Thanks to \eqref{det-struct} it holds
\begin{equation}
\begin{split}
\mathlarger \int_0^t
\mathlarger \int_{\OS}
\Big \lvert
\partial_s \tdet(\bm \nabla \tp) \bm {\tilde \gamma}^2_n(t) 
\Big \lvert \ d\bm {\tilde x} \ ds
&\leq
T
\norm \partial_t \tdet(\bm \nabla \tp) \norm_{L^\infty(H^2(\OS))}
\norm \bm {\tilde \gamma}_n \norm^2_{L^\infty(L^2(\OS))}
\leq
CT^\kappa M \norm \bm {\tilde \gamma}_n \norm^2_{L^\infty(L^2(\OS))}
.
\end{split}
\end{equation}
Using \eqref{coercivity} and \eqref{coercivity1} together with Korn's Inequality give
\begin{equation}\label{structure-estimate1}
\begin{split}
\dfrac{1}{2}
\displaystyle\sum\limits_{i,\al,j,\be=1}^3 \displaystyle \int_{\OS} \tbiajb(t) \partial_\be (\mathsmaller \int_0^t \tilde{\gamma}_n(s) ds)_j \ \partial_\al (\mathsmaller \int_0^t \tilde{\gamma}_n(s) ds)_i  \ d\bm {\tilde x}
%
\geq &
\mu_s C_k \lvert \lvert \mathsmaller \int_0^t \bm {\tilde \gamma}_n(s) ds \lvert \lvert^2_{H^1(\OS)}
+\dfrac{\mathsf{C}+\lambda_s}{2} \lvert \lvert \nabla \cdot (\mathsmaller \int_0^t \bm {\tilde \gamma}_n(s) ds) \lvert \lvert^2_{L^2(\OS)}
\\
&- CT^{\kappa} M \norm \mathsmaller \int_0^t \bm {\tilde \gamma}_n(s) ds \norm^2_{H^1(\OS)}.
\end{split}
\end{equation}
On the other hand, using \eqref{lq-boundedness} and \eqref{lq1-boundedness} in addition to Young's inequality \cite[Proposition II.2.16]{Boyer} and the Sobolev embeddings yield
\begin{equation}\label{structure-estimate2}
\begin{split}
\Bigg \lvert
&
-\dfrac{1}{2} 
\mathlarger 
\sum\limits_{i,\al,j,\be=1}^3 \displaystyle \int_0^t
\displaystyle \int_{\OS}
\partial_s \tbiajb \partial_\be (\mathsmaller \int_0^s \tilde{\gamma}_n(\tau) d\tau)_j \ \partial_\al (\mathsmaller \int_0^s \tilde{\gamma}_n(\tau) d\tau)_i \ d\bm {\tilde x} \ ds
\\
&
+\dfrac{1}{2} \displaystyle \sum\limits_{i,\al,j,\be=1}^3 \displaystyle \int_0^t
\displaystyle \int_{\OS}
\partial_\al \tbiajb \partial_\be (\mathsmaller \int_0^s \tilde{\gamma}_n(\tau) d\tau)_j \ \partial_s (\mathsmaller \int_0^s \tilde{\gamma}_n(\tau) d\tau)_i \ d\bm {\tilde x} \ ds
\Bigg \lvert
\\
 & \leq
CT^\kappa M
\left[
\lvert \lvert \mathsmaller \int_0^{\bigcdot} \bm {\tilde \gamma}_n(s) ds \lvert \lvert^2_{L^\infty(H^1(\OS))}
+
\lvert \lvert \bm {\tilde \gamma}_n \lvert \lvert^2_{L^\infty(L^2(\OS))}
\right]. 
\end{split}
\end{equation}
Finally, applying integration by parts then using the trace inequality \cite[Theorem III.2.19.]{Boyer} and Young's inequality we get
\begin{align}\label{g-estimate}
\Bigg \lvert \displaystyle \int_0^t
\displaystyle \int_{\Gamma_c(0)} \bm g \cdot \bm {\tilde \gamma}_n \ d\tilde{\Gamma} \ ds \Bigg \lvert
\leq
C
\delta T \norm \mathsmaller \int_0^{\bigcdot} \bm {\tilde \gamma}_n(s) ds \norm^2_{L^\infty(H^1(\OS))}
+
C_\delta \norm \bm g \norm^2_{H^1(L^2(\Gamma_c(0)))}.
\end{align}
Where we have used H\"older's \cite[Proposition II.2.18]{Boyer} inequality with the fact that $\bm g(.,0)=0$ which gives
\begin{align*}
\norm \bm g(.) \norm^2_{L^\infty([0,T])}
\leq
T \norm \bm g(.) \norm^2_{H^1([0,T])} \quad \textup{on} \ \Gamma_c(0)
\end{align*}

In order to deal with $\mathlarger \int_{\Omega_0} \lvert \bm {\tilde \gamma}_n(0) \lvert^2 d\bm {\tilde x}$, we use Lemma\cite[Lemma 2.2]{Du} which yields
\begin{align}\label{projection-initial}
\norm \bm {\tilde \gamma}_n(0) \norm^2_{L^2(\Omega(0))}
=
\norm \pi_n \bm {\tilde \gamma}_0 \norm^2_{L^2(\Omega_0)}
\leq \norm \bm {\tilde \gamma}_0 \norm^2_{L^2(\Omega_0)}
=
\norm \bm v_0 \norm^2_{L^2(\OF)}
+
\norm \bm \xi_1 \norm^2_{L^2(\OS)}.
\end{align}
%
%
%
%
Combining \eqref{fluid-est1}-\eqref{g-estimate} and using \eqref{projection-initial} we obtain
\begin{equation}
\begin{split}
&\Big(
\mu C_k -\mu CT^{\kappa} M 
\Big) 
\norm \bm {\tilde \gamma}_n \norm^2_{L^2(H^1(\OF))}
+\dfrac{\rho_f}{2}(1-CT^\kappa M) 
\norm \bm {\tilde \gamma}_n \norm^2_{L^\infty(L^2(\OF))}
+\Big(
\dfrac{\rho_s}{2}\big(1-CT^\kappa M \big)-CT^\kappa M
\Big)
\norm \bm {\tilde \gamma}_n \norm^2_{L^\infty(L^2(\OS))}
\\
&+
\Big(
\mu_s C_k -CT^{\kappa}M -C\delta T) \Big) \norm \mathsmaller \int_0^{\bigcdot} \bm {\tilde \gamma}_n(s) ds \norm^2_{L^\infty(H^1(\OS))}
\leq
C
\Big[
\dfrac{\rho_f}{2} \norm \bm v_0 \norm^2_{L^2(\OF)}+
\dfrac{\rho_s}{2} \norm \bm \xi_1 \norm^2_{L^2(\OS)}
\Big]
+
C_\delta \lvert \lvert \bm g \lvert \lvert^2_{H^1(L^2(\Gamma_c(0)))}.
\end{split}
\end{equation}
Remark that, the constants $\mu_s,\lambda_s$ and $\mu$ are given as large values by the constitutive laws of the structure and the fluid. 
Moreover, $\delta$ is a negligible positive real number, hence norms that are factored by the term $\delta$ are being absorbed by larger terms. Finally, we take $T$ small with respect to $M$ and the initial values, that is, the factor  $CT^\kappa M$ is negligible. These assumptions lead to the following estimate
\begin{equation}\label{Estimate-Auxiliary-1}
\begin{split}
& \norm \bm {\tilde \gamma}_n \norm^2_{L^2(H^1(\OF))}
+
\norm \bm {\tilde \gamma}_n \norm^2_{L^\infty(L^2(\Omega_0))}
+
\norm \mathsmaller \int_0^{\bigcdot} \bm {\tilde \gamma}_n(s) ds \norm^2_{L^\infty(H^1(\OS))}
\leq
C
\Big[
\dfrac{\rho_f}{2} \norm \bm v_0 \norm^2_{L^2}+
\dfrac{\rho_s}{2} \norm \bm \xi_1 \norm^2_{L^2}+
\lvert \lvert \bm g \lvert \lvert^2_{H^1(L^2(\Gamma_c(0)))}
\Big].
\end{split}
\end{equation}
\subsubsection*{Step 2: Estimates on $\partial_t \bm {\tilde \gamma}_n$}
The next step is to derive some estimates on $\partial_t \bm {\tilde \gamma}_n$.
Consider a function $\bm {\tilde \eta}$ in $\widetilde{\mathcal{W}}$ such that $\norm \bm {\tilde \eta} \norm_{L^2(H^1(\Omega_0))} \leq 1$. The function $\bm {\tilde \eta}$ can be written as 
\begin{align*}
\bm {\tilde \eta} = \pi_n \bm {\tilde \eta} + (\bm {\tilde \eta}-\pi_n \bm {\tilde \eta}).
\end{align*} 
 where $\pi_n$ is the projection from $L^2(\Omega_0)$ into $\widetilde{\mathcal{W}}_n$.
 Notice that, as we have $\partial_t \bm {\tilde \gamma}_n \in \widetilde{\mathcal{W}}_n$ then
 \begin{align*}
 [[\partial_t \bm {\tilde \gamma}_n(t), \bm {\tilde \eta} ]]
 =
 [[\partial_t \bm {\tilde \gamma}_n (t), \pi_n \bm {\tilde \eta}]]
 +
 [[\partial_t \bm {\tilde \gamma}_n (t), \bm {\tilde \eta}- \pi_n \bm {\tilde \eta}]]
 =
 [[\partial_t \bm {\tilde \gamma}_n (t), \pi_n \bm {\tilde \eta}]].
 \end{align*}
Set $\bm {\tilde \eta}_n = \pi_n \bm {\tilde \eta}$ in \eqref{Galerkin-coupling-formulation}. By integrating over $(0,t)$ we obtain
\begin{equation}\label{Galerkin3}
\begin{cases}
\rho_f \displaystyle \int_0^t
\displaystyle \int_{\OF} \tdet(\bm \nabla \tca)  \partial_s \bm {\tilde \gamma}_n \cdot \pi_n \bm {\tilde \eta} \ d\bm {\tilde x} \ ds 
+ \rho_s \displaystyle \int_0^t  
\displaystyle \int_{\OS}
\tdet(\bm \nabla \tp)
\partial_s \bm {\tilde \gamma}_n \cdot \pi_n \bm {\tilde \eta} \ d\bm {\tilde x} \ ds 
+\displaystyle \int_0^t 
\displaystyle \int_{\OF} 
\bm {\breve \sigma}^0_f(\bm {\tilde \gamma}_n) 
: 
\bm \nabla \pi_n \bm {\tilde \eta} \ d\bm {\tilde x} \ ds 
\vspace{2mm} \\
+\displaystyle\sum\limits_{i,\al,j,\be=1}^3
\displaystyle \int_0^t \displaystyle \int_{\OS} \tbiajb \partial_\be (\mathsmaller  \int_0^s  \tilde{\gamma}_n(\tau) d\tau)_j \ \partial_{\al} (\pi_n \tilde{\eta})_i \ d\bm {\tilde x} \ ds
\vspace{2mm} \\
+ \displaystyle \sum\limits_{i,\al,j,\be=1}^3
\displaystyle \int_0^t \displaystyle \int_{\OS} \partial_\al \tbiajb \partial_\be ( \mathsmaller \int_0^s  \tilde{\gamma}_n(\tau) d\tau)_j (\pi_n \tilde{\eta})_i \ d\bm {\tilde x} \ ds
 =
\displaystyle \int_0^t \displaystyle \int_{\Gamma_c(0)} 
\bm g \cdot \pi_n \bm {\tilde \eta} \ d\tilde{\Gamma} \ ds.


\end{cases}
\end{equation}
This is equivalent to say,
\begin{align*}
(1-CT^\kappa M) 
\displaystyle \int_0^t[[\partial_s \bm {\tilde \gamma}_n(s), \pi_n \bm {\tilde \eta}(s)]] \ ds
=
&-\displaystyle \int_0^t \displaystyle \int_{\OF} \bm {\breve \sigma}^0_f(\bm {\tilde \gamma}_n) : \bm \nabla \pi_n \bm {\tilde \eta} \ d\bm {\tilde x} \ ds 
-\displaystyle\sum\limits_{i,\al,j,\be=1}^3
\displaystyle \int_0^t \displaystyle \int_{\OS} \tbiajb \partial_\be ( \mathsmaller \int_0^s  \tilde{\gamma}_n(\tau) d\tau)_j \ \partial_{\al} (\pi_n \tilde{\eta})_i \ d\bm {\tilde x} \ ds
\vspace{1mm} \\
&- \displaystyle \sum\limits_{i,\al,j,\be=1}^3
\displaystyle \int_0^t \displaystyle \int_{\OS} \partial_\al \tbiajb \partial_\be (\mathsmaller \int_0^s  \tilde{\gamma}_n(\tau) d\tau)_j (\pi_n \tilde{\eta})_i \ d\bm {\tilde x} \ ds
+
\displaystyle \int_0^t \displaystyle \int_{\Gamma_c(0)} \bm g \cdot \pi_n \bm {\tilde \eta} \ d\tilde{\Gamma} \ ds.
\end{align*}
Bounding the terms of the right hand side of the above equality yields
\begin{equation}\label{EQN}
\begin{split}
(1-CT^\kappa M)
\displaystyle
\int_0^t
[[ \partial_s \bm {\tilde \gamma}_n(s), \pi_n \bm  {\tilde \eta}(s)]] \ ds
\leq&
\displaystyle \int_0^t 
\norm \bm {\breve \sigma}^0_f(\bm {\tilde \gamma}_n) \norm_{L^2(\OF)}
\norm \bm \nabla \pi_n \bm {\tilde \eta} \norm_{L^2(\OF)} \ ds
 \vspace{2mm} \\
&+
\mathlarger  \sum\limits_{i,\al,j,\be=1}^3
\displaystyle \int_0^t 
\norm \tbiajb \norm_{L^\infty(\OS)} 
\norm \partial_\be (\mathsmaller \int_0^s \tilde{\gamma}_n(\tau) d\tau)_j \norm_{L^2(\OS)}
\norm \partial_{\al} (\pi_n \tilde{\eta})_i \norm_{L^2(\OS)} \ ds
\vspace{2mm} \\ 
&+
 \displaystyle \sum\limits_{i,\al,j,\be=1}^3
\displaystyle \int_0^t
\norm \partial_\al \tbiajb \norm_{L^\infty(\OS)}
\norm \partial_\be ( \mathsmaller \int_0^s  \tilde{\gamma}_n(\tau) d\tau)_j
\norm_{L^2(\OS)}
\norm (\pi_n \tilde{\eta})_i \norm_{L^2(\OS)} \ ds
\vspace{2mm} \\
&+
\norm \bm g \norm_{H^1(L^2(\Gamma_c(0)))} 
\norm \pi_n \bm {\tilde \eta} \norm_{L^2(H^1(\OS))}.
\end{split}
\end{equation}
Using H\"older's inequality with the Sobolev embeddings ($H^2 \subset L^\infty$), the right hand side of \eqref{EQN} is bounded above by
\begin{equation}
\begin{split}
\Big[
C
\norm  \bm{\tilde \gamma}_n \norm_{L^2(H^1(\OF))}
+
CT^\kappa M
\norm \mathsmaller \int_0^{\bigcdot} \bm {\tilde \gamma}_n(s) ds \norm_{L^\infty(H^1(\OS))}
+
\norm \bm g \norm_{H^1(L^2(\Gamma_c(0)))} 
\Big]
\norm \pi_n \bm {\tilde \eta} \norm_{L^2(H^1(\Omega_0))}.
\end{split}
\end{equation}
Then using the previous estimate \eqref{Estimate-Auxiliary-1} we get
\begin{align*}
\displaystyle
\int_0^t
[[\partial_s \bm {\tilde \gamma}_n(s),\pi_n \bm {\tilde \eta}(s)]] \ ds
\leq 
C
\Big[
\dfrac{\rho_f}{2}
\norm \bm v_0 \norm_{L^2}
+
\dfrac{\rho_s}{2}
\norm \bm \xi_1 \norm_{L^2}
+
C \norm \bm g \norm_{H^1(L^2(\Gamma_c(0)))}
\Big]
\norm \pi_n \bm {\tilde \eta} \norm_{L^2(H^1(\Omega_0))}.
\end{align*}
Using the fact that $\norm \pi_n \bm {\tilde \eta} \norm_{L^2(H^1)} \leq
\norm \bm {\tilde \eta} \norm_{L^2(H^1(\Omega_0))} \leq 1$ we get 
\begin{align}\label{Estimate-Partial_t-Aux}
\norm \partial_t \bm {\tilde \gamma}_n \norm^2_{L^2(L^2)}
\leq 
C
\Big[
\dfrac{\rho_f}{2}
\norm \bm v_0 \norm^2_{L^2(\OF)}
+
\dfrac{\rho_s}{2}
\norm \bm \xi_1 \norm^2_{L^2(\OS)}
+
\norm \bm g \norm^2_{H^1(L^2(\Gamma_c(0)))}
\Big].
\end{align}
%
%
%
%
%
From the estimates \eqref{Estimate-Auxiliary-1} and \eqref{Estimate-Partial_t-Aux} we may extract a subsequence of $\{ \bm {\tilde \gamma}_n \}_n$ which we also denote by $\{ \bm {\tilde \gamma}_n \}_n$ such that
\[
\bm {\tilde \gamma}_n \overset{*}{\rightharpoonup} \bm {\tilde \gamma}  \ \textrm{in} \ L^\infty(0,T;L^2(\Omega_0)),
\qquad
\bm {\tilde \gamma}_n \overset{}{\rightharpoonup} \bm {\tilde \gamma}  \ \textrm{in} \ L^2(0,T;L^2(\Omega_0)),
\]
\[
\partial_t \bm {\tilde \gamma}_n \overset{}{\rightharpoonup} \partial_t \bm {\tilde \gamma}  \ \textrm{in} \ L^2(0,T;L^2(\Omega_0)),
\]
\[\qquad
\bm {\tilde \gamma}_n|_{\OF} \overset{*}{\rightharpoonup} \bm {\tilde \gamma}|_{\OF}  \ \textrm{in} \ L^\infty(0,T;H^1(\OF)),
\qquad
\bm {\tilde \gamma}_n|_{\OF} \overset{}{\rightharpoonup} \bm {\tilde \gamma}|_{\OF}  \ \textrm{in} \ L^2(0,T;H^1(\OF)),
\]
and
\[
\displaystyle \int_0^t \bm {\tilde \gamma}_n(s)\Big|_{\OS} ds  \overset{*}{\rightharpoonup} \int_0^t \bm {\tilde \gamma}(s)\Big|_{\OS} ds  \ \textrm{in} \ L^\infty(0,T;H^1(\OS)).
\]
By Aubin-Lion-Simon Theorem 
\cite[Theorem II.5.16]{Boyer} we get
\[
 \bm {\tilde \gamma}_n 
\rightarrow 
\bm {\tilde \gamma}
 \in \mathcal{C}^0 \big([0,T];L^2(\Omega_0)\big) .
 \]





\subsubsection*{Existence of the Weak Solution}

Now passing to the limit as $n \rightarrow \infty$ in \eqref{Estimate-Auxiliary-1} and \eqref{Estimate-Partial_t-Aux} gives us the estimates on $\bm {\tilde \gamma}$.
To show that $\bm {\tilde \gamma}$ satisfies \eqref{final-eqn-weak-1} we proceed as follows. We fix an integer $N$ and choose a function
$\bm {\tilde \eta} \in \mathcal{C}^1([0,T],\widetilde{\mathcal{W}})$ of the form

\begin{align}\label{form-of-function}
\bm {\tilde \eta} = \sum_{l=1}^N d_l(t) \bm \psi_l(\bm {\tilde x}).
\end{align}
For $n>N$, we integrate \eqref{Galerkin-coupling-formulation} with respect to $t$ to get
\begin{equation}\label{*}
\begin{cases}
\rho_f
\displaystyle \int_0^T
\displaystyle \int_{\OF} 
\tdet(\bm \nabla \tca) \ \partial_t \bm {\tilde \gamma}_n \cdot \bm {\tilde \eta} \ d\bm {\tilde x} \ dt
+
\displaystyle \int_0^T
\displaystyle \int_{\OF} \bm {\breve \sigma}^0_f(\bm {\tilde \gamma}_n) : \bm \nabla \bm {\tilde \eta} \ d\bm {\tilde x} \ dt
+ 
\rho_s
\displaystyle \int_0^T
\displaystyle \int_{\OS}
\tdet(\bm \nabla \tp)
\partial_t \bm {\tilde \gamma}_n \cdot \bm {\tilde \eta} \ d\bm {\tilde x} \ dt 
\vspace{1mm} \\
+
\displaystyle\sum\limits_{i,\al,j,\be=1}^3
\displaystyle \int_0^T
\displaystyle \int_{\OS} \tbiajb \partial_\be ( \mathsmaller \int_0^t \gamma_n(s) ds)_j \ \partial_\al \tilde {\eta}_i \ d\bm {\tilde x} \ dt
\vspace{1mm} \\
+
\displaystyle \sum\limits_{i,\al,j,\be=1}^3
\displaystyle \int_0^T
\displaystyle \int_{\OS} \partial_\al \tbiajb \partial_\be (\mathsmaller \int_0^t  \tilde{\gamma}_n(s) ds)_j \ \tilde{\eta}_i \ d\bm {\tilde x} \ dt 
=
\displaystyle \int_0^T
\displaystyle \int_{\Gamma_c(0)} \bm g \cdot \bm {\tilde \eta} \ d\tilde{\Gamma} \ dt.

\end{cases}
\end{equation}
By passing to the limit as $n$ goes to infinity we get
\begin{equation}\label{**}
\begin{cases}
\rho_f
\displaystyle \int_0^T
\displaystyle \int_{\OF} 
\tdet(\bm \nabla \tca)\partial_t \bm {\tilde \gamma} \cdot \bm {\tilde \eta} \ d\bm {\tilde x} \ dt
+
\displaystyle \int_0^T
\displaystyle \int_{\OF} \bm {\breve \sigma}^0_f(\bm {\tilde \gamma}) : \bm \nabla \bm {\tilde \eta} \ 
d\bm {\tilde x} \ dt
+ 
\rho_s
\displaystyle \int_0^T
\displaystyle \int_{\OS} \tdet(\bm \nabla \tp) \partial_t \bm {\tilde \gamma} \cdot \bm {\tilde \eta} \ d\bm {\tilde x} \ dt 
\vspace{1mm} \\
+
\displaystyle\sum\limits_{i,\al,j,\be=1}^3
\displaystyle \int_0^T
\displaystyle \int_{\OS} \tbiajb \partial_\be 
(\mathsmaller \int_0^t  \tilde{\gamma}(s) ds)_j \ \partial_\al \tilde{\eta}_i \ d\bm {\tilde x} \ dt
\vspace*{1mm} \\
+
\displaystyle \sum\limits_{i,\al,j,\be=1}^3
\displaystyle \int_0^T
\displaystyle \int_{\OS} \partial_\al \tbiajb \partial_\be (\mathsmaller \int_0^t  \tilde{\gamma}(s) ds)_j \ \tilde{\eta}_i \ d\bm {\tilde x} \ dt 
=
\displaystyle \int_0^T
\displaystyle \int_{\Gamma_c(0)} \bm g \cdot \bm {\tilde \eta} \ d\tilde{\Gamma} \ dt,

\end{cases}
\end{equation}
holds true for all $\bm {\tilde \eta} \in L^2([0,T],\widetilde{\mathcal{W}})$ due to the fact that the space spanned by the functions of the form \eqref{form-of-function} is dense in $L^2([0,T],\widetilde{\mathcal{W}})$. Hence, \eqref{**} implies \eqref{final-eqn-weak-1}.
\\
To show that the initial conditions are satisfied we will consider $\bm {\tilde \eta} \in \mathcal{C}^1([0,t],\widetilde{\mathcal{W}})$ in \eqref{**} and integrate by parts to get
\begin{equation}\label{***}
\begin{cases}
\rho_f
\displaystyle \int_0^T
\displaystyle \int_{\OF} \bm {\tilde \gamma} \cdot \partial_t  (\bm {\tilde \eta} \ \tdet(\bm \nabla \tca)) \ d\bm {\tilde x} \ dt
-
\displaystyle \int_0^T
\displaystyle \int_{\OF} \bm {\breve \sigma}^0_f(\bm {\tilde \gamma}) :\bm \nabla \bm {\tilde \eta} \ 
d\bm {\tilde x}  \ dt
%
+
\rho_s
\displaystyle \int_0^T
\displaystyle \int_{\OS} \bm {\tilde \gamma} \cdot \partial_t (\bm {\tilde \eta} \ \tdet(\bm \nabla \tp)) \ d\bm {\tilde x} \ dt

\vspace{1mm} \\
-
\displaystyle\sum\limits_{i,\al,j,\be=1}^3
\displaystyle \int_0^T
 \displaystyle \int_{\OS} \tbiajb \partial_\be
(\mathsmaller \int_0^t  \tilde{\gamma}(s) ds)_j
\ \partial_\al \tilde{\eta}_i \ d\bm {\tilde x} \ dt


-
\displaystyle \sum\limits_{i,\al,j,\be=1}^3
\displaystyle \int_0^T
 \displaystyle \int_{\OS} \partial_\al \tbiajb \partial_\be
(\mathsmaller \int_0^t  \tilde{\gamma}(s) ds)_j
 \ \tilde{\eta}_i \ d\bm {\tilde x} \ dt

\vspace*{1mm} \\

=
-
\displaystyle \int_0^T
\displaystyle \int_{\Gamma_c(0)}
\bm  g \cdot \bm {\tilde \eta} \ d\tilde{\Gamma} \ dt
-
\rho_f \displaystyle \int_{\OF} \bm \vv(0) \cdot \bm {\tilde \eta}(0) \ d\bm {\tilde x}
-
\rho_s \displaystyle \int_{\OS} \partial_t \bm  \xxi(0) \cdot \bm {\tilde \eta}(0) \ d\bm {\tilde x}.


\end{cases}
\end{equation}
%
%
On the other hand, integrating by parts in time Equation \eqref{*} and passing to the limit we get
\begin{equation}\label{****}
\begin{cases}

\rho_f
\displaystyle \int_0^T
\displaystyle \int_{\OF}  \bm {\tilde \gamma} \cdot \partial_t (\bm {\tilde \eta} \ \tdet(\bm \nabla \tca)) d\bm {\tilde x} \ dt


-\displaystyle \int_0^T
\displaystyle \int_{\OF} \bm {\breve \sigma}^0_f(\bm {\tilde \gamma}) : \bm \nabla \bm {\tilde \eta} \ 
d\bm {\tilde x} \ dt

+ \rho_s
\displaystyle \int_0^T
\displaystyle \int_{\OS} \bm {\tilde \gamma} \cdot \partial_t (\bm {\tilde \eta} \ \tdet(\bm \nabla \tp)) \ d\bm {\tilde x} \ dt 



\vspace{1mm} \\
-
\displaystyle\sum\limits_{i,\al,j,\be=1}^3
\displaystyle \int_0^T
 \displaystyle \int_{\OS} \tbiajb \partial_\be (\mathsmaller \int_0^t  \tilde{\gamma}(s) ds)_j
  \ \partial_\al \eta_i \ d\bm {\tilde x} \ dt


-
\displaystyle \sum\limits_{i,\al,j,\be=1}^3
\displaystyle \int_0^T
 \displaystyle \int_{\OS} \partial_\al \tbiajb \partial_\be (\mathsmaller \int_0^t  \tilde{\gamma}(s) ds)_j \ \tilde{\eta}_i \ d\bm {\tilde x} \ dt

\vspace*{1mm} \\ =
-
\displaystyle \int_0^T
\displaystyle \int_{\Gamma_c(0)} \bm  g \cdot \bm {\tilde \eta} \ d\tilde{\Gamma} \ dt
-
\rho_f \displaystyle \int_{\OF} \bm v_0 \cdot \bm {\tilde \eta}(0) \ d\bm {\tilde x}
-
\rho_s \displaystyle \int_{\OS} \bm  \xi_1 \cdot \bm {\tilde \eta}(0) \ d\bm {\tilde x}.


\end{cases}
\end{equation}
Comparing \eqref{***} and \eqref{****} yields 
\[
[[\bm {\tilde \gamma}_0,\bm {\tilde \eta}(0)]] = [[\bm {\tilde \gamma}(0),\bm {\tilde \eta}(0)]].
\]
Since $\bm {\tilde \eta}(0) \in \widetilde{\mathcal{W}}$ is arbitrary, then the initial conditions are verified. 
\\
Finally, by passing to the limit in 
\eqref{galerkin-boundary}, we obtain
\eqref{final-eqn-weak-1}$_3$.
This yields the existence of the weak solution $\bm {\tilde \gamma}$ of System \eqref{Auxiliary-pro}.
\subsubsection*{Uniqueness of the Weak Solution}
To prove the uniqueness we assume that $\bm {\tilde \gamma}_1$ and
$\bm {\tilde \gamma}_2$ are two solutions of \eqref{final-eqn-weak-1} associated to $(\bm {\breve v},\bm {\breve \xi})$. Setting $\bm {\tilde \varsigma} =\bm  {\tilde \gamma}_1 - \bm {\tilde \gamma}_2$.
Then for all $\bm {\tilde \eta} \in \mathcal{C}^0(0,T;\widetilde{\mathcal{W}})$, the solution $\bm {\tilde \varsigma}$ satisfies the following variational formulation
\begin{equation}
\begin{cases}

\rho_f \displaystyle \int_{\OF} 
\tdet(\bm \nabla \tca)\partial_t \bm {\tilde \varsigma} \cdot \bm {\tilde \eta} \ d\bm {\tilde x} +
\displaystyle \int_{\OF} \bm {\tilde \sigma}(\bm {\tilde \varsigma}) :\bm \nabla \bm {\tilde \eta} \ 
d\bm {\tilde x}

+\rho_s \displaystyle \int_{\OS} \tdet(\bm \nabla \tp) \partial_t \bm {\tilde \varsigma} \cdot \bm {\tilde \eta} \ d\bm {\tilde x}

\vspace*{1mm} \\
+\displaystyle \sum_{i,\al,j,\be=1}^3 \displaystyle \int_{\OS} \tbiajb 
\partial_\be 
(\mathsmaller \int_0^t {\tilde \varsigma}(s) ds)_j \ \partial_\al \tilde {\eta}_i \ d\bm {\tilde x}

+\displaystyle \sum_{i,\al,j,\be=1}^3 \displaystyle \int_{\OS} \partial_\al \tbiajb \partial_\be (\mathsmaller \int_0^t {\tilde \varsigma}(s) ds)_j \ \tilde {\eta}_i \ d\bm {\tilde x}

=0
\end{cases}
\end{equation}
Taking $\bm {\tilde \eta} = \bm {\tilde \varsigma}$ and integrating over $(0,t)$ we get
\begin{equation}\label{system of uniqueness}
\begin{cases}

\dfrac{\rho_f}{2} \displaystyle \int_{\OF} 
\tdet(\bm \nabla \tca) \lvert \bm {\tilde \varsigma}(t) \lvert^2 \ d\bm {\tilde x}

+ \displaystyle \int_0^t
\displaystyle \int_{\OF} \bm {\tilde \sigma}(\bm {\tilde \varsigma}) :\bm \nabla \bm{\tilde \varsigma} \ d\bm {\tilde x} \ ds

+\dfrac{\rho_s}{2}
\displaystyle \int_{\OS} \tdet(\bm \nabla \tp) \lvert \bm {\tilde \varsigma}(t) \lvert^2 \ d\bm {\tilde x}

\vspace*{1mm} \\

-\dfrac{\rho_f}{2}
\mathlarger \int_0^t
\mathlarger \int_{\OF}
\lvert \bm {\tilde \varsigma} \lvert^2 \partial_s \tdet(\bm \nabla \tca) \ d\bm {\tilde x} \ ds

-\dfrac{\rho_s}{2}
\mathlarger \int_0^t
\mathlarger \int_{\OS}
\lvert \bm {\tilde \varsigma} \lvert^2 \partial_s \tdet(\bm \nabla \tp) \ d\bm {\tilde x} \ ds

\vspace*{1mm} \\

+
\dfrac{1}{2}
\displaystyle \sum_{i,\al,j,\be=1}^3 \displaystyle \int_{\OS}
\tbiajb 
(t) \partial_\be (\mathsmaller \int_0^t {\tilde \varsigma}(s) ds)_j \ \partial_\al {\tilde \varsigma}_i(t)
\ d\bm {\tilde x}



-
\dfrac{1}{2}
\displaystyle \sum_{i,\al,j,\be=1}^3
\displaystyle \int_0^t
\displaystyle \int_{\OS}
\partial_s
\tbiajb 
\partial_\be (\mathsmaller \int_0^s {\tilde \varsigma}(\tau) d\tau)_j 
\ \partial_\al {\tilde \varsigma}_i
\ d\bm {\tilde x} \ ds

\vspace*{1mm} \\
 
+
\displaystyle \sum_{i,\al,j,\be=1}^3
\displaystyle \int_0^t
\displaystyle \int_{\OS}
\partial_\al \tbiajb
\partial_\be (\mathsmaller \int_0^s {\tilde \varsigma}(\tau) d\tau)_j \ \partial_s {\tilde \varsigma}_i \ d\bm {\tilde x} \ ds

=
0.

\end{cases}
\end{equation}
%
%
%
%
%
%
%
%
%
%
Using \eqref{stress-fluid-estimate}-\eqref{structure-estimate2} we get
\begin{align}\label{estimate for uniqueness}
& \norm \bm {\tilde \varsigma} \norm^2_{L^2(H^1(\OF))}
+
\norm \bm  {\tilde \varsigma} \norm^2_{L^\infty(L^2(\Omega_0))}
+
\norm \mathsmaller \int_0^{\bigcdot} \bm {\tilde \varsigma}(s) ds \norm^2_{L^\infty(L^2(\OS))}
+
\norm \mathsmaller \int_0^{\bigcdot} \bm {\tilde \varsigma}(s) ds \norm^2_{L^\infty(H^1(\OS))}
\leq
0, 
\end{align}
which yields that $\bm {\tilde \gamma}_1=\bm {\tilde \gamma}_2$. Therefore, $\bm {\tilde \gamma}$ is a unique solution of \eqref{final-eqn-weak-1}.
In addition, we have 
\begin{align}
\bm {\tilde \gamma}|_{\OF} \in L^\infty(L^2(\OF)) \cap L^2(H^1(\OF)), \;
\bm {\tilde \gamma}|_{\OS} \in L^\infty(L^2(\OS)) 
\; \textrm{and} \;
\int_0^t \bm {\tilde \gamma}(s)|_{\OS} ds \in L^\infty(H^1(\OS)).
\end{align}
Consequently, setting $\bm \vv= \bm {\tilde \gamma}|_{\OF}$ and $\bm \xxi=\bm \xi_0 + \displaystyle \int_0^t \bm {\tilde \gamma}(s)|_{\OS} \ ds$, we obtain the existence and uniqueness of the weak solution $(\bm \vv,\bm \xxi) \in F^T_1 \times S^T_1$ for the System \eqref{Auxiliary-pro}.


\section{Existence of Solution for the Linearized System }
\label{Existence of linear system}
The linear problem is given by the following system
\begin{equation}\label{Linear-System}
\begin{cases}
    \rho_f \tdet(\bm \nabla \tca)\partial_t \bm \vv -\bm \nabla \cdot \bm {\breve \sigma}^0_f(\bm \vv,					\tilde{p}_f) = 0 & \textrm{in} \ \OF \times (0,T) , 
\\    
    \nabla \cdot (\tdet(\bm \nabla \tc)(\bm{\nabla \tc})^{-1} \bm \vv) = 0 & \textrm{in} \ \OF \times (0,T), 					
\\    
   \bm  \vv= \bm v_{\rm in} \circ \tc & \textrm{on} \ \Gamma_{\textrm{in}}(0) \times (0,T),
\\    
    \bm {\breve \sigma}^0_f(\bm \vv,\tilde{p}_f) \ \bm {\tilde n} = 0 & \textrm{on} \ \Gamma_{\textrm{out}}(0) \times (0,T),		
 \\
   	\rho_s \tdet(\bm \nabla \tp) \partial^2_t \tilde{\xi}_i 
   	-\displaystyle \sum_{\al,j,\be=1}^{3} \tbiajb
   	 \partial_{\al \be}^2 \xit_j 
=0 
\quad i=1,2,3,
& 			 
   	 \textrm{in} \ \OS\times (0,T),
\\	
	\bm \xxi = \bm 0  & \textrm{on} \ \Gamma_2(0) \times (0,T),
\\
	\bm \vv = \partial_t \bm \xxi & \textrm{on} \ \Gamma_c(0) \times (0,T) ,					
	\\
\left
[\bm {\breve \sigma}^0_f(\bm \vv,\tilde{p}_f) \bm {\tilde n} \right]_i =
	\displaystyle \sum_{\al,j,\be=1}^3 \bigg(
	\mathlarger \int_0^t 			
	\tbiajb
	 \partial^2_{s \be} \xit_j ds
	\bigg) \tilde{n}_{\al} \quad i=1,2,3,
	 & \textrm{on} \
	 \Gamma_c(0) \times (0,T),
	\\
	\bm \vv(.,0)=\bm v_0 \quad \textrm{and} \quad \tilde{p}_f(.,0)=p_{f_0} & \textrm{in} \ \OF,			
	\\
   	\bm \xxi(.,0)= \bm 0 \quad \textrm{and} \quad  \partial_t \bm \xxi(.,0)=\bm \xi_1  & \textrm{in} \ \OS, 
\end{cases}
\end{equation}
which is nothing but the auxiliary problem \eqref{Auxiliary-pro} when considering
\[
g_i=-\sum\limits_{\al,j,\be=1}^3 \bigg(\displaystyle \int_0^t \partial_s \tbiajb \partial_\be \xit_j \ ds \bigg) \tilde{n}_{\al}, \quad i=1,2,3.
\]




\begin{proposition}\label{Ex-Uni-Linear}

Let $(\bm {\breve v},\bm {\breve \xi}) \in A^T_M$, $\bm v_0 \in H^1(\OF), \ \textrm{and} \ \bm \xi_1 \in H^1(\OS)$ satisfying \eqref{intialcdtns} and \eqref{comp-cdtns}$_1$. For $T$ small with respect to $M$ and the initial conditions, there exists a unique weak solution $(\bm \vv,\bm \xxi) \in F^T_2 \times S^T_2$ of \eqref{Linear-System}. Moreover we get \textit{a piori} estimate on the solution given by
\begin{equation}\label{Second-APriori-estimate}
\norm \bm \vv \norm^2_{F^T_2} + \norm \bm \xxi \norm^2_{S^T_2} 
\leq 
C \norm \bm v_0 \norm^2_{H^1} + C \norm \bm \xi_1\norm^2_{H^1}.
\end{equation}
\end{proposition}
Notice that, increasing the regularity of the initial data by considering $\bm v_0 \in H^1(\OF)$ and $\bm \xi_1 \in H^1(\OS)$ will lead to a more regular solution \cite{Evans, Brezis}. Using the regularity results we achieve a solution $(\bm \vv,\bm \xxi) \in F^T_2 \times S^T_2$.
%
%
%
Now we prove Proposition \ref{Ex-Uni-Linear}.
The proof is based on the fixed point theorem. Indeed, the first step is to find estimates on $\partial_t \bm {\tilde \gamma}$ in $F^T_2 \times S^T_2$ then we prove the existence and uniqueness of a weak solution of \eqref{Linear-System}.
\\
First, we consider System \eqref{Auxiliary-pro} with the function
\[
g_i=\hat{h}_i =-\sum\limits_{\al,j,\be=1}^3 \bigg(\displaystyle \int_0^t \partial_s \tbiajb \partial_\be \hat{\xi}_j \ ds \bigg) \tilde{n}_{\al}, \quad i=1,2,3.
\]
Observe that $\bm {\hat{h}} \in H^1_l(0,T;L^2(\Gamma_c(0))$. Thanks to \eqref{lq-boundedness} and the trace inequality we have 
\begin{equation}\label{hat-h-estimate}
\norm \bm {\hat{h}} \norm_{H^1_l(0,T;L^2(\Gamma_c(0)))}
\leq CT^{\kappa} M \norm \bm {\hat{\xi}} \norm_{S^T_2}.
\end{equation}
Therefore, as $\bm {\hat{\xi}}$ is fixed, by Lemma \ref{Ex-Uni-AuxPro} we get the existence and uniqueness of $(\bm \vv,\bm \xxi) \in F^T_1 \times S^T_1$ satisfying
\begin{equation}\label{First-A Priori estimate-modified}
\norm \bm \vv \norm^2_{F^T_1} + \norm \bm \xxi \norm^2_{S^T_1}
\leq
C
\Big[
\dfrac{\rho_f}{2} \norm \bm  v_0 \norm^2_{L^2}
+
\dfrac{\rho_s}{2} \norm \bm \xi_1 \norm^2_{L^2}
+
T^{\kappa}M \norm \bm {\hat{\xi}} \norm_{S^T_2}
\Big].
\end{equation}
To prove that the solution $(\bm \vv,\bm \xxi)$ is in the space $F^T_2 \times S^T_2$ we use the fixed point theorem. To this end we introduce the map $\Psi_0$ from $S^T_2$ to $S^T_2$ defined as 
\[
\Psi_0: \bm {\hat{\xi}} \longmapsto \bm \xxi.\]
As we mentioned previously, we ensure the existence of a weak solution $(\bm \vv,\bm \xxi) \in F^T_2 \times S^T_2$. In order to prove its uniqueness it is sufficient to prove that $\Psi_0$ is a contraction on $S_2^T$. This is achieved by deriving some a priori estimates on $\partial_t \bm {\tilde \gamma}$.

		


\subsection{Estimates on $\partial_t \tilde{v}$ and $\partial^2_t \tilde{\xi}$}

We proceed to derive \textit{a priori} estimates on $\partial_t \bm {\tilde \gamma}$. Differentiating in time the weak formulation \eqref{final-eqn-weak-1}. Taking $\bm {\tilde \eta}= \partial_t \bm {\tilde \gamma}$ yields
\begin{equation}\label{diff-in-t-weakeqn2}
\begin{cases}
\dfrac{\rho_f}{2} 
\displaystyle \int_{\OF} 
\tdet(\bm \nabla \tca)(t)
\lvert \partial_t \bm {\tilde \gamma} (t) \lvert^2 \ d\bm {\tilde x} +

\displaystyle \int_0^t \displaystyle \int_{\OF} \partial_s \bm {\breve \sigma}^0_f(\bm {\tilde \gamma}) : \partial_s \bm \nabla \bm {\tilde \gamma} \ d\bm {\tilde x} \ ds



\vspace{1mm}\\

+
\dfrac{\rho_s}{2}
\mathlarger \int_{\OS}
\tdet(\bm \nabla \tp)(t) \lvert \partial_t \bm {\tilde \gamma}(t) \lvert^2 \ d\bm {\tilde x}

+
\dfrac{\rho_f}{2}
\mathlarger \int_0^t
\mathlarger \int_{\OF}
\partial_s \tdet(\bm \nabla \tca) \lvert \partial_s \bm {\tilde \gamma} \lvert^2 \ d\bm {\tilde x} \ ds

\vspace{1mm}\\
+
\dfrac{\rho_s}{2}
\displaystyle \int_0^t 
\displaystyle \int_{\OS} 
\partial_s \tdet(\bm \nabla \tp)
\lvert \partial_s \bm {\tilde \gamma} \lvert^2 \ d\bm {\tilde x} \ ds


+\dfrac{1}{2} \displaystyle \sum_{i,\al,j,\be=1}^3  \displaystyle \int_{\OS} 
\Big[
\tbiajb \partial_{\be} \tilde{\gamma}_j
\partial_{\al} \tilde{\gamma}_i 
\Big](t) \ d\bm {\tilde x}
\vspace{1mm} \\
-\dfrac{1}{2}
\displaystyle \sum_{i,\al,j,\be=1}^3 \displaystyle \int_0^t
 \displaystyle \int_{\OS} \partial_s\tbiajb \partial^2_{s\be} (\mathsmaller  \int_0^s \tilde{\gamma}(\tau) d\tau)_j
 \
\partial^2_{s\al} ( \mathsmaller \int_0^s \tilde{\gamma}(\tau) d\tau)_i \ d\bm {\tilde x} \ ds


\vspace{1mm}\\
+\displaystyle \sum_{i,\al,j,\be=1}^3 \displaystyle \int_0^t
\displaystyle \int_{\OS} 
\partial_\al \tbiajb
\partial^2_{s\be} ( \mathsmaller \int_0^s \tilde{\gamma}(\tau) d\tau)_j \ 
\partial^2_s ( \mathsmaller \int_0^s \tilde{\gamma}(\tau) d\tau)_i \ d\bm {\tilde x} \ ds


\vspace{1mm}\\
-\displaystyle \sum_{i,\al,j,\be=1}^3 \displaystyle \int_0^t  
\displaystyle \int_{\OS}
\partial_s\tbiajb 
\partial^2_{\al \be} (\mathsmaller  \int_0^s \tilde{\gamma}(\tau) d\tau)_j
\
\partial^2_{s} (\mathsmaller  \int_0^s \tilde{\gamma}(\tau) d\tau)_i \ d\bm {\tilde x} \ ds


\vspace{1mm}\\
=
\displaystyle \int_0^t  \displaystyle \int_{\Gamma_c(0)} \partial_s \bm {\hat{h}} \cdot \partial^2_s (\mathsmaller  \int_0^s \bm {\tilde \gamma}(\tau) d\tau) \ d\tilde{\Gamma} \ ds
+
\displaystyle \int_{\OS}
\Big[ \mu_s \lvert \bm \epsilon(\bm {\tilde \gamma}(0))  \lvert^2 + \dfrac{\lambda_s}{2} \lvert \nabla \cdot ( \bm {\tilde \gamma}(0)) \lvert^2 \Big](t) \  d\bm {\tilde x}


\vspace{1mm}\\
+\displaystyle \sum_{i,\al,j,\be=1}^3 \displaystyle \int_0^t 
\displaystyle \int_{\Gamma_c(0)}
\partial_s\tbiajb
\partial_{\be} (\mathsmaller  \int_0^s \tilde{\gamma}(\tau) d\tau)_j  \partial^2_{s} (\mathsmaller  \int_0^s \tilde{\gamma}(\tau) d\tau)_i \ n_\al \ d\tilde{\Gamma} \ ds


\vspace{1mm} \\
+
\dfrac{\rho_f}{2}
\mathlarger \int_{\OF}
\lvert \partial_t \bm {\tilde \gamma}(0) \lvert^2 \ d\bm {\tilde x}

+
\dfrac{\rho_s}{2}
\mathlarger \int_{\OS}
\lvert \partial_t \bm {\tilde \gamma}(0) \lvert^2 \ d\bm {\tilde x}.

\end{cases}
\end{equation}
As for the stress term in \eqref{diff-in-t-weakeqn2} we have
\begin{equation}
\label{fstress-estimate-partialt}
\begin{split}
\displaystyle \int_0^t 
\displaystyle \int_{\OF}
 & \partial_s \bm {\breve \sigma}^0_f(\bm {\tilde \gamma}) :
 \partial_s \bm \nabla \bm {\tilde \gamma} \ d\bm {\tilde x} \ ds
=
 A_1+A_2+A_3.
 \end{split}
\end{equation}
\begin{align}\label{A_1}
 A_1  
=&
  \dfrac{\mu}{2} \displaystyle \int_0^t \displaystyle \int_{\OF}
 \tdet(\bm \nabla \tca)
 \Big \lvert
 \partial_s \nga (\nc)^{-1}+ \itnc \partial_s (\nga)^t \Big \lvert^2 \ d\bm {\tilde x} \ ds
\nonumber \\
& 
 \geq
\mu (C_k-CT^{\kappa}M) \lvert \lvert \bm {\tilde \gamma}_n \lvert \lvert^2_{H^1(H^1(\OS))}.
\end{align}
For $A_2$ we use Young's inequality with H\"older's inequality to obtain
\begin{equation}\label{A2}
\begin{split}
|A_2|
=&
\bigg \lvert
\mu \displaystyle \int_0^t \displaystyle \int_{\OF}
 \Big( \nga \partial_s \inc+ \partial_s \itnc (\nga)^t  \Big) \cch : \partial_s \nga 
 \ d\bm {\tilde x} \ ds 
 \bigg \lvert
\\
& \leq
CT^{1/2}M 
\Big[
C_\delta \norm \bm {\tilde \gamma} \norm^2_{L^\infty(H^2(\OF))}
+
\delta \norm \bm {\tilde \gamma} \norm^2_{H^1(H^1(\OF))} 
\Big].
\end{split}
\end{equation}
Similarly, for $A_3$ we have
\begin{equation}\label{A_3}
\begin{split}
|A_3|
=&
\bigg \lvert
\mu \displaystyle \int_0^t \displaystyle \int_{\OF} \Big( \nga (\nc)^{-1}+ \itnc (\nga)^t \Big) \partial_s \cch : \partial_s \nga \ d\bm {\tilde x} \ ds
\bigg \lvert
\\
&
\leq
CT^{1/2} M
\Big[
C_\delta \norm \bm {\tilde \gamma} \norm^2_{L^\infty(H^2(\OF))}
+
\delta \norm \bm {\tilde \gamma} \norm^2_{H^1(H^1(\OF))} 
\Big].
\end{split}
\end{equation}
Therefore, the summation of Equations \eqref{A2} and \eqref{A_3} is bounded above by
\begin{align}\label{A2+A3}
CT^{1/2} M 
\Big[
C_\delta \norm \bm {\tilde \gamma} \norm^2_{L^\infty(H^2)}
+
\delta \norm \bm {\tilde \gamma} \norm^2_{H^1(H^1(\OF))} 
\Big].
\end{align}
As for the integrals over $\OS$, first we have
%
%
\begin{equation}\label{estimate-partial-1}
\begin{split}
\dfrac{\rho_s}{2}
\mathlarger \int_{\OS}
\tdet(\bm \nabla \tp)(t) \lvert \partial_t \bm {\tilde \gamma}(t) &\lvert^2 \ d\bm {\tilde x}
+
\dfrac{\rho_s}{2}
\displaystyle \int_0^t 
\displaystyle \int_{\OS} 
\partial_s \tdet(\bm \nabla \tp)
\lvert \partial_s \bm {\tilde \gamma} \lvert^2 \ d\bm {\tilde x} \ ds
\\
&\geq
\dfrac{\rho_s}{2}(1-CT^\kappa M) \norm \partial_t \bm {\tilde \gamma} \norm^2_{L^\infty(L^2(\OS))}.
\end{split}
\end{equation}
On the contrary, using \eqref{coercivity} and \eqref{coercivity1} with Korn's inequality gives
\begin{equation}\label{estimate-partial-2}
\begin{split}
\dfrac{1}{2} 
&
\displaystyle \sum_{i,\al,j,\be=1}^3  \displaystyle \int_{\OS} 
\Big[ \tbiajb \partial_{\be} \gamma_j \partial_{\al} \gamma_i \Big](t) \ d\bm {\tilde x}
\\
&\geq 
\mu_s C_k \norm \bm {\tilde \gamma} \norm^2_{L^\infty(H^1(\OS))}
+
\dfrac{\mathsf{C}+\lambda_s}{2} \norm \nabla \cdot \bm {\tilde \gamma} \norm^2_{L^\infty(L^2(\OS))}
- CT^\kappa M \norm \bm {\tilde \gamma} \norm^2_{L^\infty(H^1(\OS))}.
\end{split}
\end{equation}
On the other hand, using \eqref{lq-boundedness} and \eqref{lq1-boundedness} we have
%
%
\begin{equation}\label{total-sum-3-estimates}
\begin{split}
\Bigg \lvert
\displaystyle \sum_{i,\al,j,\be=1}^3
&
\Bigg[
-\dfrac{1}{2}
\displaystyle \int_0^t
\displaystyle \int_{\OS}
 \partial_s\tbiajb 
\partial^2_{s\be} ( \mathsmaller \int_0^s \tilde{\gamma}(\tau) d\tau)_j
 \
\partial^2_{s\al} ( \mathsmaller \int_0^s \tilde{\gamma}(\tau) d\tau)_i \ d\bm {\tilde x} \ ds 
\\ 
&+
\displaystyle \int_0^t
\displaystyle \int_{\OS} 
\partial_\al \tbiajb
\partial^2_{s\be} ( \mathsmaller \int_0^s \tilde{\gamma}(\tau) d\tau)_j \ 
\partial^2_s ( \mathsmaller \int_0^s \tilde{\gamma}(\tau) d\tau)_i \ d\bm {\tilde x} \ ds
\vspace{3mm} \nonumber \\
&-\displaystyle \int_0^t  
\displaystyle \int_{\OS}
\partial_s\tbiajb 
\partial^2_{\al \be} ( \mathsmaller \int_0^s \tilde{\gamma}(\tau) d\tau)_j
\
\partial^2_{s} ( \mathsmaller \int_0^s \tilde{\gamma}(\tau) d\tau)_i \ d\bm {\tilde x} \ ds
\Bigg ]
\Bigg \lvert
\\
\leq &
CT^\kappa M 
\Big[
C_\delta
\norm \mathsmaller \int_0^{\bigcdot} \bm {\tilde \gamma}(s) ds 
\norm^2_{L^\infty(H^2(\OS))}
+
\delta
\norm \partial_t \bm {\tilde \gamma} \norm^2_{L^\infty(L^2(\OS))}
+
\norm \bm {\tilde \gamma} \norm^2_{L^\infty(H^1(\OS))}
\Big].
\end{split}
\end{equation}
For the integrals across the boundary we use \eqref{hat-h-estimate},Young's inequality in addition to the trace inequality to obtain
\begin{equation}\label{partial_h}
\begin{split}
\Bigg \lvert
\displaystyle \int_0^t  \displaystyle \int_{\Gamma_c(0)} \partial_s \bm {\hat{h}} \cdot \partial_s \bm {\tilde \gamma} \ d\tilde{\Gamma} \ ds
\Bigg \lvert
\leq
CT^{\kappa}M 
\big[
C_\delta
\norm \bm {\hat{\xi}} \norm^2_{S^T_2} 
+
\delta
\norm \bm {\tilde \gamma} \norm^2_{H^1(H^1(\OF))}
\big]
,
\end{split}
\end{equation}
and for $i,\al,j,\be=1,2,3,$ we have
\begin{equation}\label{partial-boundary}
\begin{split}
\Bigg \lvert
&
\displaystyle \sum_{i,\al,j,\be=1}^3 
\displaystyle \int_0^t \displaystyle \int_{\Gamma_c(0)}
\partial_s\tbiajb 
\partial_{\be} (\mathsmaller \int_0^s \tilde{\gamma}(\tau) d\tau)_j  \partial_s \tilde{\gamma}_i \ n_\al \ d\tilde{\Gamma} \ ds
\Bigg \lvert
\\
& \leq
CT^{1/2}(M+M^2+M^3+M^4) 
\Big[
C_\delta
\norm \mathsmaller \int_0^{\bigcdot} \bm {\tilde \gamma}(s) ds
\norm^2_{L^\infty(H^2(\OS))}
+
\delta
\norm \bm {\tilde \gamma} \norm^2_{H^1(H^1(\OF))}
\Big]
.
\end{split}
\end{equation}

Therefore, considering the restriction of $\bm {\tilde \gamma}$ on each sub-domain, and for $T$ small with respect to $M$ and the initial conditions, i.e, the factors $CT^\kappa M$ and $CT^{1/2}M^4$ are negligible, and using \eqref{First-A Priori estimate-modified} we get 
\begin{align}\label{Second-A priori-estimate}
\norm \bm \vv \norm^2_{W^{1,\infty}(L^2)}+\norm \bm \vv \norm^2_{H^1(H^1)}
+
\norm \bm \xxi \norm^2_{W^{2,\infty}(L^2)} +\norm \bm \xxi \norm^2_{W^{1,\infty}(H^1)}
\nonumber
\end{align}
\begin{flalign}
\leq
C \norm \bm v_0 \norm^2_{H^1}+ C \norm \bm \xi_1 \norm^2_{H^1}
+CT^{\kappa}M\big (\norm \bm \xxi \norm^2_{S^T_2} + \norm \bm {\hat{\xi}} \norm^2_{S^T_2} \big), &
\end{flalign}




\subsection{Estimates Using Spatial Regularity}\label{Estimates1-spatial}
We have proved that the linear system has a strong solution $(\bm \vv,\bm \xxi) \in F^T_2 \times S^T_2$. Therefore, for all $t \in (0,T)$, the fluid velocity $\bm \vv$ satisfies the following equation
\begin{align*}
\bm \nabla \cdot \bm {\breve \sigma}^0_f (\bm {\tilde v})=\rho_f
\tdet(\bm \nabla \tca) \partial_t \bm  \vv \qquad & \textrm{in} \quad \OF,
\end{align*}
which can be rewritten as
\begin{align*}
\mu \bm \nabla \cdot \big(\nv+(\nv)^t)=\rho_f
\tdet(\bm \nabla \tca) \partial_t \bm \vv + F_{\tilde {v}} \qquad & \textrm{in} \quad \OF,
\end{align*}
with
\begin{align*}
F_{\tilde {v}} = - \mu \nabla \cdot
\bm {f_{\vv}},
\end{align*}
where 
\begin{align*}
\bm {f_{\vv}}=
\Big(\nv \big( \inc-\Id \big) + \big(\itnc -\Id \big) \nv)^t \Big) \cch 
- \big( \nv +(\nv)^t \big ) \big(\cch -\Id \big). 
\end{align*}
Using Lemma \ref{Estimates on Flow} we have
\begin{align*}
\norm F_{\tilde v} \norm_{L^\infty(L^2)} \leq \norm \bm {f_{\vv}} \norm_{L^\infty(H^1)}
\leq 2\mu CT^{\kappa}M \norm \bm  \vv \norm_{L^\infty(H^2)}.
\end{align*}
Hence, we obtain
\begin{align}\label{elliptic-estimate-fluid}
\mu \norm \bm \vv \norm_{L^\infty(H^2)} \leq \rho_f
CT^\kappa M \norm \partial_t \bm  \vv \norm_{L^\infty(L^2)} + 2 \mu CT^{\kappa}M \norm \bm \vv \norm_{L^\infty(H^2)}.
\end{align}
%
%
%
%
%
Besides, the structure displacement $\bm \xxi$ satisfies the following equation
\begin{align}\label{elliptic-structure-disp}
-\bm \nabla \cdot \Big(2 \mu_s \bm \epsilon(\bm \xxi)+ \lambda_s(\nabla \cdot \bm \xxi)\Id \Big)=
 -\rho_s \tdet(\bm \nabla \tp)\partial^2_t \xxi
 +
 \bm {H^c_{\xxi}} 
+\bm {H^d_{\xxi}},
\end{align}
with
\begin{align*}
H^c_{\xit,i} = 
\sum_{\al,j,\be=1}^3 \big( \tciajb^l + \tciajb^q \big) \partial^2_{\al \be} \xit_j , \quad \textup{for} \ i=1,2,3,
\end{align*}
and
\begin{align*}
H^d_{\xit,i}
=
\mathsf{C}\sum_{\al,j,\be=1}^3 \big( \tdiajb^L + \tdiajb^Q+
\tdiajb^T+\tdiajb^F \big) \partial^2_{\al \be} \xit_j , \quad \textup{for} \ i=1,2,3.
\end{align*}
Using elliptic estimates and thanks to \eqref{lq-boundedness}
we get
\begin{align}\label{elliptic-estimate-structure}
\norm \bm \xxi \norm_{L^\infty(H^2)} \leq \rho_s CTM \norm \partial^2_t \bm \xxi \norm_{L^\infty(L^2)}+CT(M+M^2+M^3+M^4) \norm \bm \xxi \norm_{L^\infty(H^2)}.
\end{align}
\\
To bound $\norm \partial_t \vv \norm_{L^\infty(L^2(\OF))}$ and $\norm \partial_t^2 \xxi \norm_{L^\infty(L^2(\OS))}$ we use \eqref{Second-A priori-estimate}.
Finally, taking $T$ small with respect to $M$ and the initial conditions in \eqref{elliptic-estimate-fluid} and \eqref{elliptic-estimate-structure}, then combining them with \eqref{Second-A priori-estimate}, we achieve the following estimate
\begin{align}\label{SEcond-estimate-with given hat fn}
\norm \bm \vv \norm^2_{F^T_2} + \norm \bm \xxi \norm^2_{S^T_2} \leq CT^\kappa M \norm \bm {\hat{\xi}} \norm^2_{S^T_2} +C \norm \bm v_0 \norm^2_{H^1} + C \norm \bm \xi_1 \norm^2_{H^1}.
\end{align}

\subsection{Fixed Point Theorem for the Linearized System}
Based on the estimate \eqref{SEcond-estimate-with given hat fn} on the solution $(\bm \vv,\bm \xxi)$ of the linear system \eqref{Linear-System}, we proceed to prove that the function $\Psi_0$ is a contraction on $S^T_2$.
Let $\bm {\hat{\xi}}_1, \bm {\hat{\xi}}_2 \in S^T_2$. For $a=1,2$, we denote by $(\bm \vv_a,\bm \xxi_a)$ the solution of \eqref{Auxiliary-pro} with
\begin{align*}
g_i=\hat{h}_i^a = -\sum\limits_{\al,j,\be=1}^3 \bigg(\displaystyle \int_0^t \partial_s \tbiajb \partial_\be (\hat{\xi}_a)_j \ ds \bigg) \tilde{n}_{\al}, \quad i=1,2,3.
\end{align*}
Since $(\bm \vv_1,\bm \xxi_1)$ and $(\bm \vv_2,\bm \xxi_2)$ satisfy System \eqref{Auxiliary-pro}, then we can say that $(\bm \vv_1-\bm \vv_2,\bm \xxi_1-\bm \xxi_2)$ satisfies System \eqref{Auxiliary-pro} with $g_i = \hat{h}_i^1-\hat{h}_i^2$ and null initial data. Hence, applying \eqref{SEcond-estimate-with given hat fn} to $(\bm \vv_1-\bm \vv_2,\bm \xxi_1-\bm \xxi_2)$ and noticing that the right hand side of the estimate contains only a norm on $S^T_2$ given by $\norm \bm {\hat{\xi}}_1 - \bm {\hat{\xi}}_2 \norm_{S^T_2}$ added to some constants, consequently we get 
\begin{align}
\norm \bm \xxi_1 - \bm \xxi_2 \norm_{S^T_2} 
= 
\norm \Psi_0(\bm {\hat{\xi}_1})-\Psi_0(\bm {\hat{\xi}_2}) \norm_{S^T_2} \leq CT^{\kappa} M \norm \bm {\hat{\xi}_1} -\bm { \hat{\xi}_2}\norm_{S^T_2}.
\end{align}
Taking $T$ small enough with respect to $M$, gives that $\Psi_0$ is a contraction on $S^T_2$. Therefore, we assure the existence and uniqueness of a fixed point $\bm \xxi \in S^T_2$.
Consequently, we obtain the existence and uniqueness of a solution $(\bm \vv,\bm \xxi)$ for the system \eqref{Linear-System}.
Finally, with the assumption of $T$ being small with respect to $M$ and denoting $C \norm \bm v_0 \norm^2_{H^1} + C \norm \bm \xi_1 \norm^2_{H^1}$ by $\bm C_0$
 we obtain
\begin{align}\label{2nd-estimate-with given hat fn}
\norm \bm \vv \norm^2_{F^T_2} + \norm \bm \xxi \norm^2_{S^T_2} 
\leq 
\bm C_0.
\end{align}

\section{Regularity of Solution of the Linearized System
}\label{Regularity-sec}

\subsection{Regularity of the solution}

\begin{proposition}\label{Regularity-LinearSystem}
Let $(\bm {\breve v},\txi)\in A^T_M$, with the  assumption that $\bm v_0 \in H^6(\OF)$ and $\bm \xi_1 \in H^3(\OS)$ and satisfies \eqref{comp-cdtns}. For $T$ small with respect to $M$ and the initial conditions, the solution $(\vv,\xxi)$ is in the space $F^T_4 \times S^T_4$. Further, it satisfies
\begin{align}
\norm \vv \norm_{F^T_4}
+
\norm \xxi \norm_{S^T_4}
\leq
\bm C_0,
\end{align} 
where $\bm C_0$ denotes a constant in the norms $\norm \bm v_0 \norm_{H^6(\OF)}$ and $\norm \bm \xi_1 \norm_{H^3(\OS)}$.
\end{proposition}
By Proposition \ref{Ex-Uni-Linear}, we have proved the existence and uniqueness of $(\vv,\xxi) \in F^T_2 \times S^T_2$. Increasing the regularity of the initial conditions results a more regular solution \cite{Evans, Brezis}. The regularity of the solution in case of a linear fluid-structure interaction problem where the structure is considered to be quasi-incompressible have been proved in \cite{Coutand-Quasi}. Hence $(\vv,\xxi)$ belongs to $F^T_4 \times S^T_4$. 
\\

Next, we proceed to derive a priori estimates on the solution $(\vv,\xxi)$ in $F^T_4 \times S^T_4$.

\subsection{A Priori estimates on $\tilde{\gamma}$ in $A^T_M$}

\subsubsection*{A Priori Estimates Using Time Regularity}

The solution $\bm {\tilde \gamma}$ satisfies \eqref{final-eqn-weak-1}
with
\[
g_i=-\sum\limits_{\al,j,\be=1}^3 \bigg(\displaystyle \int_0^t \partial_s \tbiajb \partial_\be \xit_j \ ds \bigg) \tilde{n}_{\al}, \qquad i=1,2,3.
\]
Differentiating three times with respect to time and taking $\bm {\tilde \eta}=
\partial_t^3 \bm{\tilde \gamma}$ yield
\begin{equation}\label{final-eqn-weak-regular}
\begin{cases}
\rho_f 
\displaystyle
\int_{\OF} 
\tdet(\bm \nabla \tca)\partial^4_t \bm {\tilde \gamma} \cdot \partial_t^3 \bm{\tilde \gamma} \ d\bm {\tilde x}
+
C_1

+\rho_s 
\displaystyle
\int_{\OS} \tdet(\bm \nabla \tp) \partial^4_t \bm {\tilde \gamma} \cdot \partial_t^3 \bm{\tilde \gamma} \ d\bm {\tilde x}
+
C_2

\vspace{2mm} \\

+
\displaystyle
\mu
\int_{\OF}

\partial^3_t
 \Big(
  \nga (\nc)^{-1}
 +  \itnc (\nga)^t 
 \Big) \cch 
 : 
  \bm \nabla \partial^3_t \bm {\tilde \gamma} \ d\bm {\tilde x}
+
C_3

\vspace{2mm} \\

+
\displaystyle
\sum_{i,\al,j,\be=1}^3 
\int_{\OS} 
\tbiajb \partial_\be 
\partial^2_t
\tilde{\gamma}_j  
\ \partial_\al \partial_t^3 \tilde {\gamma}_i \ d\bm {\tilde x} 
+C_4

=

\displaystyle
\int_{\Gamma_c(0)}
\partial_t^3 \bm g
\cdot
\partial_t^3 \bm {\tilde \gamma} \ d\tilde{\Gamma},

\end{cases}
\end{equation}
where,
\begin{align*}
C_1
=
\displaystyle
3 \rho_f
\int_{\OF} 
\partial_t
\tdet(\bm \nabla \tca) \partial^3_t \bm {\tilde \gamma} \cdot \partial_t^3 \bm{\tilde \gamma} \ d\bm {\tilde x}
+
\displaystyle
3 \rho_f
\int_{\OF} 
\partial_t^2
\tdet(\bm \nabla \tca)\partial^2_t \bm {\tilde \gamma} \cdot \partial_t^3 \bm{\tilde \gamma} \ d\bm {\tilde x}
+ \rho_f
\displaystyle
\int_{\OF} 
\partial_t^3
\tdet(\bm \nabla \tca)\partial_t \bm {\tilde \gamma} \cdot \partial_t^3 \bm{\tilde \gamma} \ d\bm {\tilde x},
\end{align*}
\begin{align*}
C_2
=
\displaystyle
3 \rho_s
\int_{\OS} 
\partial_t
\tdet(\bm \nabla \tp) \partial^3_t \bm {\tilde \gamma} \cdot \partial_t^3 \bm{\tilde \gamma} \ d\bm {\tilde x}
+
\displaystyle
3 \rho_s
\int_{\OS} 
\partial_t^2
\tdet(\bm \nabla \tp)\partial^2_t \bm {\tilde \gamma} \cdot \partial_t^3 \bm{\tilde \gamma} \ d\bm {\tilde x}
+ \rho_s
\displaystyle
\int_{\OS} 
\partial_t^3
\tdet(\bm \nabla \tp)\partial_t \bm {\tilde \gamma} \cdot \partial_t^3 \bm{\tilde \gamma} \ d\bm {\tilde x},
\end{align*}
\begin{align*}
C_3
=&
\displaystyle
3 \mu
\int_{\OF} 
\partial^2_t
 \Big(
  \nga (\nc)^{-1}
 +  \itnc (\nga)^t 
 \Big) 
 \partial_t \cch 
 :
 \bm \nabla \partial^3_t \bm {\tilde \gamma} \ d\bm {\tilde x}
\\
&
+
\displaystyle
3 \mu
\int_{\OF} 
\partial_t
 \Big(
  \nga (\nc)^{-1}
 +  \itnc (\nga)^t 
 \Big) 
 \partial^2_t \cch 
 : 
 \bm \nabla \partial^3_t \bm {\tilde \gamma} \ d\bm {\tilde x}
\\
&+
\displaystyle
\mu
\int_{\OF} 
 \Big(
  \nga (\nc)^{-1}
 +  \itnc (\nga)^t 
 \Big) 
 \partial^3_t \cch 
 : 
 \bm \nabla \partial^3_t \bm {\tilde \gamma} \ d\bm {\tilde x},
\end{align*}
\begin{align*}
C_4
=& 
\displaystyle
\sum_{i,\al,j,\be=1}^3 
\int_{\OS} \partial^3_t
\tbiajb \partial_\be 
( 
\mathsmaller 
\int_0^t \tilde{\gamma}(s) ds)_j  
\ \partial_\al \partial_t^3 \tilde {\gamma}_i \ d\bm {\tilde x} 
+
\displaystyle
3
\sum_{i,\al,j,\be=1}^3 
\int_{\OS} \partial^2_t
\tbiajb \partial_\be 
\tilde{\gamma}_j  
\ \partial_\al \partial_t^3 \tilde {\gamma}_i \ d\bm {\tilde x}
\\
&
+3
\displaystyle
\sum_{i,\al,j,\be=1}^3 
\int_{\OS} \partial_t
\tbiajb \partial_\be 
\partial_t
\tilde{\gamma}_j  
\ \partial_\al \partial_t^3 \tilde {\gamma}_i \ d\bm {\tilde x}
+
\displaystyle
\sum_{i,\al,j,\be=1}^3 
\int_{\OS} 
\partial_t
\Big(\partial_\al \tbiajb \partial_\be 
( 
\mathsmaller 
\int_0^t 
\tilde{\gamma}(s) ds)_j 
\Big) 
\ \partial_t^3 \tilde {\gamma}_i \ d\bm {\tilde x}.
\end{align*}
First proceeding as in \eqref{fluid-est1} we get
\begin{align}\label{Reg-fe1}
\rho_f 
\displaystyle
\int_0^t
\int_{\OF} 
\tdet(\bm \nabla \tca)\partial^4_s \bm {\tilde \gamma} \cdot \partial_s^3 \bm{\tilde \gamma} \ d\bm {\tilde x} \ ds
\geq
\dfrac{\rho_f}{2}
(1-CT^\kappa M) \norm \partial^3_t \bm {\tilde \gamma} \norm^2_{L^\infty(L^2(\OF))}
-
\norm \bm v_0 \norm^2_{H^6}.
\end{align}
For the fluid stress term we proceed as in \eqref{fstress-estimate-partialt} to get
\begin{align}\label{Regular-fe2}
\displaystyle
\mu
\int_0^t
\int_{\OF}  
\partial^3_s
 \Big(
  \nga (\nc)^{-1}
 +  \itnc (\nga)^t 
 \Big) \cch 
 : \bm \nabla \partial^3_s \bm {\tilde \gamma}\ d\bm {\tilde x} \ ds
\geq
\mu ( C_k- CT^\kappa M )
\norm \partial^3_t \bm {\tilde \gamma} \norm_{L^2(H^1(\OF))}.
\end{align}
On the domain $\OS$, similarly as \eqref{Reg-fe1}, we have
\begin{align}\label{Regular-se1}
\rho_s 
\displaystyle
\int_0^t
\int_{\OS} 
\tdet(\bm \nabla \tp)\partial^4_s \bm {\tilde \gamma} \cdot \partial_s^3 \bm{\tilde \gamma} \ d\bm {\tilde x} \ ds
\geq
\dfrac{\rho_s}{2}
(1-CT^\kappa M) \norm \partial^3_t \bm {\tilde \gamma} \norm^2_{L^\infty(L^2(\OS))}
-
\norm \bm \xi_1 \norm^2_{H^3}.
\end{align}
Using \eqref{lq-boundedness}-\eqref{coercivity} with Korn's inequality give
\begin{align}\label{Regular-se2}
\displaystyle
\sum_{i,\al,j,\be=1}^3
\int_0^t
\int_{\OS} 
\tbiajb \partial_\be 
\partial^2_s
\tilde{\gamma}_j  
\ \partial_\al \partial_s^3 \tilde{\gamma}_i \ d \bm{\tilde x} \ ds
\geq
\mu_s C_k \norm \partial^2_t \bm {\tilde \gamma} 
\norm^2_{L^\infty(H^1(\OS))}
-
CT^\kappa M 
\norm \mathsmaller 
\int_0^{\bigcdot}
\bm {\tilde \gamma}(s) ds \norm^2_{S^T_4}. 
\end{align}
Further, proceeding as in \eqref{partial_h} and \eqref{partial-boundary} we get
\begin{equation}\label{Reg-g}
\begin{split}
\Bigg \lvert
\displaystyle
\int_0^t
\int_{\Gamma_c(0)}
\partial_s^3 \bm g
\cdot
\partial_s^3 \bm {\tilde \gamma}
\ d\tilde{\Gamma} \ ds
\Bigg \lvert
\leq
CT^\kappa M 
\norm \mathsmaller \int_0^\bigcdot 
\bm {\tilde \gamma}(s) ds \norm_{S^T_4}
\norm \bm {\tilde \gamma} \norm_{F^T_4}.
\end{split}
\end{equation}
On the other hand, to deal with $C_1$, $C_2$, $C_3$ and $C_4$ we use the following bounds 
\begin{align}\label{BOUNDS-higher-der}
\norm \partial^k_t \tdet(\bm \nabla \tca) \norm_{L^\infty(L^\infty(\OF))}
\leq
CM^k, \qquad k=1,2,3.
\\
\norm \partial^k_t \tdet(\bm \nabla \tp) \norm_{L^\infty(L^\infty(\OS))}
\leq
CM^k, \qquad k=1,2,3.
\\
\norm \partial^k_t \itnc \norm_{L^\infty(L^\infty(\OF))}
\leq
CM^k, \qquad k=1,2,3.
\\
\norm \partial^k_t \cof(\bm \nabla \tca) \norm_{L^\infty(L^\infty(\OF))}
\leq
CM^k, \qquad k=1,2,3.
\end{align}
Then,
\begin{align}\label{regular-fe3}
\displaystyle 
\int_0^t 
C_1 \ ds
\leq
CT^\kappa M \norm \bm {\tilde \gamma} \norm^2_{F^T_4}.
\end{align}
Similarly
\begin{align}\label{regular-fe4}
\displaystyle 
\int_0^t 
C_2 \ ds
\leq
CT^\kappa M \norm \bm {\tilde \gamma} \norm^2_{F^T_4}.
\end{align}
On the other hand,
\begin{align}\label{regular-se3}
\displaystyle 
\int_0^t 
C_3 \ ds
\leq
CT^\kappa M \norm 
\mathsmaller 
\int_0^\bigcdot
\bm {\tilde \gamma}(s) ds \norm^2_{S^T_4}
\end{align}
and
\begin{align}\label{regular-se4}
\displaystyle 
\int_0^t 
C_4 \ ds
\leq
CT^\kappa M \norm 
\mathsmaller 
\int_0^\bigcdot
\bm {\tilde \gamma}(s) ds \norm^2_{S^T_4}.
\end{align}
Combining \eqref{Reg-fe1}-\eqref{Reg-g} with \eqref{regular-fe3}-\eqref{regular-se4}
and considering the restriction of $\bm {\tilde \gamma}$ on each sub-domain give
\begin{align}\label{Estimate-reg-time}
&\norm \partial_t^3 \vv \norm^2_{L^\infty(L^2(\OF))}
+
\norm \partial_t^2 \vv \norm^2_{L^2(H^1(\OF))}
+
\norm 
\partial_t^3 \xxi \norm^2_{L^\infty(H^1(\OS))}
+
\norm 
\partial_t^4
\xxi \norm^2_{L^\infty(L^2(\OS))}
\nonumber \\
&\leq
CT^\kappa M 
(
\norm \vv \norm^2_{F^T_4}
+
\norm \xxi \norm^2_{S^T_4}
)
+
C
(
\norm \bm \xi_1 \norm^2_{H^3}
+
\norm \bm v_0 \norm^2_{H^6}
).
\end{align}
This estimate together with \eqref{2nd-estimate-with given hat fn} lead to the following estimate
\begin{align}\label{Estimate-reg-time11}
&\norm \vv \norm^2_{W^{3,\infty}(L^2(\OF))}
+
\norm \vv \norm^2_{W^{2,\infty}(H^1(\OF))}
+
\norm 
\vv \norm^2_{H^3(H^1(\OS))}
+
\norm 
\xxi \norm^2_{W^{3,\infty}(H^1(\OS))}
+
\norm 
\xxi \norm^2_{W^{4,\infty}(L^2(\OS))}
\nonumber \\
&\leq
CT^\kappa M 
(
\norm \vv \norm^2_{F^T_4}
+
\norm \xxi \norm^2_{S^T_4}
)
+
C
(
\norm \bm \xi_1 \norm^2_{H^3}
+
\norm \bm v_0 \norm^2_{H^6}
).
\end{align}
\subsection*{Spatial Regularity}
$\bullet$ \textbf{Step 1:} Estimates on $\bm {\tilde v}$ in $W^{2,\infty}(H^2(\OF))$ and $\bm {\tilde \xi}$ in $W^{2,\infty}(H^2(\OS))$.
\\
The fluid velocity $\bm {\tilde v}$ satisfies the elliptic equation
\begin{align}\label{Elliptic-fluid}
\mu \bm \nabla \cdot \big(\nv+(\nv)^t)=\rho_f
\tdet(\bm \nabla \tca) \partial_t \bm \vv + F_{\tilde {v}} \qquad & \textrm{in} \quad \OF,
\end{align}
with
\begin{align*}
F_{\tilde {v}} = - \mu \nabla \cdot
\bm {f_{\vv}},
\end{align*}
where 
\begin{align*}
\bm {f_{\vv}}=&
\Big(\nv \big( \inc-\Id \big) + \big(\itnc -\Id \big) \nv)^t \Big) \cch 
- \big( \nv +(\nv)^t \big ) \big(\cch -\Id \big). 
\end{align*}
as defined in Subsection \ref{Estimates1-spatial}.
First we have 
\begin{align*}
\norm 
\partial^2_t 
\big(
\tdet(\bm \nabla \tca) \partial_t \bm {\tilde v}
\big)
 \norm_{L^\infty(L^2)}
\leq
CM^2 \norm \vv \norm_{W^{3,\infty}(L^2)}
\leq 
C\norm \bm v_0 \norm_{H^6}
+
C \norm \bm \xi_1 \norm_{H^3}
+
CT^\kappa M 
(\norm \vv \norm_{F^T_4}
+
\norm \xxi \norm_{S^T_4})
.
\end{align*}
First, let us estimate $F_{\tilde v}$ in $W^{2,\infty}(L^2)$. In fact differentiating $f_{\vv}$ two times in time gives 
\begin{align*}
&\Big[
\big( \partial_t^2 \bm \nabla \vv \big)
\big( \inc - \Id \big)
+2(\partial_t \bm \nabla \vv )
\big( \partial_t \inc \big)
+
(\bm \nabla \vv) 
\big( \partial^2_t \inc \big)
\Big] \cof(\bm \nabla \tca)
\\
&
+
\Big
(\partial_t \bm \nabla \vv )
\big(\inc - \Id  \big)
+
(\bm \nabla \vv) 
\big( \partial_t \inc  \big)
\Big] \big( \partial_t \cof(\bm \nabla \tca) \big)
+
(\bm \nabla \vv )
\big(\inc - \Id  \big)
\big( \partial^2_t \cof(\bm \nabla \tca) \big).
\end{align*}
Using \eqref{BOUNDS-higher-der} with the embedding of $H^2$ in $L^\infty$ and taking into consideration
\begin{align*}
\norm \vv \norm_{L^\infty(H^1)}
\leq C \norm \bm v_0 \norm_{H^6}
+
T
\norm \vv \norm_{H^3(H^1)}
\end{align*}
yield
\begin{align}
\norm F_{\tilde{v}} \norm_{W^{2,\infty}(L^2)}
\leq
CT^\kappa M
\norm \vv \norm_{W^{2,\infty}(H^1)}
+
C \norm \bm v_0 \norm_{H^6}. 
\end{align}
Therefore, using \eqref{Estimate-reg-time11} and the elliptic estimates on $\vv$ we obtain
\begin{align}\label{Spatial-estimate1-fluid}
\norm \vv \norm_{W^{2,\infty}(H^2)}
\leq
C\norm \bm v_0 \norm_{H^6}
+
C \norm \bm \xi_1 \norm_{H^3}
+
CT^\kappa M 
(\norm \vv \norm_{F^T_4}
+
\norm \xxi \norm_{S^T_4}).
\end{align}
\par 
\indent
On the other hand, the structure displacement $\bm {\tilde \xi}$ satisfies \eqref{elliptic-structure-disp}. Differentiating two times in time yield
\begin{align*}
-
\partial_t^2
\Big[\bm \nabla \cdot \Big(2 \mu_s \bm \epsilon(\bm \xxi)+ \lambda_s(\nabla \cdot \bm \xxi)\Id \Big)
\Big]
=
-\rho_s 
\partial_t^2 
\Big(\tdet(\bm \nabla \tp)\partial^2_t \xxi \Big)
+
\partial_t^2
\bm {H^c_{\xxi}} 
+
\partial_t^2
\bm {H^d_{\xxi}},
\end{align*}
with
\begin{align*}
H^c_{\xit,i} =
 \sum_{\al,j,\be=1}^3 \big( \tciajb^l + \tciajb^q \big) \partial^2_{\al \be} \xit_j , \quad \textup{for} \ i=1,2,3,
\end{align*}
and
\begin{align*}
H^d_{\xit,i}
=
\mathsf{C}\sum_{\al,j,\be=1}^3 \big( \tdiajb^L + \tdiajb^Q+
\tdiajb^T+\tdiajb^F \big) \partial^2_{\al \be} \xit_j , \quad \textup{for} \ i=1,2,3.
\end{align*}
First, we have
\begin{align*}
\norm 
\tdet(\bm \nabla \tp)
\partial_t^2 
\xxi
\norm_{W^{2,\infty(L^2(\OS))}}
\leq
C \norm \bm {\tilde \xi}
\norm_{W^{4,\infty}(L^2)}.
\end{align*}
Then using \eqref{Estimate-reg-time11} we get
\begin{align}\label{Structure-det-step1}
\norm 
\tdet(\bm \nabla \tp)
\partial_t^2 
\bm {\tilde \xi}
\norm_{W^{2,\infty(L^2(\OS))}}
\leq
C\norm \bm v_0 \norm_{H^6}
+
C \norm \bm \xi_1 \norm_{H^3}
+
CT^\kappa M 
(\norm \vv \norm_{F^T_4}
+
\norm \xxi \norm_{S^T_4}).
\end{align}
Further, for $\partial_t^2
\bm {H^c_{\xxi}}$ we have
\begin{align}
\partial_t^2 \bm {H^c_{\xxi}}(\bm {\tilde x},t)
=
\partial_t^2 \bm {H^c_{\xxi}}(\bm {\tilde x},0)
+
\displaystyle
\int_0^t
\partial_s^3
\bm {H^c_{\xxi}} (\bm {\tilde x},s) \ ds \qquad \forall \ \bm {\tilde x} \in \OS.
\end{align}
Simple calculation of $\partial_t^2 \bm {H^c_{\xxi}} (\bm {\tilde{x}},s)$ then setting $t=0$ and using the fact that $\partial_t \tciajb^l(\bm {\tilde x},0)$ is a function of $\bm \xi_1$ give \\
$\norm
\partial_t^3
\bm {H^c_{\xxi}} (\bm {\tilde x},s)\norm_{L^\infty(L^2(\OS))} \leq C \norm \bm \xi_1 \norm_{H^3}$.
Moreover,
\begin{align*}
\displaystyle
\int_0^t
\partial_s^3
\bm {H^c_{\xxi}} (\bm {\tilde x},s) \ ds
= &
\sum_{i,\al,j,\be=1}^3
\Bigg[
\int_0^t
\partial^3_s
\big( \tciajb^l + \tciajb^q \big) \partial^2_{\al \be} \xit_j \ ds
+
3
\int_0^t
\partial_s
\big( \tciajb^l + \tciajb^q \big) \partial^2_{\al \be} (\partial^2_s \xit_j) \ ds
\\
&
+
3
\int_0^t
\partial^2_s
\big( \tciajb^l + \tciajb^q \big) \partial^2_{\al \be} (\partial_s \xit_j) \ ds
+
\int_0^t
\big( \tciajb^l + \tciajb^q \big) \partial^2_{\al \be} (\partial^3_s \xit_j) \ ds
\Bigg].
\end{align*}
Hence, integrating over $\OS$ and using \eqref{lq-boundedness}, we get that the first three terms of the right hand side can be estimated in $L^\infty(L^2(\OS))$ by
\begin{align*}
CT^\kappa M \norm \xxi \norm_{S^T_4}.
\end{align*}
On the other hand, integrating by parts in time in the last integral of the right hand side gives
\begin{align}\label{Structure-Hc-ipb-0}
\big( \tciajb^l + \tciajb^q \big) \partial^2_{\al \be} (\partial^2_s \xit_j)(t)
-
\big( \tciajb^l + \tciajb^q \big) \partial^2_{\al \be} (\partial^2_s \xit_j)(0)
-
\displaystyle
\int_0^t
\partial_s
\big( \tciajb^l + \tciajb^q \big) \partial^2_{\al \be} (\partial^2_s \xit_j) \ ds
\end{align}
As $\txi(0)=0$, then 
$\displaystyle \sum_{i,\al,j,\be=1}^3 
\Big[
\big( \tciajb^l + \tciajb^q \big) \partial^2_{\al \be} (\partial^2_s \xit_j)\Big]
(0)=0$.
In addition,
\begin{align*}
\norm \tciajb^l + \tciajb \norm_{L^\infty(L^2)}
\leq
T
\norm 
\partial_t \tciajb^l + \partial_t \tciajb^q \norm_{L^\infty(L^2)}
\leq CT^\kappa M.
\end{align*}
Therefore,
\begin{align*}
\eqref{Structure-Hc-ipb-0}
\leq
CT^\kappa M \norm \xxi \norm_{S^T_4}.
\end{align*}
Consequently,
\begin{align}\label{Structure-Hc-ipb-1}
\norm 
\partial_t^2 \bm {H^c_{\xxi}}
\norm_{L^\infty(L^2(\OS))}
\leq
CT^\kappa M \norm \xxi \norm_{S^T_4}.
\end{align}
Similarly, one can show that
\begin{align}\label{Structure-Hd-ipb}
\norm 
\partial_t^2 \bm {H^d_{\xxi}}
\norm_{L^\infty(L^2(\OS))}
\leq 
CT^\kappa M \norm \xxi \norm_{S^T_4}.
\end{align}
As a result, combining \eqref{Structure-det-step1}, \eqref{Structure-Hc-ipb-1} and \eqref{Structure-Hd-ipb} an estimate on $\xxi$ in 
$W^{2,\infty}(L^2(\OS))$ is given by
\begin{align}\label{Spatial-estimate1-structure}
\norm \xxi \norm_{W^{2,\infty}(H^2(\OS))}
\leq
C\norm \bm v_0 \norm_{H^6}
+
C \norm \bm \xi_1 \norm_{H^3}
+
CT^\kappa M 
(\norm \vv \norm_{F^T_4}
+
\norm \xxi \norm_{S^T_4}).
\end{align}
Finally, combining \eqref{Spatial-estimate1-fluid} and \eqref{Spatial-estimate1-structure} we get
\begin{align}\label{Estimate1-reg-time}
\norm \vv \norm_{W^{2,\infty}(H^2(\OF))}
+
\norm \xxi \norm_{W^{2,\infty}(H^2(\OS))}
\leq
C\norm \bm v_0 \norm_{H^6}
+
C \norm \bm \xi_1 \norm_{H^3}
+
CT^\kappa M 
(\norm \vv \norm_{F^T_4}
+
\norm \xxi \norm_{S^T_4}).
\end{align}
$\bullet$ \textbf{Step 2:} Estimates on $\bm {\tilde v}$ in $L^\infty(H^4(\OF))$ and $\bm {\tilde \xi}$ in $L^\infty(H^4(\OS))$.
\\
Again, the fluid velocity satisfies \eqref{Elliptic-fluid}.
We estimate $F_{\tilde{v}}$ in $L^\infty(H^2(\OF))$. First,
\begin{align*}
\norm F_{\tilde{v}} \norm_{L^\infty(H^2(\OF))} 
\leq
\mu
\norm \bm f_{\bm {\tilde{v}}} \norm_{L^\infty(H^3(\OF))}.
\end{align*}
But,
\begin{align*}
\norm \bm f_{\bm {\tilde{v}}} \norm_{L^\infty(H^3)}
\leq&
2 \norm \tv \norm_{L^\infty(H^3)}
\norm  \inc-\Id \norm_{L^\infty(H^3)}
\norm \cch \norm_{L^\infty(H^3)}
+
\norm \cch -\Id \norm_{L^\infty(H^3)}
\norm \tv \norm_{L^\infty(H^3)}
\\
\leq &
CT^\kappa M
\norm \bm {\tilde v} \norm_{L^\infty(H^4)}.
\end{align*}
Further, using Estimate \eqref{Estimate1-reg-time} we have
\begin{align*}
\norm \tdet(\bm \nabla \tca) \partial_t \vv \norm_{L^\infty(H^2)}
\leq
CM \norm \vv \norm_{W^{2,\infty}(H^2)}
\leq
C\norm \bm v_0 \norm_{H^6}
+
C \norm \bm \xi_1 \norm_{H^3}
+
CT^\kappa M
(
\norm \vv \norm_{F^T_4}
+
\norm \xxi \norm_{S^T_4}
).
\end{align*}
Hence, the elliptic estimates yield
\begin{align}\label{Spatial-estimate2-fluid}
\norm \vv \norm_{L^\infty(H^4)}
\leq
C\norm \bm v_0 \norm_{H^6}
+
C \norm \bm \xi_1 \norm_{H^3}
+
CT^\kappa M
(
\norm \vv \norm_{F^T_4}
+
\norm \xxi \norm_{S^T_4}
).
\end{align}
Besides, the structure displacement $\xxi$ satisfies \eqref{elliptic-structure-disp}.
Then, by using the fact that $\bm \xi_0=0$ with \eqref{lq-boundedness} we have
\begin{align*}
\norm \tciajb^l + \tciajb \norm_{L^\infty(H^2(\OS))}
\leq
T
\norm 
\partial_t \tciajb^l + \partial_t \tciajb^q \norm_{L^\infty(H^2(\OS))}
\leq CT^\kappa M.
\end{align*}
Thus, $\bm H_{\xxi}^c$ can be estimated by 
\begin{align}\label{reg-Hc}
C \norm \partial_t^2 \xxi \norm_{L^\infty(H^2)} 
+ 
CT^\kappa M 
\norm \xxi \norm_{L^\infty(H^4)}.
\end{align}
By a similar argument, we find that $\bm H^d_{\xxi}$ can be estimated by
\begin{align*}
CT^\kappa M \norm \xxi \norm_{L^\infty(H^4)}.
\end{align*}
Thanks to the Estimate \eqref{Estimate1-reg-time} on $\xxi$, \eqref{reg-Hc} can be estimated by
\begin{align}\label{reg-Hc1}
C\norm \bm v_0 \norm_{H^6}
+
C \norm \bm \xi_1 \norm_{H^3}
+
CT^\kappa M ( \norm \vv \norm_{F^T_4}
+ \norm \xxi \norm_{S^T_4}).
\end{align}
Therefore, using the elliptic estimate we get
\begin{align}\label{Spatial-estimate2-structure}
\norm \xxi \norm_{L^\infty(H^4)}
\leq
C\norm \bm v_0 \norm_{H^6}
+
C \norm \bm \xi_1 \norm_{H^3}
+
CT^\kappa M ( \norm \vv \norm_{F^T_4}
+ \norm \xxi \norm_{S^T_4}).
\end{align}
Combining \eqref{Spatial-estimate2-fluid} and \eqref{Spatial-estimate2-structure} yield
\begin{align}\label{Estimate2-reg-time}
\norm \vv \norm_{L^\infty(H^4(\OF))}
+
\norm \xxi \norm_{L^\infty(H^4(\OS))}
\leq
C\norm \bm v_0 \norm_{H^6}
+
C \norm \bm \xi_1 \norm_{H^3}
+
CT^\kappa M 
(\norm \vv \norm_{F^T_4}
+
\norm \xxi \norm_{S^T_4}).
\end{align}
Finally, Estimates \eqref{Estimate-reg-time11}, \eqref{Estimate1-reg-time} and \eqref{Estimate2-reg-time} give
\begin{align}
\norm \vv \norm_{F^T_4}
+
\norm \xxi  \norm_{S^T_4}
\leq
C\norm \bm v_0 \norm_{H^6}
+
C \norm \bm \xi_1 \norm_{H^3}
+
CT^\kappa M
(
\norm \vv \norm_{F^T_4}
+
\norm \xxi  \norm_{S^T_4}
)
.
\end{align}
Assuming that $T$ small with respect to $M$ and the initial values yield
\begin{align}
\norm \vv \norm_{F^T_4}
+
\norm \xxi  \norm_{S^T_4}
\leq
C\norm \bm v_0 \norm_{H^6}
+
C \norm \bm \xi_1 \norm_{H^3}
=
\bm C_0.
\end{align}


\section{Existence of Solution of the Non-Linear Coupled Problem}
\label{Existence for nonlinear}
\indent
From Proposition \ref{Ex-Uni-Linear}, there exists $\bm {\hat{C}}_0>0$ and $\hat{\kappa}>0$ such that for all $M>0 \ \textrm{and} \ (\bm {\breve v},\bm {\breve \xi}) \in A^T_M$, there exists $T_1>0$ so that the solution of \eqref{Linear-System} satisfies
\begin{align}
\norm \bm \vv \norm^2_{F^T_4} + \norm \bm \xxi \norm^2_{S^T_4} \leq 
\bm {\hat{C}}_0,
\end{align}
for all $T \leq T_1$.\\
Taking $\hat{M}=\bm {\hat {C}}_0$ 
we get
\begin{align}\label{estimate with M}
\norm \bm \vv \norm^2_{F^T_4} + \norm \bm \xxi \norm^2_{S^T_4} \leq \hat{M}.
\end{align}

We seek to prove the existence of a solution of the non-linear coupled problem  \eqref{eqcoup1}-\eqref{eqcoup102}. To establish this result we use the fixed point theorem.
For this sake, 
for any $T \leq \hat{T}$, we setting $E=F^T_2 \times S^T_2$ and $W=A^T_{\hat{M}}$. The set $W$ is a closed subset of $E$. We define the function $\Psi : (\bm {\breve v},\bm {\breve \xi}) \longrightarrow (\bm \vv,\bm \xxi)$ that maps $(\bm {\breve v}, \txi) \in W $ into $(\bm \vv,\bm \xxi) \in W$  which is the solution of the linear system \eqref{Linear-System}.
An element $(a,b) \in \Psi(W)$ is written as $(a,b)=\Psi(\bm {\breve v},\bm {\breve \xi})$ where $(\bm {\breve v},\bm {\breve \xi})$ belongs to $W$. But the definition of $\Psi$ gives that $\Psi(\bm {\breve v},\bm {\breve \xi})=
(\vv,\xxi)$ which is the unique solution of the linear problem \eqref{Linear-System} in $W$, consequently $(a,b)=(\vv,\xxi) \in W$. Therefore, $\Psi(W) \subset W$.

Consider two pairs $(\bm {\breve v}_1,\bm {\breve \xi}_1)$ and $(\bm {\breve v}_2,\bm {\breve \xi}_2) \in W$ and two solutions $(\vv_1,\xxi_1),(\vv_2,\xxi_2)$ of the linear system \eqref{Linear-System} associated to $(\bm {\breve v}_1,\bm {\breve \xi}_1)$ and $(\bm {\breve v}_2,\bm {\breve \xi}_2)$, respectively. Therefore $\bm \vv_1,\bm \vv_2,\bm \xxi_1$ and $\bm \xxi_2$ satisfy the variational formulations \eqref{final-eqn-weak} and \eqref{differentiated weak form} with
\[
g_i =-\sum\limits_{\al,j,\be=1}^3 \bigg(\displaystyle \int_0^t \partial_s \tbiajb \partial_\be \xit_j \ ds \bigg) \tilde{n}_{\al}, \quad i=1,2,3.
\]
Set $\bm {\tilde \zeta} = \bm {\tilde \gamma}_1 -\bm {\tilde \gamma}_2$, then $\bm {\tilde \zeta}(0)=0$ . The main work in this section is to find estimates on $\bm {\tilde \zeta}$ and $\partial_t \bm {\tilde \zeta}$. These estimates will enable us to apply the fixed point theorem for a suitable choice of $T$ to be precised later.


\subsection{Estimates on $ \tilde{\zeta} $}
\label{1st-nonlinear-Estimate}
Consider $\bm {\tilde \zeta} = \bm {\tilde \gamma}_1 -\bm {\tilde \gamma}_2$ in \eqref{final-eqn-weak-1} then $\bm {\tilde \zeta}$ satisfies the following variational formulation
\begin{equation}\label{Vf-zeta}
\begin{cases}
\rho_f \displaystyle \int_{\OF} \tdet(\bm \nabla \tca_1) \partial_t \bm {\tilde \zeta} \cdot \bm {\tilde \eta} \ d\bm {\tilde x}

+
\displaystyle \int_{\OF} \bm {\breve \sigma}^0_1(\bm {\tilde \zeta}) :\bm \nabla \bm {\tilde \eta} \ d\bm {\tilde x}

+\rho_s \displaystyle \int_{\OS} \tdet(\bm \nabla \tp_1) \partial_t \bm {\tilde \zeta} \cdot \bm {\tilde \eta} \ d\bm {\tilde x}
\vspace{1mm} \\
+
\rho_f 
\mathlarger \int_{\OF}
\partial_t \bm {\tilde \gamma}_2 
\cdot
\big[ \tdet(\nc_1) - \tdet(\nc_2) \big] \bm {\tilde \eta} \ d\bm {\tilde x}

\vspace{1mm} \\
+
\rho_s \mathlarger \int_{\OS}
\partial_t \bm {\tilde \gamma}_2 
\cdot
\big[
\tdet(\bm \nabla \tp_1)-\tdet(\bm \nabla \tp_2)
\big] \bm {\tilde \eta} \ d\bm {\tilde x}

\vspace{1mm} \\

+\displaystyle \sum_{i,\al,j,\be=1}^3 \displaystyle \int_{\OS} \biajb (\bm \nabla \bm {\breve \xi}_1) 
\partial_\be (\int_0^t \tilde{\zeta}(s) ds)_j \ \partial_\al \tilde {\eta}_i \ d\bm {\tilde x}

\vspace{1mm} \\

+\displaystyle \sum_{i,\al,j,\be=1}^3 \displaystyle \int_{\OS} \partial_\al \biajb(\bm \nabla \bm {\breve \xi}_1) \partial_\be(\int_0^t \tilde{\zeta}(s) ds)_j \ \tilde {\eta}_i \ d\bm {\tilde x}

\vspace{1mm} \\

+\displaystyle \sum_{i,\al,j,\be=1}^3 \displaystyle \int_{\OS}
\Big(  \biajb (\bm \nabla \bm {\breve \xi}_1) - \biajb( \bm \nabla \bm {\breve \xi}_2) \Big) \partial_\be (\int_0^t \tilde{\gamma}_2(s) ds)_j \ \partial_\al \tilde {\eta}_i \ d\bm {\tilde x} 

\vspace{1mm} \\

+\displaystyle \sum_{i,\al,j,\be=1}^3 \displaystyle \int_{\OS}
\Big( \partial_\al \biajb (\bm \nabla \bm {\breve \xi}_1) - \partial_\al \biajb(\bm \nabla \bm {\breve \xi}_2) \Big) \partial_\be (\int_0^t \tilde{\gamma}_2(s) ds)_j \tilde {\eta}_i \ d\bm {\tilde x}

\vspace*{1mm} \\
+\displaystyle \int_{\OF} \bm F_0 : \bm \nabla \bm {\tilde \eta} \ d\bm {\tilde x}

=
\displaystyle \int_{\Gamma_c(0)} \bm G \cdot \bm {\tilde \eta} \ d\tilde{\Gamma}

\qquad

\qquad
 \forall \ \bm {\tilde \eta} \in \widetilde{\mathcal{W}},


%


\end{cases}
\end{equation}
where
\begin{align}\label{sigma-in-gamma}
\bm {\breve \sigma}^0_1(\bm {\tilde \zeta})
&=
\mu \big[ \bm \nabla \bm {\tilde \zeta} (\nc_1)^{-1} + (\nc_1)^{-t} (\nze)^{t} \big] \cof(\nc_1)
\end{align}
and
\begin{align}\label{F_0}
\bm F_0
&=
\mu \big[ \nga_2 \big( (\nc_1)^{-1} \cof(\nc_1) - (\nc_2)^{-1} \cof(\nc_2) \big) \big]
\nonumber \\
&
+
\mu
\big [ (\nc_1)^{-t} (\nga_2)^t \cof(\nc_1) - (\nc_2)^{-t} (\nga_2)^t \cof(\nc_2) \big].
\end{align}
Further, for $i=1,2,3$,
\begin{align}
G_i=
\displaystyle \sum_{\al,j,\be=1}^3
\Bigg(
\displaystyle \int_0^t
\partial_s \biajb(\bm \nabla \bm {\breve \xi}_1) \partial_\be (\mathsmaller \int_0^s \tilde{\zeta}(\tau) d\tau)_j \ ds
\Bigg) \tilde{n}_{\al}
%
+
\displaystyle \sum_{\al,j,\be=1}^3
\Bigg(
\displaystyle \int_0^t
\Big[
\partial_s \biajb(\bm \nabla \bm {\breve \xi}_1) - \partial_s \biajb(\bm \nabla \bm {\breve \xi}_2)
\Big] \partial_\be (\mathsmaller \int_0^s  \tilde{\gamma}_2(\tau) d\tau)_j \ ds
\Bigg) \tilde{n}_{\al}.
\end{align}
Moreover, for simplicity, in what follows we set
\begin{align}\label{L0-L1}
\bm L_0 = \rho_f 
\partial_t \bm {\tilde \gamma}_2 \big[ \tdet(\bm \nabla \tca_1) -  \tdet(\bm \nabla \tca_2) \big] 
\quad
\textup{and}
\quad
\bm L_1=
\rho_s
\partial_t \bm {\tilde \gamma}_2 
\big[
\tdet(\bm \nabla \tp_1)-\tdet(\bm \nabla \tp_2)
\big].
\end{align}
Taking $\bm {\tilde \eta} =\bm {\tilde \zeta}$ and using the fact that $\bm {\tilde \zeta}(0)=0$, then proceeding as in \eqref{Galerkin-2} yield
\begin{equation}\label{Nonlinear-vf}
\begin{cases}
\dfrac{\rho_f}{2} \displaystyle \int_{\OF} \tdet(\bm \nabla \tca_1) \lvert \bm \tilde{\zeta}(t) \lvert^2 \ d\bm {\tilde x}
-\dfrac{\rho_f}{2}
\mathlarger \int_0^t 
\mathlarger \int_{\OF}
\partial_s \tdet(\bm \nabla \tca_1) \lvert \bm {\tilde \zeta} \lvert^2 \ d\bm {\tilde x} \ ds
\vspace{1mm} \\
+ \displaystyle \int_0^t
\displaystyle \int_{\OF} \bm {\breve \sigma}_1(\bm {\tilde \zeta}) :\nze \ d\bm {\tilde x} \ ds
+\dfrac{\rho_s}{2}
\displaystyle \int_{\OS} \tdet(\bm \nabla \tp_1) \lvert \bm {\tilde \zeta}(t) \lvert^2 \ d\bm {\tilde x}
\vspace{1mm} \\
-\dfrac{\rho_s}{2}
\mathlarger \int_0^t
\mathlarger \int_{\OS}
\partial_s \tdet(\bm \nabla \tp_1) \lvert \bm {\tilde \zeta} \lvert^2 \ d\bm {\tilde x} \ ds
+
\displaystyle \int_0^t
\displaystyle \int_{\OF} F_0 : \nze \ d\bm {\tilde x} \ ds 

\vspace{1mm} \\

+
\mathlarger \int_0^t  
\mathlarger \int_{\OF}
\bm L_0 \cdot \bm {\tilde \eta} \ d\bm {\tilde x} \ ds

+

\mathlarger \int_0^t 
\mathlarger \int_{\OS}
\bm L_1 \cdot \bm {\tilde \eta} \ d\bm {\tilde x} \ ds

\vspace{1mm} \\
+
\dfrac{1}{2}
\displaystyle \sum_{i,\al,j,\be=1}^3 \displaystyle \int_{\OS}
\biajb (\bm \nabla \bm {\breve \xi}_1) \partial_\be (\mathsmaller \int_0^t \tilde{\zeta}(s) ds)_j \ \partial_\al (\mathsmaller \int_0^t \tilde{\zeta}(s) ds)_i
\ d\bm {\tilde x}

\vspace{1mm} \\
-
\dfrac{1}{2}
\displaystyle \sum_{i,\al,j,\be=1}^3
\displaystyle \int_0^t
\displaystyle \int_{\OS}
\partial_s
\biajb (\bm \nabla \bm {\breve \xi}_1) \partial_\be (\mathsmaller \int_0^s \tilde{\zeta}(\tau) d\tau)_j \ \partial_\al (\mathsmaller \int_0^s \tilde{\zeta}(\tau) d\tau)_i
\ d\bm {\tilde x} \ ds
\vspace{1mm} \\

+
\displaystyle \sum_{i,\al,j,\be=1}^3
\displaystyle \int_0^t
\displaystyle \int_{\OS} \partial_\al \biajb(\bm \nabla \bm {\breve \xi}_1) \partial_\be (\mathsmaller \int_0^s \tilde{\zeta}(\tau) d\tau)_j \ \partial_s (\mathsmaller \int_0^s \tilde{\zeta}(\tau) d\tau)_i \ d\bm {\tilde x} \ ds

\vspace{1mm} \\

+
\displaystyle \sum_{i,\al,j,\be=1}^3
\displaystyle \int_0^t
\displaystyle \int_{\OS}
\Big(  \biajb (\bm \nabla \bm {\breve \xi}_1) - \biajb( \bm \nabla \bm {\breve \xi}_2) \Big) \partial_\be (\mathsmaller \int_0^s \gamma_2(\tau) d\tau)_j \ \partial^2_{s \al} (\mathsmaller \int_0^s \tilde{\zeta}(\tau) d\tau)_i \ d\bm {\tilde x} \ ds 
\vspace{1mm} \\

+
\displaystyle \sum_{i,\al,j,\be=1}^3
\displaystyle \int_0^t
\displaystyle \int_{\OS}
\Big( \partial_\al \biajb (\bm \nabla \bm {\breve \xi}_1) - \partial_\al \biajb(\bm \nabla \bm {\breve \xi}_2) \Big) \partial_\be (\mathsmaller \int_0^s \tilde{\gamma}_2(\tau) d\tau)_j \partial_s (\mathsmaller \int_0^s \tilde{\zeta}(\tau) d\tau)_i \ d\bm {\tilde x} \ ds

 \vspace{1mm} \\
=
\displaystyle \int_0^t
\displaystyle \int_{\Gamma_c(0)} \bm G \cdot \partial_s(\mathsmaller \int_0^s \bm {\tilde \zeta}(\tau) d\tau) \ d\tilde{\Gamma} \ ds
.

\qquad
\qquad



\end{cases}
\end{equation}
%
%
%
%
%
We proceed to estimate the terms of 
\eqref{Nonlinear-vf} in the spirit of \cite{Boulakia2} by using the fact that 
\begin{equation}\label{Est-det-cof}
\begin{split}
\norm
\cof(\nc_1) - \cof(\nc_2)
\norm_{L^\infty(H^1)}
\leq
C \norm \bm {\breve v}_1 - \bm {\breve v}_2 \norm_{F^T_2},
\\
\norm (\nc_1)^{-1} - (\nc_2)^{-1}
\norm_{L^\infty(H^1)} 
\leq
C \norm \bm {\breve v}_1 - \bm {\breve v}_2 \norm_{F^T_2},
\\
\norm 
\tdet(\nc_1)-\tdet(\nc_2)
\norm_{L^\infty(H^1)} 
\leq
C \norm \bm {\breve v}_1 - \bm {\breve v}_2 \norm_{F^T_2}
\\
\textup{and} \hspace{5cm}
\\
\norm
\tdet(\bm \nabla \tp_1)-\tdet(\bm \nabla \tp_2)
\norm_{L^\infty(H^1)} \leq
C 
\norm \txi_1 - \txi_2
\norm_{S^T_2}
\end{split}
\end{equation}
which can be established in the similar manner used in Lemma \ref{Estimates on Flow}.
\\
First, using Lemma \ref{Estimate-structure-defo} we have
\begin{align*}
\dfrac{\rho_s}{2}
\mathlarger \int_{\OS}
\tdet(\bm \nabla \tp_1) \lvert \bm {\tilde \zeta}(t) \lvert^2 \ d\bm {\tilde x}
-
\dfrac{\rho_s}{2}
\mathlarger \int_0^t
\mathlarger \int_{\OS}
\partial_t \tdet(\bm \nabla \tp_1)
\lvert \bm {\tilde \zeta} \lvert^2 \ d\bm {\tilde x} \ ds
\\
\geq
\rho_s(1-CT^\kappa M) \norm \bm {\tilde \zeta} \norm^2_{L^\infty(L^2(\OS))}.
\end{align*}
Using \eqref{lq-boundedness} and \eqref{lq1-boundedness}, for all $i,\al,j,\be \in \{1,2,3 \}$ we have 
\begin{equation}\label{eq-xi12}
\begin{split}
&\norm \biajb(\bm \nabla \bm {\breve \xi}_1)- \biajb(\bm \nabla \bm {\breve \xi}_2) \norm_{L^\infty(H^1)} \leq C \norm \bm {\breve \xi}_1-  \bm {\breve \xi}_2 \norm_{L^\infty(H^2)},
\\
&\norm \partial_\al \biajb(\bm \nabla \bm {\breve \xi}_1)- \partial_\al \biajb(\bm \nabla \bm {\breve \xi}_2) \norm_{L^\infty(L^2)} \leq C \norm \bm {\breve \xi}_1-  \bm {\breve \xi}_2 \norm_{L^\infty(H^2)}
\\
& \hspace{4cm}\textup{and}
\\
&\norm \partial_t \biajb(\bm \nabla \bm {\breve \xi}_1)- \partial_t \biajb(\bm \nabla \bm {\breve \xi}_2) \norm_{L^\infty(L^2)} \leq C \norm \bm {\breve \xi}_1-  \bm {\breve \xi}_2 \norm_{W^{1,\infty}(H^1)}.
\end{split}
\end{equation}
Then an estimate on $\bm G$ is given by
\begin{align}
\norm\bm  G \norm_{H^1(L^2(\Gamma_c(0)))} \leq
CT^\kappa M
\Big(
\norm \bm {\breve \xi}_1 - \bm {\breve \xi}_2 \norm_{S^T_2} + \norm \bm {\tilde \zeta} \norm_{L^\infty(H^1(\OS))}
\Big).
\end{align}
Hence, proceeding similarly as in \eqref{g-estimate} we get 
\begin{align}\label{estimate-G-int}
\displaystyle \int_0^t
\displaystyle \int_{\Gamma_c(0)} \bm  G \cdot 
\partial_s(\mathsmaller \int_0^s \bm {\tilde \zeta}(\tau) d\tau) \ d\tilde{\Gamma} \ ds
& \leq
CT^\kappa M \norm \bm {\breve \xi}_1 - \bm {\breve \xi}_2 \norm^2_{S^T_2}
+
\delta CT^\kappa M \norm \bm {\tilde \zeta} \norm^2_{S^T_2}.
\end{align}
Taking into consideration \eqref{eq-xi12} and the embedding $H^1 \subset L^6$ \cite[Theorem 9.9]{Brezis} we obtain

\begin{align*}
&\Bigg \lvert
-
\displaystyle \sum_{i,\al,j,\be=1}^3
\displaystyle \int_0^t
\displaystyle \int_{\OS}
\Big(  \biajb (\bm \nabla \bm {\breve \xi}_1) - \biajb( \bm \nabla \bm {\breve \xi}_2) \Big) 
\partial_\be (\mathsmaller \int_0^s  \tilde{\gamma}_2(\tau) d\tau)_j 
\ 
\partial^2_{s \al} (\mathsmaller \int_0^s \tilde{\zeta}(\tau) d\tau)_i 
\ d\bm {\tilde x} \ ds \\ 
&-
\displaystyle \sum_{i,\al,j,\be=1}^3
\displaystyle \int_0^t
\displaystyle \int_{\OS}
\Big( \partial_\al \biajb (\bm \nabla \bm {\breve \xi}_1) - \partial_\al \biajb(\bm \nabla \bm {\breve \xi}_2) \Big) 
\partial_\be (\mathsmaller \int_0^s  \tilde{\gamma}_2(\tau) d\tau)_j 
\partial_s (\mathsmaller \int_0^s \tilde{\zeta}(\tau) d\tau)_i \ d\bm {\tilde x} \ ds
\Bigg \lvert
\\
&\leq
CTM \norm \bm {\breve \xi}_1 - \bm {\breve \xi}_2 \norm^2_{L^\infty(H^2(\OS))}
+
CTM \norm \bm  {\tilde \zeta} \norm^2_{L^\infty(H^1(\OS))}.
\end{align*}
On the contrary, using \eqref{coercivity} and \eqref{coercivity1} we have
\begin{equation}
\begin{split}
&\dfrac{1}{2}
\displaystyle \sum_{i,\al,j,\be=1}^3 
\displaystyle \int_{\OS}
\biajb (\bm \nabla \bm {\breve \xi}_1)(t) 
\partial_\be (\mathsmaller \int_0^t \tilde{\zeta}(s) ds)_j
\partial_\al (\mathsmaller \int_0^t \tilde{\zeta}(s) ds)_i
d\bm {\tilde x}
\\
&\geq
\mu_s \norm \mathsmaller \int_0^{\bigcdot} \bm {\tilde \zeta}(s) ds \norm^2_{L^\infty(H^1(\OS))}
+
\dfrac{\lambda_s+\mathsf{C}}{2}
\norm \nabla \cdot \mathsmaller \int_0^{\bigcdot} \bm {\tilde \zeta}(s) \ ds \norm^2_{L^\infty(L^2(\OS))}
%
%
%
%
-CT^\kappa M \norm \mathsmaller \int_0^{\bigcdot} \bm {\tilde \zeta}(s) \ ds \norm^2_{L^\infty(H^1(\OS))}.
\end{split}
\end{equation}
Whereas, for the integrals on the fluid domain $\OF$ we have 
\begin{align*}
\dfrac{\rho_f}{2}
\mathlarger \int_{\OF}
&
\tdet(\bm \nabla \tca_1) \lvert \bm {\tilde \zeta}(t) \lvert^2 \ d\bm {\tilde x}
+
\dfrac{\rho_f}{2}
\mathlarger \int_0^t
\mathlarger \int_{\OF}
\partial_t \tdet(\bm \nabla \tca_1)
\lvert \bm {\tilde \zeta} \lvert^2 \ d\bm {\tilde x} \ ds
\geq
\dfrac{\rho_f}{2}
(1-CT^\kappa M) \norm \bm {\tilde \zeta} \norm^2_{L^\infty(L^2(\OF))}.
\end{align*}
On the other hand, for $\bm F_0$ we have
\begin{align}\label{Estimate-fluid-f0}
\norm \bm F_0 \norm^2_{L^2(L^2(\OF))}
\leq
CT \norm  \bm {\breve v}_1 - \bm {\breve v}_2 \norm^2_{F^T_2}.
\end{align}
Then, using Young's inequality we bound the integral $\displaystyle \int_0^t \displaystyle \int_{\OF} \bm F_0 : \bm \nabla \bm {\tilde \zeta} \ d\bm {\tilde x} \ ds$ as
\begin{equation}\label{Estimate-integral-F0-1}
\begin{split}
\Bigg \lvert 
\mathlarger \int_0^t 
\displaystyle \int_{\OF}  
\bm F_0 : \bm \nabla \bm {\tilde \zeta} \ d\bm {\tilde x} \ ds \Bigg \lvert
\leq
\mathlarger \int_0^t 
\displaystyle \int_{\OF} 
\lvert \bm F_0 \lvert \lvert \nze \lvert \ d\bm {\tilde x} \ ds
\leq
C_\delta CT \norm  \bm {\breve v}_1 - \bm {\breve v}_2 \norm^2_{F^T_2}+ 
\delta \lvert \bm {\tilde \zeta} \norm^2_{L^2(H^1(\OF))}.
\end{split}
\end{equation}
In order to deal with the integral in $\bm L_0$ we use \eqref{Est-det-cof} and Young's inequality to get
\begin{equation}\label{int1-det-fluid}
\begin{split}
\Bigg \lvert
\mathlarger \int_0^t
\mathlarger \int_{\OF}
\partial_t \gamma_2
\big[ \tdet(\bm \nabla \tca_1) - \tdet(\bm \nabla \tca_2) \big]
\bm {\tilde \zeta} \ d\bm {\tilde x} \  ds
\Bigg \lvert
%
&
\leq
T
\norm \partial_t \bm {\tilde \gamma}_2 \norm_{L^\infty(L^3(\OF))}
\norm \tdet(\bm \nabla \tca_1) - \tdet(\bm \nabla \tca_2) \norm_{L^\infty(L^6(\OF))}
\norm \bm {\tilde \zeta} \norm_{L^\infty(L^2(\OF))}
\\
& \leq 
CT^\kappa M
\Big[
C_\delta \norm \bm {\breve v}_1 - \bm {\breve v}_2 \norm^2_{F^T_2}
+
\delta \norm \bm {\tilde \zeta} \norm^2_{L^\infty(L^2(\OF))}
 \Big]. 
\end{split}
\end{equation}
Similarly, for the integral in $\bm L_1$ we have
\begin{equation}\label{in1-det-structure1}
\begin{split}
\Bigg \lvert
\mathlarger \int_0^t
\mathlarger \int_{\OS}
\partial_t \gamma_2
\big[ \tdet(\bm \nabla \tp_1) 
- \tdet(\bm \nabla \tp_2) \big]
\bm {\tilde \zeta} \ d\bm {\tilde x} \  ds
\Bigg \lvert
\leq
CT^\kappa M 
\Big[
C_\delta \norm \bm {\breve \xi}_1 - \bm {\breve \xi}_2 \norm^2_{S^T_2}
+
\delta \norm \bm {\tilde \zeta} \norm^2_{W^{1,\infty}(L^2(\OS))}
\Big]
.
\end{split}
\end{equation}
Finally, proceeding in a similar manner as Subsection \ref{A Priori Estimates} with the use of \eqref{Estimate-fluid-f0}-\eqref{in1-det-structure1} and taking into consideration that $T$ is small with respect to $M$ we get

\begin{align}\label{1st-nl-estimate}
\norm \bm {\tilde \zeta} \norm^2_{F^T_1}
+
\norm \mathsmaller \int_0^\bigcdot \bm  {\tilde \zeta}(s) ds \norm^2_{S^T_1}
\leq
CT^\kappa M
\Big[
\norm \bm {\breve \xi}_1 - \bm {\breve \xi}_2 \norm^2_{S^T_2}
+
\norm \bm {\breve v}_1 - \bm {\breve v}_2 \norm^2_{F^T_2}
\Big ].
\end{align}

\subsection{Estimates on $\partial_t \tilde{\zeta}$}
The weak solution $\bm {\tilde \zeta}$ satisfies \eqref{Vf-zeta}. Differentiating \eqref{Vf-zeta} in times gives the following variational formulation
\begin{equation}\label{formulation gamma-zeta;general}
\begin{cases}
\rho_f 
\displaystyle \int_{\OF} 
\tdet(\bm \nabla \bm{\breve{\mathcal{A}}}_{1})
\partial^2_t \bm {\tilde \zeta} \cdot \bm {\tilde \eta} \ d\bm {\tilde x} 
+\rho_f
\mathlarger \int_{\OS}
\partial_t \tdet(\bm \nabla \tca_1)
\partial_t \bm {\tilde \zeta} \cdot \bm {\tilde \eta} \ d\bm {\tilde x}
+
\displaystyle \int_{\OF} \partial_t \bm {\breve \sigma}^0_{1}(\bm {\tilde \zeta}) :\bm \nabla \bm {\tilde \eta} \ d\bm {\tilde x}
\vspace{1mm} \\
+\rho_s \mathlarger \int_{\OS} \tdet(\bm \nabla \tp_1) \partial^2_t \bm {\tilde \zeta} \cdot \bm {\tilde \eta} \ d\bm {\tilde x}
+ \rho_s \mathlarger \int_{\OS} 
\partial_t \tdet(\bm \nabla \tp_1) \partial_t \bm {\tilde \zeta} \cdot \bm {\tilde \eta} \ d\bm {\tilde x}
\vspace{1mm} \\
-\displaystyle \sum_{i,\al,j,\be=1}^3 \displaystyle \int_{\OS} \partial_t \biajb (\bm \nabla \bm {\breve \xi}_1) \partial^2_{\al \be} (\int_0^t \tilde{\zeta}(s) ds)_j \ \tilde {\eta}_i \ d\bm {\tilde x} 
+\displaystyle \sum_{i,\al,j,\be=1}^3
\displaystyle \int_{\OS} 
\partial_\al \biajb (\bm \nabla \bm {\breve \xi}_1) \partial^2_{t \be} (\int_0^t \tilde{\zeta}(s) ds)_j \ \tilde {\eta}_i \ d\bm {\tilde x}
\vspace{1mm} \\
+\displaystyle \sum_{i,\al,j,\be=1}^3
\displaystyle \int_{\OS} \biajb (\bm \nabla \bm {\breve \xi}_1)
\partial^2_{t \be} (\int_0^t \tilde{\zeta}(s) ds)_j \ \partial_\al \tilde {\eta}_i \ d\bm {\tilde x}
+
\displaystyle \int_{\OF} \partial_t \bm F_0 : \bm \nabla \bm {\tilde \eta} \ d\bm {\tilde x} 

\vspace{1mm} \\
+

\mathlarger \int_{\OF}
\partial_t \bm  L_0  \cdot \bm {\tilde \eta} \ d\bm {\tilde x}

+
\mathlarger \int_{\OS}
\partial_t \bm L_1 \cdot \bm {\tilde \eta} \ d\bm {\tilde x}

-\displaystyle \sum_{i,\al,j,\be=1}^3 \displaystyle \int_{\OS}
\Big( \partial_t \biajb (\bm \nabla \bm {\breve \xi}_1)-\partial_t \biajb (\bm \nabla \bm {\breve \xi}_2) \Big) \partial^2_{\al \be} (\int_0^t \tilde{\gamma}_2(s) ds)_j \ \tilde {\eta}_i \ d\bm {\tilde x}

\vspace{1mm} \\
+\displaystyle \sum_{i,\al,j,\be=1}^3 \displaystyle \int_{\OS}
\Big(
 \partial_\al \biajb (\bm \nabla \bm {\breve \xi}_1)-\partial_\al \biajb (\bm \nabla \bm {\breve \xi}_2) \Big)
\partial^2_{t \be} (\int_0^t \tilde{\gamma}_2(s) ds)_j \ \tilde {\eta}_i \ d\bm {\tilde x}
\vspace{1mm} \\
+
\displaystyle \sum_{i,\al,j,\be=1}^3 \displaystyle \int_{\OS}
\Big(
\biajb (\bm \nabla \bm {\breve \xi}_1)-\biajb (\bm \nabla \bm {\breve \xi}_2)
\Big)
\partial^2_{t \be} (\int_0^t \tilde{\gamma}_2(s) ds)_j \ \partial_\al \tilde {\eta}_i \ d\bm {\tilde x}
=0 \qquad \forall \ \bm {\tilde \eta} \in \widetilde{\mathcal{W}},
\end{cases}
\end{equation}
where $\bm {\breve \sigma}^0_{1}(\bm {\tilde \zeta})$, $\bm F_0$, $\bm L_0$ and $\bm L_1$ are defined in \eqref{sigma-in-gamma} and \eqref{F_0}-\eqref{L0-L1}, respectively.
Take $\bm {\tilde \eta} = \partial_t \bm {\tilde \zeta}$ in \eqref{formulation gamma-zeta;general} to get
%
%
\begin{equation}\label{formulation gamma-zeta1}
\begin{cases}
\dfrac{\rho_f}{2} \dfrac{d}{dt} \displaystyle \int_{\OF}
\tdet(\bm \nabla \tca_1) \lvert \partial_t \bm {\tilde \zeta}(t) \lvert^2 \ d\bm {\tilde x}
+
\dfrac{\rho_f}{2}
\mathlarger \int_{\OF}
\partial_t \tdet(\bm \nabla \tca_1)
\lvert \partial_t \bm {\tilde \zeta} \lvert^2 \ d\bm {\tilde x}
\vspace{2mm} \\
+
\dfrac{\rho_s}{2} \dfrac{d}{dt} \displaystyle \int_{\OS}
\tdet(\bm \nabla \tp_1) \lvert \partial_t \bm {\tilde \zeta}(t) \lvert^2 \ d\bm {\tilde x}
+\dfrac{\rho_s}{2} \dfrac{d}{dt} \displaystyle \int_{\OS} \lvert \partial_t \bm {\tilde \zeta}(t) \lvert^2 \ d\bm {\tilde x} 
+
\displaystyle \int_{\OF} \partial_t \bm {\breve \sigma}_{1}(\bm {\tilde \zeta}) : \partial_t \nze \ d\bm {\tilde x}
\vspace{2mm} \\
+\dfrac{1}{2} \dfrac{d}{dt}
\displaystyle \sum_{i,\al,j,\be=1}^3
\displaystyle \int_{\OS} \biajb (\bm \nabla \bm {\breve \xi}_1)
\partial^2_{t \be} (\mathsmaller \int_0^t \tilde{\zeta}(s) ds)_j
\ \partial^2_{t \al} (\mathsmaller \int_0^t \tilde{\zeta}(s) ds)_i \ d\bm {\tilde x} 
\vspace{2mm} \\
-\dfrac{1}{2}
\displaystyle \sum_{i,\al,j,\be=1}^3
\displaystyle \int_{\OS}
\partial_t \biajb (\bm \nabla \bm {\breve \xi}_1) \partial^2_{t \be} (\mathsmaller \int_0^t \tilde{\zeta}(s) ds)_j
\
\partial^2_{t \al} (\mathsmaller \int_0^t \tilde{\zeta}(s) ds)_i \ d\bm {\tilde x}
\vspace{2mm} \\
+\displaystyle \sum_{i,\al,j,\be=1}^3
\displaystyle \int_{\OS}
\partial_\al \biajb (\bm \nabla \bm {\breve \xi}_1) \partial^2_{t \be} (\mathsmaller \int_0^t \tilde{\zeta}(s) ds)_j
\
\partial^2_t (\mathsmaller \int_0^t \tilde{\zeta}(s) ds)_i \ d\bm {\tilde x}
\vspace{2mm} \\
-\displaystyle \sum_{i,\al,j,\be=1}^3 \displaystyle \int_{\OS}
\partial_t \biajb (\bm \nabla \bm {\breve \xi}_1) \partial^2_{\al \be} (\mathsmaller \int_0^t \tilde{\zeta}(s) ds)_j 
\
\partial^2_t (\mathsmaller \int_0^t \tilde{\zeta}(s) ds)_i \ d\bm {\tilde x} \vspace{2mm} \\
=
-\displaystyle \int_{\OF} \partial_t \bm F_0 : \partial_t \nze \ d\bm {\tilde x}

+\displaystyle \int_{\OS}
\partial_t \bm H_0 \cdot \partial^2_t (\mathsmaller \int_0^t \bm {\tilde \zeta}(s) ds) \ d\bm {\tilde x}

-

\mathlarger \int_{\OF}
\partial_t \bm L_0 \cdot \partial_t \bm {\tilde \zeta} \ d\bm {\tilde x}

-
\mathlarger \int_{\OS}
\partial_t \bm L_1 \cdot \partial^2_t (\mathsmaller \int_0^t \bm {\tilde \zeta}(s) ds) \ d\bm {\tilde x}

\vspace*{2mm} \\
-\displaystyle \sum_{i,\al,j,\be=1}^3 \displaystyle \int_{\Gamma_c(0)}
\Big(
\biajb (\bm \nabla \bm {\breve \xi}_1)-\biajb (\bm \nabla \bm {\breve \xi}_2)
\Big)
\partial^2_{t \be} (\mathsmaller \int_0^t \tilde{\gamma}_2(s) ds)_j \ \tilde{n}_{\al}
\
 \partial_t^2 (\mathsmaller \int_0^t \tilde{\zeta}(s) ds)_i \ d\tilde{\Gamma} ,
\end{cases}
\end{equation}
where
\begin{align}\label{H_0}
 H_{0,i}=
\displaystyle \sum_{\al,j,\be=1}^3
\Big(
\biajb (\bm \nabla \bm {\breve \xi}_1)-\biajb (\bm \nabla \bm {\breve \xi}_2)
\Big)
\partial^2_{\al \be} (\mathsmaller \int_0^t  \tilde{\gamma}_2(s) ds)_j \qquad \textup{for} \ i=1,2,3.
\end{align}%
%
\subsubsection*{\large Step 1}
Now we proceed to derive some estimates on 
\[\bm {\tilde \zeta}|_{\OF} \in H^1(H^1)\cap W^{1,\infty}(L^2),
\bm {\tilde \zeta}|_{\OS} \in W^{1,\infty}(L^2) \cap L^{\infty}(H^1)
\quad \textup{and} \quad  
\int_0^t \bm {\tilde \zeta}(s)|_{\OS} ds \in L^\infty(H^1).\]
%
%
First, we have
\begin{equation}\label{estimate-fluid-det}
\begin{split}
\mathlarger \int_{\OF}
\tdet(\bm \nabla \tca_1) \lvert \partial_t \bm {\tilde \zeta}(t) \lvert^2 \ d\bm {\tilde x}
+
\mathlarger \int_0^t 
\mathlarger \int_{\OF}
\partial_t \tdet(\bm \nabla \tca_1) \lvert \partial_t \bm {\tilde \zeta} \lvert^2 \ d\bm {\tilde x} \ ds
%
%
\geq
(1-CT^\kappa M) \norm \bm {\tilde \zeta} \norm^2_{W^{1,\infty}(L^2(\OF))}
-CT^\kappa M \norm \bm {\tilde \zeta} \norm^2_{F^T_2}.
\end{split}
\end{equation}
Whereas, for the fluid stress term, proceeding as in \eqref{A_1} and \eqref{A2+A3} we get
\begin{align}\label{Estimate-fluid1}
\displaystyle \int_0^t \displaystyle \int_{\OF} \partial_s \bm {\breve \sigma}_{1}(\bm {\tilde \zeta}) : \partial_s \nze \ d\bm {\tilde x} \ ds
\geq
\mu C_k \norm \partial_t \bm {\tilde \zeta} \norm^2_{L^2(H^1(\OF))}
-
\mu CT^{\kappa}M \norm \partial_t \bm {\tilde \zeta} \norm^2_{L^2(H^1(\OF))}.
\end{align}
For $\displaystyle \int_{\OF} \partial_t \bm F_0 : \partial_t \nze \ d\bm {\tilde x} $, we argue as in \eqref{Estimate-integral-F0-1} to obtain
\begin{align}\label{Estimate-fluid2-Space2-F_0}
\Bigg \lvert \displaystyle \int_{\OF} \partial_t \bm F_0 : \partial_t \nze \ d\bm {\tilde x} \Bigg \lvert
&\leq
C_\delta \norm \bm {\breve v}_1 - \bm {\breve v}_2 \norm^2_{F^T_2} 
+
\displaystyle \int_{\OF} \delta \lvert \partial_t \nze \lvert^2 \ d\bm {\tilde x}.
\end{align}
Similarly, we have
\begin{align}\label{estimate-L0}
\Bigg \lvert
\mathlarger \int_{\OF} \partial_t \bm L_0 \cdot \partial_t \bm {\tilde \zeta} \ d\bm {\tilde x} \Big \lvert
\leq
C_\delta \norm \bm {\breve v}_1 - \bm {\breve v}_2 \norm^2_{F^T_2}
+
\delta \norm \partial_t \bm {\tilde \zeta}(t) \norm^2_{L^2(\OF)}.
\end{align}
Combining \eqref{Estimate-fluid1}-\eqref{estimate-L0} and integrating over $(0,t)$ we get 
\begin{align}\label{nonlinear estimate - velocity}
\rho_f \norm \bm {\tilde \zeta} \norm^2_{W^{1,\infty}(L^2)}
+
\mu \norm \bm {\tilde \zeta} \norm^2_{H^1(H^1)}
- CT^{\kappa}M \norm \partial_t \bm {\tilde \zeta} \norm^2_{L^2(H^1)}
&\leq
CT \norm  \bm {\breve v}_1 - \bm {\breve v}_2 \norm^2_{F^T_2} +
 \delta \norm \partial_t \nze \norm^2_{L^2(L^2(\OF))}.
\end{align}
As for the integrals on the domain $\OS$, first we have
\begin{align}\label{estimate-structure-det}
\mathlarger \int_{\OS}
\tdet(\bm \nabla \tp_1) \lvert \partial_t \bm {\tilde \zeta}(t) \lvert^2 \ d\bm {\tilde x}
+
\mathlarger \int_0^t 
\mathlarger \int_{\OS}
\partial_t \tdet(\bm \nabla \tp_1) \lvert \partial_t \bm {\tilde \zeta} \lvert^2 \ d\bm {\tilde x} \ ds
%
%
\geq
(1-CT^\kappa M) \norm \bm {\tilde \zeta} \norm^2_{W^{1,\infty}(L^2(\OS))}
-CT^\kappa M \norm \bm {\tilde \zeta} \norm^2_{L^\infty(H^1(\OS))}.
\end{align}
Further, using \eqref{coercivity} then taking supremum over $(0,T)$ yield
\begin{equation}\label{est97}
\begin{split}
\dfrac{1}{2}
\displaystyle \sum_{i,\al,j,\be=1}^3 \displaystyle \int_{\OS}
&\Big[
\biajb (\bm \nabla \bm {\breve \xi}_1)
\partial_\be \tilde{\zeta}_j \ \partial_\al \tilde{\zeta}_i
\Big](t) \ d\bm {\tilde x}
\\
&\geq \mu_s \norm \bm {\tilde \zeta} \norm^2_{L^\infty(H^1(\OS))}
+ \dfrac{\lambda_s+\mathsf{C}}{2} \norm \nabla \cdot \bm {\tilde \zeta} \norm^2_{L^\infty(L^2(\OS))}
- CT^\kappa M \norm \bm {\tilde \zeta} \norm^2_{S^T_2}.
\end{split}
\end{equation}
On the other hand, using \eqref{lq-boundedness} we have 
\begin{equation}\label{est97-1}
\begin{split}
\Bigg \lvert
-\dfrac{1}{2}
\displaystyle \sum_{i,\al,j,\be=1}^3
\displaystyle \int_{\OS}
&
\partial_t \biajb (\bm \nabla \bm {\breve \xi}_1) \partial
_\be \tilde{\zeta}_j
\
\partial_\al \tilde{\zeta}_i \ d\bm {\tilde x}
+\displaystyle \sum_{i,\al,j,\be=1}^3 \displaystyle \int_{\OS}  \partial_\al \biajb (\bm \nabla \bm {\breve \xi}_1) \partial_\be \tilde{\zeta}_j \ \partial_t \tilde{\zeta}_i \ d\bm {\tilde x} 
\\
&
-\displaystyle \sum_{i,\al,j,\be=1}^3 \displaystyle \int_{\OS}
\partial_t \biajb (\bm \nabla \bm {\breve \xi}_1) \partial_{\al \be} (\mathsmaller \int_0^t \tilde{\zeta}(s) ds)_j \ \partial_t \tilde{\zeta}_i \ d\bm {\tilde x} \Bigg \lvert
\leq C T^\kappa M \norm \bm  {\tilde \zeta} \norm^2_{S^T_2}.
\end{split}
\end{equation}
%
%
%
%
In order to estimate
$ \displaystyle \int_{\OS}
\partial_t \bm H_0 \cdot \partial^2_t (\mathsmaller \int_0^t \bm {\tilde \zeta}(s) ds) \ d\bm {\tilde x}  $
we use the following two inequalities
\begin{align}\label{eq1-H0}
\norm \biajb(\bm \nabla \bm {\breve \xi}_1)- \biajb(\bm \nabla \bm {\breve \xi}_2) \norm_{L^\infty(H^1)} \leq C \norm \bm {\breve \xi}_1-  \bm {\breve \xi}_2 \norm_{L^\infty(H^2(\OS))}
\end{align}
and
\begin{align}\label{eq2-H0}
\norm \partial_t \biajb(\bm \nabla \bm {\breve \xi}_1)- \partial_t \biajb(\bm \nabla \bm {\breve \xi}_2) \norm_{L^\infty(L^2)} \leq C \norm \bm {\breve \xi}_1-  \bm {\breve \xi}_2 \norm_{W^{1,\infty}(H^1(\OS))}.
\end{align}
These inequalities together with Young's inequality give
\begin{align}\label{est98}
\displaystyle \int_{\OS} \partial_t \bm  H_0 \cdot \partial^2_t (\mathsmaller\int_0^t \bm {\tilde \zeta}(s) ds) \ d\bm {\tilde x}
\leq
C_\delta C \big( \norm \bm {\breve \xi}_1-  \bm {\breve \xi}_2 \norm^2_{W^{1,\infty}(H^1)}
+ \norm \bm {\breve \xi}_1-  \bm {\breve \xi}_2 \norm^2_{L^\infty(H^2)} \big)
+\delta \norm \partial_t \bm {\tilde \zeta} \norm^2_{L^\infty(L^2(\OS))}.
\end{align}
%
%
Further,
\begin{align}\label{estimate-L1}
\Bigg \lvert
\mathlarger \int_{\OS} \partial_t \bm L_1 \cdot \partial^2_t (\mathsmaller\int_0^t \bm {\tilde \zeta}(s) ds) \ d\bm {\tilde x} \Bigg \lvert
\leq
C_\delta \norm \bm {\breve \xi}_1 - \bm {\breve \xi}_2 \norm^2_{S^T_2}
+
\delta
\norm \partial_t \bm {\tilde \zeta}(t) \norm^2_{L^2(\OS)}.
\end{align}
\\
Finally, thanks to the trace inequality and \eqref{eq-xi12}, for $i,\al,j,\be \in \{1,2,3\}$ we have
\begin{equation}\label{est98-2}
\begin{split}
&\Bigg \lvert \Bigg \lvert \displaystyle \int_{\Gamma_c(0)}
\Big(
\biajb (\bm \nabla \bm {\breve \xi}_1)-\biajb (\bm \nabla \bm {\breve \xi}_2)
\Big)
\partial^2_{t \be} (\mathsmaller \int_0^t \tilde{\gamma}_2(s) ds)_j \ \tilde{n}_{\al} \ \partial_t^2 (\mathsmaller \int_0^t \tilde{\zeta}(s) ds)_i \ d\tilde{\Gamma}
\Bigg \lvert \Bigg \lvert_{L^1(\Gamma_c(0))}
\\
&
\leq 
\norm
\biajb(\bm \nabla \bm {\breve \xi}_1)(t)-\biajb(\bm \nabla \bm {\breve \xi}_2)(t) \norm_{H^1(\Gamma_c(0))}
\norm \bm \nabla \bm {\tilde \gamma}_2(t) \norm_{H^2(\OS)}
\norm \partial_t \bm {\tilde \zeta} (t) \norm_{H^1(\OF)}.
\end{split}
\end{equation}
Hence, after combining \eqref{estimate-structure-det}-\eqref{est97-1} and \eqref{est98}-\eqref{est98-2} then  integrating over $(0,t)$ we obtain
\begin{align}\label{nonlinear-defo-estimate}
&\dfrac{\rho_s}{2} \norm \bm {\tilde \zeta} \norm_{L^\infty(L^2(\OS))} +
\mu_s \norm \bm {\tilde \zeta} \norm_{L^2(H^1(\OS))}
\\ \nonumber
& \leq
CT \big[ M^4 \norm \bm {\tilde \zeta} \norm_{S^T_2} + \norm \bm {\breve \xi}_1-  \bm {\breve \xi}_2 \norm^2_{W^{1,\infty}(H^1)}
+ \norm \bm {\breve \xi}_1-  \bm {\breve \xi}_2 \norm^2_{L^\infty(H^2)}  \big].
\end{align}

\subsubsection*{\large Step 2}
Our next step is to estimate $\bm {\tilde \zeta}|_{\OF} \in L^{\infty}(H^2(\OF))$
and
$\int_0^t \bm {\tilde \zeta}(s)|_{\OS} ds \in  L^\infty(H^2(\OS))$.
The fluid velocity $\bm {\tilde \zeta}|_{\OF}$ satisfies the following elliptic equation
\begin{equation}
-\bm \nabla \cdot 
\big(
(\nze|_{\OF}) + (\nze|_{\OF})^t
\big) = \nabla \cdot \bm F_0+ \bm \nabla \cdot \bm F_1 + \bm \nabla \cdot \bm L_0 -\tdet(\bm \nabla \tca)\partial_t \bm {\tilde \zeta}|_{\OF} \quad \textrm{in} \ \OF
\end{equation}
where $\bm F_0$ is defined in \eqref{F_0} and
\begin{align}
\bm F_1 =
\nze \Big( \Id- (\nc_1)^{-1} \cof(\nc_1) \Big )
+ \Big( (\nze)^t- (\nc_1)^{-1} (\nze)^t \cof(\nc_1) \Big).
\end{align}
We have 
\begin{align}
\norm \bm \nabla \cdot \bm F_0 \norm_{L^\infty (L^2(\OF))}
\leq
\norm \bm F_0 \norm_{L^\infty (H^1(\OF))}
\leq
CT \norm \bm {\breve v}_1 - \bm {\breve v}_2 \norm_{F^T_2}
\end{align}
%
and
\begin{align*}
\norm \bm \nabla \cdot \bm L_0 \norm_{L^\infty(L^2(\OF))}
\leq
\norm \bm  L_0 \norm_{L^\infty(H^1(\OF))} 
\leq
CT \norm \bm {\breve v}_1 - \bm {\breve v}_2 \norm_{F^T_2}.
\end{align*}
For $\bm F_1$ we have
%
%
\begin{equation}
\begin{split}
\norm \bm \nabla \cdot \bm F_1 \norm_{L^\infty (L^2(\OF))}
\leq
&\underbrace{
\norm \nabla \cdot
\big[
\nze \big( \Id- (\nc_1)^{-1} \cof(\nc_1) \big )
\big]
 \norm_{L^\infty(L^2(\OF))}}_{B_1}
+ \underbrace{
 \norm \nabla \cdot 
 \big[ (\nze)^t- (\nc_1)^{-1} (\nze)^t \cof(\nc_1)
 \big]
  \norm}_{B_2}.
\end{split}
\end{equation}
%
%
For the term $B_1$ we use the embedding of $H^3 \subset L^\infty$ and Lemma \ref{Estimates on Flow} to get
%
%
\begin{equation}
\begin{split}
B_1
\leq &
\Bigg \lvert
\Bigg \lvert
\nabla \cdot 
\bigg[ \nze
\Big(
 \Id- (\nc_1)^{-1} + (\nc)^{-1}   \big(\Id- \cof(\nc_1) \big) 
 \Big )
\bigg]
\Bigg \lvert
\Bigg \lvert
_{L^\infty(L^2)}
\leq 
CT^\kappa M 
\norm \bm {\tilde \zeta} \norm_{L^\infty(H^2)}.
\end{split}
\end{equation}
%
%
On the other hand,
%
%
\begin{align*}
\norm B_2 \norm_{L^\infty(L^2)}
\leq&
CT^\kappa M \norm \bm {\tilde \zeta} \norm_{L^\infty(H^2)}.
\end{align*}
Consequently, we obtain
%
%
\begin{align}
\norm \bm \nabla \cdot \bm F_1 \norm_{L^\infty(L^2)}
\leq
CT^\kappa M \norm \bm {\tilde \zeta} \norm_{L^\infty(H^2)}.
\end{align}
Therefore, $\bm {\tilde \zeta}|_{\OF} \in L^{\infty}(H^2(\OF))$ and
\begin{align}\label{spatial-estimat1}
\mu \norm \bm {\tilde \zeta} \norm_{ L^{\infty}(H^2(\OF)) }
\leq
C\norm \partial_t \bm {\tilde \zeta} \norm_{L^{\infty}(L^2(\OF))}
+
CT \norm \bm {\breve v}_1-\bm {\breve v}_2 \norm_{F^T_2}.
\end{align}
Besides, the displacement $
\int_0^t \bm {\tilde \zeta}(s)|_{\OS} ds$, satisfies the following equation
\begin{align}
-\mu_s \bm \nabla \cdot (\bm \nabla \mathsmaller \int_0^t \zeta(s)|_{\OS} ds + (\bm \nabla \mathsmaller \int_0^t \bm {\tilde \zeta}(s)|_{\OS} ds )^t)
=
\bm H_0+\bm H_1+\bm H_2- \bm \nabla \cdot \bm L_1,
\end{align}
%
%
where
%
%
$\bm H_0$ is defined by \eqref{H_0}.
As for $\bm H_1$, it is given by 
%
%
\begin{align}
H_{1,i}= -\displaystyle \sum_{\al,j,\be=1}^3
\Big [ \biajb^l (\bm \nabla \bm {\breve \xi}_1) + \biajb^q (\bm \nabla \bm {\breve \xi}_1) \Big] \partial^2_{\al \be} (\mathsmaller \int_0^t \tilde{\zeta}(s)|_{\OS} ds)_j
\qquad \textrm{for} \quad i=1,2,3,
\end{align}
%
%
and the expression of $\bm H_2$ is
%
%
\begin{align*}
\bm H_2= - \tdet(\bm \nabla \tp_1) \partial_t \bm {\tilde \zeta}.
\end{align*}
%
%
Using \eqref{eq1-H0} and \eqref{eq2-H0} we have
%
%
\begin{align*}
\norm \bm H_0 \norm_{L^{\infty}(L^2(\OS))}
&\leq
\displaystyle \sum_{i,\al,j,\be=1}^3
\norm \biajb (\bm \nabla \bm {\breve \xi}_1) - \biajb (\bm \nabla \bm {\breve \xi}_2) \norm_{L^{\infty}(L^2(\OS))}
\ \norm \mathsmaller \int_0^{\bigcdot} \bm {\tilde \gamma}_2(s)|_{\OS} ds \norm_{L^\infty(H^2(\OS))}
\\
&
\leq
CT^\kappa M \norm \bm {\breve \xi}_1 - \bm {\breve \xi}_2 \norm_{L^{\infty}(H^1(\OS))}.
\end{align*}
%
%
For $\bm H_1$ we use \eqref{lq-boundedness} to obtain
%
%
\begin{equation}
\begin{split}
\norm \bm H_1 \norm_{L^{\infty}(L^2(\OS))}
&\leq
\displaystyle \sum_{i,\al,j,\be=1}^3
\norm
\biajb^l (\bm \nabla \bm {\breve \xi}_1) + \biajb^q (\bm \nabla \bm {\breve \xi}_1)  \norm_{L^{\infty}(L^2)} \ \norm \bm {\tilde \zeta} \norm_{L^{\infty}(H^2(\OS))}
\\
&\leq
CT(M+M^2) 
\norm \mathsmaller \int_0^{\bigcdot} \bm {\tilde \zeta}(s)|_{\OS} ds \norm_{L^\infty(H^2(\OS))}.
\end{split}
\end{equation}
%
%
%
In addition, we have
\begin{align*}
\norm \bm H_2 \norm_{L^\infty(L^2(\OS))}
\leq
CT\norm \tdet(\bm \nabla \tp_1) \partial_t \bm {\tilde \zeta} \norm_{L^\infty(L^2(\OS))}
\leq
CTM \norm \mathsmaller \int_0^{\bigcdot} \bm {\tilde \zeta}(s) ds \norm_{W^{2,\infty}(L^2(\OS))}.
\end{align*}
Finally, for $\bm L_1$ it holds
\begin{align*}
\norm \bm \nabla \cdot \bm L_1 \norm_{L^\infty(L^2(\OS))}
\leq
\norm \bm L_1 \norm_{L^\infty(H^1(\OS))}
\leq
CT \norm \bm {\breve \xi}_1 - \bm {\breve \xi}_2 \norm_{S^T_2}.
\end{align*}
Whence,
$ \int_0^t \bm {\tilde \zeta}(s)|_{\OS} ds \in L^\infty( H^2( \OS))$ and a priori estimate is given as
%
%
\begin{align}\label{spatial estimate2}
\mu_s
\norm 
\mathsmaller \int_0^\bigcdot \bm {\tilde \zeta}(s)|_{\OS} ds
 \norm_{L^\infty(H^2(\OS))}
&\leq
CT \norm \bm {\breve \xi}_1-\bm {\breve \xi}_2 \norm_{S^T_2} 
\end{align}
Therefore, combining estimates \eqref{nonlinear estimate - velocity}, \eqref{nonlinear-defo-estimate}, \eqref{spatial-estimat1} and \eqref{spatial estimate2} we arrive to
\begin{align}
\norm \bm {\tilde \zeta}|_{\OF} \norm_{F^T_2} 
+ \norm \mathsmaller \int_0^{\bigcdot} \bm {\tilde \zeta}(s)|_{\OS} ds \norm_{S^T_2}
\leq CT^\kappa M ( \norm \bm {\breve v}_1-  \bm {\breve v}_2 \norm_{F^T_2} + \norm \bm {\breve \xi}_1- \bm {\breve \xi}_2 \norm_{S^T_2}).
\end{align}
Taking $T$ small with respect to $M$ gives that $\Psi$ is a contraction on $A^T_M$. This yields the existence of a unique solution $(\bm \vv,\bm \xxi)$ in $A^T_M$ of the non-linear coupled system \eqref{eqcoup1}-\eqref{eqcoup102}.

%
%
%
%
%
%
%

%
%
%
%
%
%
\section{Existence and Uniqueness of the Fluid Pressure}\label{existence-of-pressure}

\subsection{Existence and Uniqueness of an $L^2$-Pressure}
After we have proved the existence and uniqueness of the fluid velocity $\bm v$ and the structure displacement $\bm \xi$, we need to prove the existence of the fluid pressure $p_f$ so that the proof of the existence of the weak solution for the coupled system \eqref{eqcoup1}-\eqref{eqcoup102} is complete. 
The proof of existence of the $L^2$ function $p_f$ is based on Lemma \cite[p.58, Lemma 4.1]{Girault-Raviart} \cite{Brezzi} that reduces the proof to showing that the following inf-sup condition holds for the functional spaces
$\{\mathcal{W} ,L^2(\Omega_f(t))\}$:
\begin{align}\label{Inf-Sup condition}
\inf_{q \in L^2(\Omega_f(t))} \sup_{\bm z \in \mathcal{W}} \dfrac{b(\bm z,q)}{\norm \bm z \norm_{H^1(\Omega(t))} \norm q \norm_{L^2(\Omega_f(t))}} \geq C_1 >0,
\end{align}
with
\begin{align}
b(\bm z,q) = - \int_{\Omega_f(t)} q \ \textrm{div}\bm  z \ d\bm x \quad
\quad \textrm{and} \quad \bm z \in \mathcal{W}, \ q \in L^2(\Omega_f(t)).
\end{align}
\begin{theorem}
The inf-sup condition \eqref{Inf-Sup condition} holds for the functional spaces
$\{\mathcal{W} ,L^2(\Omega_f(t))\}$.
\end{theorem}
\begin{proof}
We will proceed in a similar manner as \cite[Lemma 3.1]{Berm}. To show that the condition holds, it suffices to show that
\begin{align}\label{Assumption}
\forall \ q \in L^2(\Omega_f(t)), 
&
\exists \ \bm  z \in \mathcal{W} \ \textrm{such that}, \ \textrm{div} \bm z|_{\Omega_f(t)} = q \ \textrm{in} \ \Omega_f(t) 
\\ \nonumber 
&
\textrm{and} \ \norm \bm z \norm_{H^1(\Omega(t))} \leq C_1 \norm q \norm_{L^2(\Omega_f(t))}.
\end{align}
Let $\overline{q} \in L^2(\Omega(t))$ be the extension of $q$ obtained by defining
\begin{align}
\overline{q} = -\dfrac{1}{|\Omega_s(t)|}\int_{\Omega_f(t)} q \ d\bm x, \quad \textrm{in} \ \Omega_s(t).
\end{align}
Note that $\int_{\Omega(t)}\overline{q}\ d\bm x= \int_{\Omega_f(t)} \overline{q} \ d\bm x+\int_{\Omega_s(t)} \overline{q} \ d\bm x=0$, this gives $\overline{q} \in L^2_0(\Omega(t))$. Hence, by the virtue of \cite[Theorem IV.3.1]{Boyer}, there exists a unique $\bm z \in H^1_0(\Omega(t))$ such that
\begin{align}
\textup{div} \bm z = \overline{q} \quad \textup{on} \quad \Omega(t) \quad \textrm{and} \quad \norm \bm z \norm_{H^1(\Omega(t))} \leq C \norm \overline{q} \norm_{L^2_0(\Omega(t))} \leq C_1 \norm q \norm_{L^2(\Omega_f(t))}.
\end{align}
Since $H^1_0(\Omega(t)) \subset \mathcal{W}$, then $ \bm z \in \mathcal{W}$. Moreover, by restricting $\textup{div} \bm z= \overline{q}$ to $\Omega_f(t)$ we get that 
$\textup{div} \bm z|_{\Omega_f(t)}=q$.
Therefore \eqref{Assumption} is proved, consequently the inf-sup Condition \eqref{Inf-Sup condition} is verified. 
\hspace*{16.6cm}\end{proof}
\\
By the end of this proof, we get the existence of a pressure $p_f \in L^\infty\big(L^2(\Omega_f(t))\big)$ which is unique due to \cite[Theorem IV.2.4]{Boyer}.
\subsection{Regularity of the Fluid Pressure}

The fluid pressure $p_f$ is related to the fluid velocity $\bm v$ by the Navier-Stokes equations. Indeed, at $t=0$ we have
\begin{equation*}
\rho_f \tdet(\bm{\nabla \mathcal{A}})\partial_t \bm \vv -\bm \nabla \cdot \bm {\tilde \sigma}^0_f(\bm \vv,\tilde{p}_f) = 0 \qquad \textrm{in $\Omega_f(0) \times (0,T) $},
\end{equation*}
As a result, the regularity of $\tilde{p}_f$ is linked to the regularity of $\vv$ which is proved straight forward using Ne\u cas inequality \cite[Theorem IV.1.1]{Boyer}. Therefore, as $\vv \in F^T_4$ then $\tilde{p}_f \in \mathcal{P}^T_3$.
Again using \cite[Lemma 2.56]{Richter}, we get the existence and uniqueness of a fluid pressure $p_f$ in the set $\mathcal{Q}^T_3$ which is equivalent to $\mathcal{P}^T_3$ where the functions of $\mathcal{Q}^T_3$ are defined over $\Omega_f(t)$.
 \\

To this end, we have proved the existence and uniqueness locally in time of a solution $(\bm v, \bm \xi, p_f)$ of the non-linear coupling problem of an incompressible fluid with a quasi-incompressible structure. 

\section*{Acknowledgment}
The authors would like to thank the Rectorat of the Lebanese University for funding this
research project.

\bibliography{bibliography}

\end{document}